\documentclass[nosumlimits,twoside]{amsart}
\usepackage{amsfonts, amsmath, amssymb}
\usepackage[bookmarksnumbered,plainpages,driverfallback=dvipdfm]{hyperref}
\usepackage{graphicx}
\usepackage{float}
\usepackage{srcltx}
\usepackage[all]{xy}
\usepackage{version}
\usepackage[final]{showlabels}
\usepackage[T1]{fontenc}
\usepackage{xcolor}
\usepackage{enumerate}
\usepackage[normalem]{ulem}
\usepackage{bm}
\usepackage{latexsym}
\usepackage[2emode]{psfrag}
\usepackage{yhmath}
\usepackage{array}
\usepackage{dsfont}

\newtheorem{theorem}{\sc Theorem}[section]
\newtheorem{proposition}[theorem]{\sc Proposition}

\newtheorem{lemma}[theorem]{\sc Lemma}
\newtheorem{corollary}[theorem]{\sc Corollary}
\theoremstyle{definition}
\newtheorem{definition}[theorem]{\sc Definition}

\newtheorem{example}[theorem]{\sc Example}
\newtheorem{examples}[theorem]{\sc Examples}

\theoremstyle{remark}
\newtheorem{remark}[theorem]{\sc Remark}

\newtheorem{claim}[theorem]{}

\newenvironment{invisible}{{\noindent\sc \colorbox{yellow}{Invisible:}\;}\color{gray}}{\medskip}

\allowdisplaybreaks
\excludeversion{invisible}
\excludeversion{proof?}
\newcommand{\id}{{\sf id}}
\newcommand{\Id}{{\sf Id}}
\newcommand{\Cc}{\mathcal{C}}
\newcommand{\Dd}{\mathcal{D}}
\newcommand{\Mm}{\mathcal{M}}
\newcommand{\Nn}{\mathcal{N}}
\newcommand{\Aa}{\mathcal{A}}
\newcommand{\Bb}{\mathcal{B}}

\newcommand{\Pp}{\mathcal{P}}
\newcommand{\Rr}{\mathcal{R}}
\newcommand{\Ss}{\mathcal{S}}
\newcommand{\ot}{\otimes}
\def\Alg{{\sf Alg}}

\def\Coalg{{\sf Coalg}}

\def\Bialg{{\sf Bialg}}

\def\Hpfalg{{\sf Hopfalg}}

\def\Vec{{\sf Vec}}

\def\Rel{{\sf Rel}}

\def\hom{{\rm Hom}}
\def\ZZ{{\mathbb Z}}
\def\NN{{\mathbb N}}
\def\uoR{\underline{\overline{R}}}
\def\uoL{\underline{\overline{L}}}

\def\oR{\overline{R}}

\newcommand{\Fam}{{\sf Fam}}
\newcommand{\Maf}{{\sf Maf}}
\newcommand{\Faf}{{\sf FamRel}}

\newcommand{\ev}{{\rm ev}}
\newcommand{\coev}{{\rm coev}}

\newcommand{\unit}{\mathds{I}}
\newcommand{\initial}{\mathbf{0}}
\newcommand{\terminal}{\mathbf{1}}
\newcommand{\op}{\mathrm{op}}
\def\relto{\relbar\joinrel\mapstochar\joinrel\rightarrow}

\begin{document}
\title[Pre-rigid Monoidal Categories]{Pre-rigid Monoidal Categories}

\thanks{This article was written while the first and the third author were members of the ``National Group for Algebraic and Geometric Structures, and their Applications'' (GNSAGA-INdAM). They were both partially supported by MIUR
within the National Research Project PRIN 2017.
\\The authors would also like to thank Joost Vercruysse and Miodrag C. Iovanov for helpful discussions. 
}

\begin{abstract}
Liftable pairs of adjoint functors between braided monoidal categories in the sense of \cite{GV-OnTheDuality} provide auto-adjunctions between the associated categories of bialgebras. Motivated by finding interesting examples of such pairs, we study general pre-rigid monoidal categories. Roughly speaking, these are monoidal categories in which for every object $X$, an object $X^{\ast}$ and a nicely behaving evaluation map from $X^{\ast}\ot X$ to the unit object exist. A prototypical example is the category of vector spaces over a field, where $X^{\ast}$ is not a categorical dual if $X$ is not finite-dimensional.
	We explore the connection with related notions such as right closedness, and present meaningful examples. We also study the categorical frameworks for Turaev's Hopf group-(co)algebras in the light of pre-rigidity and closedness, filling some gaps in literature along the way. Finally, we show that braided pre-rigid monoidal categories indeed provide an appropriate setting for liftability in the sense of loc. cit. and we present an application, varying on the theme of vector spaces, showing how -in favorable cases- the notion of pre-rigidity allows to construct liftable pairs of adjoint functors when right closedness of the category is not available.

\end{abstract}

\author{Alessandro Ardizzoni}
\address{%
\parbox[b]{\linewidth}{University of Turin, Department of Mathematics ``G. Peano'', via
Carlo Alberto 10, I-10123 Torino, Italy}}
\email{alessandro.ardizzoni@unito.it}
\urladdr{\url{https://sites.google.com/site/aleardizzonihome}}
\author{Isar Goyvaerts}
\address{%
\parbox[b]{\linewidth}{Department of Mathematics, Faculty of Engineering, Vrije Universiteit Brussel, Pleinlaan 2, B-1050 Brussel, Belgium}}
\email{Isar.Goyvaerts@vub.be}
\author{Claudia Menini}
\address{%
\parbox[b]{\linewidth}{University of Ferrara, Department of Mathematics and Computer Science, Via Machiavelli
30, Ferrara, I-44121, Italy}}
\email{men@unife.it}
\urladdr{\url{https://sites.google.com/a/unife.it/claudia-menini/}}
\maketitle

\tableofcontents

\section{Introduction}
Basic Linear Algebra teaches us that any vector space $V$ (over a field $k$) has a dual space, $V^{\ast}$, which is unique up to isomorphism. When $V$ is moreover finite-dimensional, we have two $k$-linear maps, $\ev_{V}:V^{\ast }{\otimes}_k V\rightarrow k$ (the evaluation at $V$) and $\mathrm{coev}_{V}:k\rightarrow V {\otimes}_k V^{\ast }$ (the coevaluation at $V$), satisfying two compatibility relations in the form of triangular identities (cf. \cite[Definition 2.10.1]{EGNO} e.g.). With respect to the tensor product (over $k$) and the ground field as unit object, this extra data turns the category of finite-dimensional $k$-vector spaces into a rigid monoidal category. When $V$ is infinite dimensional, $V^{\ast}$ and $\ev_{V}$ can still be defined, but there is no coevaluation map anymore. In other words, the monoidal category of all $k$-vector spaces is no longer rigid; in turn, it is the prototype of a so-called {\it pre-rigid} monoidal category.
\\A monoidal category $(\Cc,\otimes,\unit)$ is said to be pre-rigid if for every object $X$ there exists an object $X^{\ast }$ and a
morphism $\ev_{X}:X^{\ast }\otimes X\rightarrow \unit$ such
that the map
$$\hom _{\mathcal{C}}\left( T,X^{\ast }\right) \rightarrow \hom %
_{\mathcal{C}}\left( T\otimes X,\unit\right) :u\mapsto \ev%
_{X}\circ \left( u\otimes X\right)
$$
is bijective for every object $T$ in $\mathcal{C}$.
The notion of pre-rigidity, in its original form, stems from \cite{GV-OnTheDuality} although, as we will see, it turns out to be equivalent to the definition of \emph{weak dual} given in  \cite{DP}. A basic fact is that a (right) rigid monoidal category is right closed. It is also easily verified that right closedness implies pre-rigidity (cf. Proposition \ref{prop:prerigid2} below). Thus, what is in a sense missing in the notion of pre-rigidity compared to the notion of rigidity is the coevaluation.
\\Pre-rigidity arose in the study of so-called \textit{liftable} pairs of adjoint functors between monoidal categories. An adjoint pair of functors $(L,R)$ between monoidal categories $\Aa$ and $\Bb$ such that $R$ is a lax monoidal functor (or, equivalently, $L$ is colax monoidal) is called liftable if the induced functor $\overline{R}=\Alg(R):{\Alg}({\Aa})\rightarrow {\Alg}({\Bb})$ between the respective categories of algebra objects has a left adjoint and if the functor $\underline{L}=\Coalg(L):{\Coalg}({\Bb})\rightarrow {\Coalg}({\Aa})$ between the respective categories of coalgebra objects has a right adjoint. If $\Aa$ and $\Bb$ come both endowed with a braiding, it is shown in \cite[Theorem 2.7]{GV-OnTheDuality} that such a liftable pair of functors $(L,R)$ gives rise to an adjunction between the respective categories of bialgebra objects
$$
\xymatrix{
\Bialg(\Aa) \ar@<.5ex>[rr]^-{\uoR} && \Bialg(\Bb) \ar@<.5ex>[ll]^-{\uoL}
}
$$
provided the functor $R$ enjoys the property of being braided with respect to the braidings of $\Aa$ and $\Bb$. Using this fact, a theorem originally due to Michaelis (cf. \cite{Michaelis-ThePrimitives}) is proven for a particular class of liftable pairs of adjoint functors between symmetric monoidal categories, the main application being a version of this theorem for Turaev's Hopf group-(co)algebras (cf. \cite[Theorem 4.16]{GV-OnTheDuality}).
\\\\In loc.cit., however, the liftability condition in the motivating examples is shown to hold by rather ad-hoc methods. The present article's origin lies in finding a general setting where the liftability condition can be proved to hold. The final Section \ref{finalsection} of this paper shows that braided pre-rigid monoidal categories indeed provide an appropriate setting for liftability, but, while dealing with liftability and while having a closer look at the categories where Turaev's Hopf group-(co)algebras live, we came to study pre-rigid monoidal categories (that are not necessarily braided) \textit{an sich}, as we realised the bare notion of pre-rigidity may have its own right to exist. To the best of our knowledge such a study has not appeared elsewhere in literature.
\\Upon first sight, as hinted at above, the pre-rigidity property has some taste of right closedness of a monoidal category. We believed it was worth further investigating the relationship between these two concepts, giving rise to the work carried out in Section \ref{liftablesection}. Section \ref{Sectionexamples} dives deeper into the categories where Turaev's Hopf group-(co)algebras reside, which are variations on categories of families of objects of a given monoidal category, known as Zunino and Turaev categories. In this section, pre-rigidity and closedness of these categories are studied in detail, providing a broader picture of the motivating examples of the article \cite{GV-OnTheDuality} -where such questions where not adressed- and filling some other gaps in literature along the way. Let us now sketch in more detail the content of the present article.
\\\\In Section \ref{liftablesection}, we first recall the definition of pre-rigid category and we provide a first instance of a pre-rigid monoidal category that is not closed, see Example \ref{ex:lambek}, which is related to syntactic calculus for categorical grammars.
We then investigate the connections between pre-rigidity and with the existence of a dualizible object on the one hand, amongst other things connecting pre-rigidity to the notion of a category with weak duals (introduced in \cite{DP}), and between pre-rigidiy and closedness on the other. Here we also present two examples of not necessarily braided categories that provide an instance of a strict monoidal functor that does not preserve pre-duals.
Although Example \ref{ex:lambek} already showed that pre-rigidity and closedness are no synonyms, by the earlier-mentioned Proposition \ref{prop:prerigid2} one expects that closed monoidal categories and pre-rigid ones do share some properties, of course. For instance, it is known that a monoidal co-reflective full subcategory of a monoidal closed category is also closed. Proposition \ref{pro:prigadj} extends this result to the pre-rigid case, but it holds in a more general setting than a co-reflection since the relevant isomorphism needs not to be the unit.
Proposition \ref{pro:terminal} enables us to obtain further examples of pre-rigid monoidal categories (although with trivial pre-dual) some of which are not closed and which do not necessarily allow for a braiding (see Examples \ref{exa:terminal}).
Given a pre-rigid monoidal category $\Cc$, we conclude Section \ref{liftablesection} by considering the construction of a functor $(-)^{*}:\Cc^\op \rightarrow \Cc$ acting as the pre-dual on objects and we investigate some of its properties. In Section \ref{finalsection} we will study under which conditions this functor is part of a liftable pair of adjoint functors.
\\\\In Section \ref{Sectionexamples}, we explore different constructions of new pre-rigid monoidal categories starting from known ones, by considering the examples which play a key role in \cite[Section 4]{GV-OnTheDuality}: the ``Zunino category'' $\Fam(\Cc)$ of families of a base category $\Cc$ and the ``Turaev category'' $\Maf(\Cc)=\Fam(\mathcal{C}^{\op })^{\op }$. \\As shown in \cite[Section 2.1]{CaeDel}, the category $\Fam({\Vec})$ serves as a categorical framework for Turaev's Hopf group-algebras. Similarly, $\Maf({\Vec})$ catches Turaev's Hopf group-coalgebras. It is worth noticing that, in contrast with classical Hopf algebras, the definition of a Hopf group-coalgebra is not selfdual.
In the discussion right above, we denoted the monoidal category of vector spaces over a field $k$ as $\Vec$; when $\Vec$ is replaced by ${\Vec}^{\textrm{f}}$, the category of finite-dimensional vector spaces, objects in $\Fam({\Vec}^{\textrm{f}})$ and $\Maf({\Vec}^{\textrm{f}})$ are called ``locally finite''. Using this framework, Corollary 4.5 in \cite{GV-OnTheDuality} provides an equivalence between the categories of locally finite Hopf group-algebras and Hopf group-coalgebras, by establishing a suitable adjunction between $\Fam({\Vec}^{\textrm{f}})$ and $\Maf({\Vec}^{\textrm{f}})$. However, in \cite{GV-OnTheDuality}, the question of the pre-rigidity of these categories was not studied (let alone the question of their closedness), the liftability condition being the main concern.
\\We first recall the structures of $\Fam(\Cc)$ and $\Maf(\Cc)$. In Proposition \ref{pro:Famprerig}, we prove the pre-rigidity of the category $\Fam(\Cc)$ for $\Cc$ any pre-rigid monoidal category possessing products of pre-duals. From this, we may deduce that the Zunino category $\Fam({\Vec})$ is pre-rigid. We notice, however, that for Proposition \ref{pro:Famprerig} to hold, arbitrary products of pre-duals are needed: Remark \ref{Remark-predualsnecessary} teaches that $\Fam({\Vec}^{\textrm{f}})$ is not pre-rigid, for instance. To the best of our knowledge, it was unknown in literature under which conditions $\Fam(\Cc)$ inherits closedness from $\Cc$, except in case $\Cc$ is cartesian (which was considered in \cite{Carboni} and in \cite{AR}). Proposition \ref{pro:Famclosed} fills this gap: $\Fam(\Cc)$ is shown to be closed monoidal whenever $\Cc$ is closed monoidal and has products.
\\$\Maf(\Cc)$ is a slightly different story: $\Maf(\Cc)$ is \emph{never} pre-rigid (Proposition \ref{pro:Maf}). This implies that $\Maf({\Vec})$, the category where generic Hopf group-coalgebras reside, does not enjoy the same pre-rigidity property as $\Fam({\Vec})$, the home of Hopf group-algebras.
We are able, however, to adjust the situation a bit: we study $\Faf(\Cc)$, an interesting variant of the category $\Maf(\Cc)$ (cf. \cite{LMM}), and prove, in Proposition \ref{pro:Faf}, that it is pre-rigid monoidal whenever $\Cc$ is. Next, Proposition \ref{pro:catfun} asserts that, given a small category $\mathcal{I}$ and a complete pre-rigid monoidal category $\Cc$, the functor category $\left[ \mathcal{I},\Cc\right] $ is pre-rigid as well. Finally, we conclude this section by considering the category $\Mm^G$ of externally $G$-graded $\Mm$-objects where $G$ is a monoid and $\Mm$ is a given monoidal category. In Proposition \ref{pro:funcatG}, under mild assumptions, we prove that the category $\Mm^G$ is pre-rigid whenever $\Mm$ is. This will allow us to provide another interesting example of a pre-rigid monoidal category which is not right closed, namely the category of externally $\mathbb{N}$-graded finite-dimensional vector spaces, see Example \ref{examplerightclosed}.
\\\\In the final Section \ref{finalsection}, we propose to study liftability of adjoint pairs of functors in the light of general pre-rigid braided monoidal categories. As already mentioned above, in \cite{GV-OnTheDuality}, the liftability condition in the main examples is verified by ad-hoc methods. It is our purpose here to treat the case of generic pre-rigid braided monoidal categories in a more systematic way. We first recall what this liftability condition precisely is (Definition \ref{def:liftable}) and what this condition means for the bialgebra objects in the involved categories. Example \ref{ex:liftable} seems to be new and is considered to be of independent interest: we show that not every suitable adjunction is liftable. More precisely, setting $S=\frac{k \left[ X\right]}{\left( X^{2}\right) },$ we obtain that the induced functor $\overline{R^{f}}=\Alg(R^f):\Alg (\Vec ^{f} ) \to \Alg( \Vec ^{f}) $ of the functor $$R^{f}:\Vec ^{\textrm{f}}\rightarrow \Vec ^{\textrm{f}},\quad V\mapsto
S{\otimes}_{k} V$$
has no left adjoint (although $R^{f}$ itself does have one!).
In Proposition \ref{prop:prerigid}, we show that the pre-dual construction defines a special type of self-adjoint functor $R=(-)^*:\Cc^\op\to\Cc$. In Proposition \ref{lem:Barop}, we show that this type of functor gives rise to a liftable pair whenever the functor $\overline{R}=\Alg(R)$ it induces at the level of algebras has a left adjoint and we apply this result to the specific functor $R=(-)^*:\Cc^\op\to\Cc$ in Corollary \ref{coro:Isar}. Then, in Proposition \ref{pro:monadj}, we provide a criterion to transport the desired liftability from one category to another in presence of a suitable monoidal adjunction and we apply it, in Corollary \ref{coro:externlift}, to transfer liftability from a category $\Mm$ to the category $\Mm^\NN$ of externally $\NN$-graded objects. As a consequence, we arrive at the final Example \ref{example-prerigidnotclosed} which revisits Example \ref{examplerightclosed} and provides an instance of a situation in which Corollary \ref{coro:externlift} (properly) holds; it shows how -in favorable cases- the notion of pre-rigidity allows to construct liftable pairs of adjoint functors when right closedness of the category is not available.

\subsection{Notational conventions}\label{notations}
When $X$ is an object in a category $\Cc$, we will denote the identity morphism on $X$ by $1_X$ or $X$ for short. For categories $\Cc$ and $\Dd$, a functor $F:\Cc\to \Dd$ will be the name for a covariant functor; it will only be a contravariant one if it is explicitly mentioned. By $\id_{\Cc}$ we denote the identity functor on $\Cc$. For any functor $F:\Cc\to \Dd$, we denote $\Id_{F}$ (or sometimes -in order to lighten notation in some computations- just $F$, if the context does not allow for confusion) the natural transformation defined by $\Id_{FX}=1_{FX}$.
\\Let $\mathcal{C}$ be a category. Denote by $\mathcal{C}^{\op }$ the opposite category of $\mathcal{C}$. Using the notation of \cite[page 12]{Pareigis-CatFunct}, an object $X$ and a morphism $f:X\rightarrow Y$ in $%
\mathcal{C}$ will be denoted by $X^{\op }$ and $f^{\op }:Y^{%
\op }\rightarrow X^{\op }$ when regarded as object and
morphism in $\mathcal{C}^{\op }$. Given a functor $F:\Cc\to \Dd$, one defines its opposite functor $F^\op :\Cc^\op \to \Dd^\op $ by setting $F^\op X^\op =(FX)^\op $ and $F^\op f^\op =(Ff)^\op $. If $\alpha:F\to G$ is a natural transformation, its opposite $\alpha^\op$ is defined by $(\alpha^\op)_{X^\op}:=(\alpha_X)^\op$ for every object $X$.
\\\\Throughout the paper, we will work in the setting of monoidal categories. It is useful to recall the following notation. Let $(\Mm,\otimes ,\unit,a,l,r)$ be a monoidal category. Following \cite[0.1.4, 1.4]{SaavedraRivano}, we have that $\Mm^{\op }$ is also monoidal, the monoidal structure being given by%
\begin{gather*}
X^\op \otimes Y^\op  :=\left( X\otimes Y\right)
^\op ,\quad\text{ the unit is }\unit^{\op }, \\
a_{X^\op ,Y^\op ,Z^\op } :=\left(
a_{X,Y,Z}^{-1}\right) ^{\op },\quad l_{X^\op } :=\left( l_{X}^{-1}\right) ^{\op },\qquad
r_{X^\op }:=\left( r_{X}^{-1}\right) ^{\op }.
\end{gather*}%
If $\Mm$ is moreover braided (with braiding $c$), then so is $\Mm^{%
\op }$, the braiding being given by
\begin{equation*}
c_{X^\op ,Y^\op }:=\left( c_{X,Y}^{-1}\right) ^{%
\op }.
\end{equation*}
As already mentioned, in this note we will operate within the framework of monoidal categories, which will be assumed to be strict from now on. By Mac Lane's Coherence Theorem, this does not impose restrictions on the obtained results. We will moreover consider braided and (pre)additive monoidal categories. A basic reference for these notions is \cite{MacLane}, for instance.\medskip

Recall (see e.g. \cite[Definition 3.1]{Aguiar-Mahajan}) that a functor $F:\Aa\to\Bb$ between monoidal categories $(\Aa,\ot,{\unit}_{\Aa})$ and $(\Bb,\ot',{\unit}_{\Bb})$ is said to be a \emph{lax monoidal functor} if it comes equipped with a family of natural morphisms $\phi_{2}(X,Y):F(X)\ot' F(Y)\to F(X\ot Y)$, $X,Y\in \Aa$ and
a $\Bb$-morphism $\phi_0:{\unit}_{\Bb}\to F({\unit}_{\Aa})$, satisfying the known suitable compatibility conditions with respect to the associativity and unit constraints of $\Aa$ and $\Bb$. Moreover, $(F,\phi_0,\phi_{2})$ is called {\it strong} if $\phi_0$ is an isomorphism and $\phi_{2}(X,Y)$ is a natural isomorphism for any objects $X,Y\in\Aa$.
$(F,\phi_0,\phi_2)$ is called {\it strict} if $\phi_0$ is the identity morphism and $\phi_2$ is the identity natural transformation.
Dually, colax monoidal functors are defined.\\
Also recall that given a lax monoidal functor $(F,\phi_2,\phi_0)$, then $(F^{\op },\phi_2^{\op },\phi_0^{\op })$ is a colax monoidal functor, where we set $\phi_2^{\op }(X^{\op },Y^{\op }):=\phi_2(X,Y)^{\op }$, see e.g.  \cite[Proposition 3.7]{Aguiar-Mahajan}.\medskip

 Throughout the paper, $k$ will be a field. The category of vector spaces over $k$ will be denoted by $\Vec$  and endowed with its usual structure of monoidal category where the tensor product is $\ot_k$ and the unit object is $k$. If we take objects only to be finite-dimensional vector spaces, we will use the notation ${\Vec ^{\textrm{f}}}$.

\subsection{Some basic known facts}\label{sub:basics} We recall some basic notions and properties which are well-known and serve as a comparison for the results we will deal with in the present paper. Here $(\Cc,\otimes,\unit)$ denotes a monoidal category and every notion stated on the right admits its proper left analogue.
\begin{itemize}
  \item An object $X$ in $\Cc$ is called \emph{right dualizable} if there are an object $X^*$, called the \emph{right dual} of $X$, and morphisms $\ev_X:X^*\otimes X\to\unit $ (called the counit or the \emph{evaluation}) and $\coev_X:\unit\to X\otimes X^*$ (called the unit or the \emph{coevaluation}) in $\Cc$ that fulfill the triangle identities $(\ev_X\otimes X)(X^*\otimes\coev_X)=\id_{X^*}$ and $(X\otimes \ev_X)\circ (\coev_X\otimes X)=\id_X$. The left dual of an object $X$, if any, is denoted by $^*X$.
  \item If every object in $\Cc$ is right dualizable, the category $\Cc$ is called right \emph{rigid} (or \emph{autonomous}).
  \item $\Cc$ is said to be right \emph{closed}, if for
every object $X\in \Cc$ the functor $\left( -\right) \otimes X:%
\Cc\rightarrow \Cc$ has a right adjoint $\left[ X,-\right] :%
\Cc\rightarrow \Cc$, called the internal-hom. Note that any right rigid monoidal category is right closed with  $[X,-]:=(-)\otimes X^*$, see e.g. \cite[Theorem 1.3]{DP}. On the other hand closedness does not imply rigidity, e.g. $\Vec$ is closed but not rigid.
\item If $\Cc$ is braided then $\Cc$ is right rigid if and only if it is left rigid and the dual object is the same, see e.g. \cite[page 780]{Street12}. However there exist monoidal categories with objects $X$ such that ${}^{\ast} X\ncong X^{\ast }$ (see \cite[Example 7.19.5]{EGNO}\footnote{Note that, by \cite[Remark 2.10.3]{EGNO}, if $X$ has both a left and a right dual, then $(^{\ast }X)^\ast\cong X$. In particular, if ${}^{\ast} X\cong X^{\ast }$, then $X^{\ast\ast }\cong X$.}, e.g.).
\item If $\Cc$ is braided then $\Cc$ is right closed if and only if it is left closed as the braiding provides a functorial isomorphism $(-)\otimes X\cong X\otimes(-)$.
\item If $\Cc$ is a symmetric monoidal category which is also rigid, then $\Cc$ is called \emph{compact closed}.
\item $\Cc$ is called \emph{$*$-autonomous category}, see \cite[(4.3), page 13]{Ba-star-auto}, if it is symmetric monoidal and equipped with a fully faithful  functor $(-)^{*}:\Cc^\op \rightarrow \Cc$ such that there is a natural isomorphism $\hom _{\Cc}\left( A\otimes B,C^{\ast }\right) \cong \hom _{\Cc}\left( A,(B\otimes C)^{\ast }\right) $. Note that a $*$-autonomous category is in particular right closed monoidal with  $[X,Y]:=(X\otimes Y^*)^*$.
\item Strong monoidal functors between right rigid categories preserve duals, see e.g. \cite[Proposition 3]{Li}.
\end{itemize}

\section{Pre-rigid monoidal categories}\label{liftablesection}
In this section we recall the definition of pre-rigid category, we connect it to the notion of weak dual introduced in \cite{DP} and we investigate its connections to the notions of rigid and closed category, providing meaningful examples. Then we study pre-rigid categories with constant pre-dual. We conclude this section by considering the functor $(-)^{*}:\Cc^\op \rightarrow \Cc$ induced by the pre-dual which will play a central role in Section \ref{finalsection} in the context of liftability.

The following definition, in its original form, see \cite[4.1.3]{GV-OnTheDuality}, required the monoidal category to be braided, the motivation being to be able to consider bi and Hopf algebra objects therein. We here remove this hypothesis, since there are interesting examples of non-braided monoidal categories that fulfil the remaining conditions, see Example \ref{ex:closed} for instance.

\begin{definition}
A monoidal category $(\Cc,\otimes,\unit)$ is called
right {\it pre-rigid} if for every object $X$ there exists an object $X^{\ast }$ (a \emph{pre-dual} of $X$) and a
morphism $\ev_{X}:X^{\ast }\otimes X\rightarrow \unit$ (the \emph{evaluation at $X$}) with the following universal property: For every morphism $t:T\otimes X\rightarrow \unit$ there is a
unique morphism $t^\dag:T\rightarrow X^{\ast }$ such that $t=\ev%
_{X}\circ \left(t^\dag\otimes X\right) .$ Equivalently the map
\begin{equation}\label{maprerig}
\hom _{\Cc}\left( T,X^{\ast }\right) \rightarrow \hom %
_{\Cc}\left( T\otimes X,\unit\right) :u\mapsto \ev%
_{X}\circ \left( u\otimes X\right)
\end{equation}%
is bijective for every object $T$ in $\Cc$.

Similarly, one could define a monoidal category $(\Cc,\otimes,\unit)$ to be
left pre-rigid if for every object $X$ there exists an object ${}^{\ast} X$ and a
morphism $\ev'_{X}:X \otimes {}^{\ast} X\rightarrow \unit$ such
that the map
\begin{equation*}
\hom _{\Cc}\left( T,{}^{\ast}X\right) \rightarrow \hom %
_{\Cc}\left( X\otimes T,\unit\right) :u\mapsto \ev'%
_{X}\circ \left( X\otimes u\right)
\end{equation*}
is bijective for every object $T$ in $\Cc$.
\end{definition}

We now give a first example of left and right pre-rigid category related to categorical grammars.
\begin{example}\label{ex:lambek}
Recall that a \emph{pomonoid} is quadruple $(P,\leq,\cdot,1)$ where $(P,\leq)$ is a poset and $(P,\cdot,1)$ is a monoid such that the multiplication is monotone, i.e. $a\leq c$ and $b\leq d$ implies $a\cdot b\leq c\cdot d$ for all $a,b,c,d\in P$. A pomonoid $(P,\leq,\cdot,1)$ can be considered as a monoidal category  $\Pp$ whose objects are the elements in $P$ and where $\hom_\Pp(a,b)$ is a singleton if $a\leq b$ and is empty otherwise; the tensor product is given by $\cdot$ and the unit object is $1$. It is well-known that
\begin{itemize}
  \item $\Pp$ is left and right rigid if and only if the pomonoid $(P,\leq,\cdot,1)$ is a \emph{pregroup}, meaning that for every $t\in P$ there are elements $t^*,\,^*t\in P$, called \emph{proto-inverses}, such that $t^*\cdot t\leq 1\leq t\cdot t^*$ and $t\cdot {^* t}\leq 1\leq {^*t}\cdot t$;
  \item $\Pp$ is left and right closed if and only if the pomonoid $(P,\leq,\cdot,1)$ is a \emph{residuated pomonoid}, meaning that for every $b,c\in P$ there are elements in $P$, denoted by  $c/b$ and $a\backslash c$ and called \emph{residuals}, such that $a\cdot b\leq c\Leftrightarrow a\leq c/b\Leftrightarrow b\leq a\backslash c$ for every $a,b,c\in P$.
\end{itemize}
Instead, if we just ask $\Pp$ to be left and right pre-rigid, we get what we can name a \emph{contractive pomonoid}, meaning that, for every $t\in P$, there are $t^*,\,^*t\in P$ such that
\begin{gather}
 t^*\cdot t\leq 1 \text{ and } t\cdot {^* t}\leq 1,\text{ for every } t\in P\quad \text{(contractions)};\label{pomon1}\\
a\cdot b\leq 1\text{ implies both } a\leq b^* \text{ and } b\leq {^* a} ,\text{ for every } a,b\in P.\label{pomon2}
\end{gather}
It is easy to check that \eqref{pomon1} and   \eqref{pomon2} are equivalent to require that
\begin{gather}
a\cdot b\leq 1\Leftrightarrow a\leq b^*\Leftrightarrow b\leq {^* a}\text{  for every }a,b\in P.\label{pomon3}
\end{gather}
\begin{invisible} Assume that \eqref{pomon1} holds true. Then $a\leq b^*\Rightarrow a\cdot b\leq b^*\cdot b\leq 1\Rightarrow a\cdot b\leq 1$. Similarly $b\leq {^* a}\Rightarrow a\cdot b\leq a\cdot {^* a}\leq 1\Rightarrow a\cdot b\leq 1$. Adding \eqref{pomon2} we get \eqref{pomon3}. Conversely by assuming \eqref{pomon3} and observing that $t^*\leq t^*$ and ${^* t}\leq {^* t}$, we get \eqref{pomon1}, while \eqref{pomon2} is evident.
\end{invisible}
In particular, if $\Pp$ is left and right pre-rigid, then the pomonoid $(P,\leq,\cdot,1)$ fulfills \eqref{pomon1} i.e. it is a special instance of a \emph{protogroup}, notion introduce by Lambek in \cite{Lambek} (together with the one of pregroup) as an aid for checking which strings of words in a natural language, such as English, are well-formed sentences.
Note that the condition $a\leq b^*\Leftrightarrow b\leq {^* a}$ means that $(-)^*:P\to P$ and $^*(-):P\to P$ defines an order-reversing Galois connection. It follows from the definition that in a contractive pomonoid $t^*,\,^*t\in P$ are necessarily unique and the following properties follow easily
$$1^*=1={^*1},\quad t\leq {}^*(t^*),\quad t\leq(^*t)^*,\quad b^*\cdot a^*\leq(a\cdot b)^*\quad {}^*b\cdot {}^*a\leq {}^*(a\cdot b).$$
\begin{invisible}
In fact, if there is there is  $t^\star$ fulfilling $t^\star\cdot t\leq 1$ and $a\cdot t\leq 1\Rightarrow a\leq t^\star$, then $t^*\cdot t\leq 1$ plus $\eqref{pomon2}^\star$ imply $t^*\leq t^\star$ while $t^\star\cdot t\leq 1$ plus \eqref{pomon2} imply $t^\star\leq t^*$ and hence $t^*= t^\star$. Similarly on the left. Since $1\cdot 1\leq 1$ and $a\cdot 1\leq 1\Rightarrow a\leq 1$, we get $1^*=1$. Similarly ${^*1}=1.$
\end{invisible}
Note that
\begin{center}
  pregoups $\subseteq$ residuated monoids $\subseteq$ contractive pomonoids $\subseteq$ protogroups.
\end{center}
Explicitly, every residuated pomonoid is contractive through $t^*:=1/t$ and $^*t:=1\backslash t$. A simple example of contractive pomonoid is given by a pomonoid $(P,\leq,\cdot,1)$ where the identity $1$ is a maximum and $t^*={^*t}=1$ for every $t\in P$, as \eqref{pomon3} is trivially satisfied.  Consider $P=\{\frac{m}{10^n}\mid m,n\in\NN,m\geq10^n\}$, the set of terminating decimals greater or equal to $1$. Then $(P,\leq,\cdot,1)$, with the ordinary order and product of rational numbers, is a pomonoid with $1$ as a minimum, so that its opposite is a contractive pomonoid. We point out that this pomonoid is not residuated: otherwise it should admit the residual  $4/3=\min\{a\in P\mid 3a\geq 4\}$ but such a minimun does not exist in $P$.
\begin{invisible}
$1.34>\frac{4}{3}>1.33$; $1.334>\frac{4}{3}>1.333$; $1.3334>\frac{4}{3}>1.3333$; ...
\end{invisible}
\end{example}

From now on it suffices us to restrict our attention to right pre-rigid monoidal categories in this article, and henceforth we omit the adjective ``right''.\medskip

The lemma below provides a different characterization of pre-rigidity.
\begin{lemma}\label{lem:presheaf}
A  monoidal category $(\Cc,\otimes,\unit)$ is pre-rigid if and only if for every object $X$ there exists an object $X^{\ast }$ and an isomorphism of presheaves on $\Cc$
\begin{equation*}
\Pi:\hom _{\Cc}\left( -,X^{\ast }\right) \rightarrow \hom %
_{\Cc}\left( -\otimes X,\unit\right)
\end{equation*}
i.e. if and only if the presheaf $\hom _{\Cc}\left( -\otimes X,\unit\right):\Cc^{\op }\to\mathsf{Set}$ is representable.
\end{lemma}

\begin{proof}
  If $\Cc$ is pre-rigid we can set $\Pi_T(u):=\ev%
_{X}\circ \left( u\otimes X\right)$ and this assignment is natural in $T$. Conversely, given $\Pi$ we can set $\ev%
_{X}:=\Pi_{X^{\ast }}(1_{X^{\ast}})$. Given $u:T\to X^\ast$, the naturality of $\Pi$ yields the equality $\Pi_T\circ \hom _{\Cc}\left( u,X^{\ast }\right)=\hom _{\Cc}\left( u\otimes X,\unit\right)\circ \Pi_{X^{\ast }}$ which evaluated on $1_{X^{\ast}}$ gives the equality $\Pi_T(u):=\ev%
_{X}\circ \left( u\otimes X\right)$.
\end{proof}

Note that an object $X$ whose presheaf $\hom _{\Cc}\left( -\otimes X,\unit\right):\Cc^{\op }\to\mathsf{Set}$ is representable is called \emph{semi-dualizable} in \cite[Definition 4.5]{ST}. Moreover the object $X^*$ which represents this functor is called the \emph{weak dual} of $X$ in  \cite[page 86]{DP}.
Therefore, Lemma \ref{lem:presheaf} says that a monoidal category $\Cc$ is pre-rigid if and only if every object in $\Cc$ is semi-dualizable and that the notion of pre-dual coincides with the one of weak dual. As a consequence, we have the following result.

\begin{corollary}\label{cor:uniqueness}
  In a pre-rigid monoidal category, for every object $X$, the pre-dual $X^*$ is unique up to isomorphism. Moreover $\unit^*\cong   \unit.$
\end{corollary}

\begin{proof}
 This is already mentioned in \cite[page 86, footnote]{DP}. We include here a proof for the reader's sake. The uniqueness follows from Lemma \ref{lem:presheaf} and Yoneda's Lemma. If $r_T:T\otimes \unit\to T$ denotes the right unit constraint, then $\hom _{\Cc}\left( r_{-},\unit\right):\hom _{\Cc}\left( -,\unit\right) \rightarrow \hom _{\Cc}\left( -\otimes \unit,\unit\right) $ is a natural isomorphism. Thus $\unit$ is a pre-dual of $\unit$.
\end{proof}

The following result provides us with a large class of examples of pre-rigid monoidal categories.
\begin{proposition}\label{prop:prerigid2}
Let $\Cc$ be a right closed monoidal category. Then $%
\Cc$ is pre-rigid.
\end{proposition}

\begin{proof}As mentioned in \cite[4.2]{ST}, in a right closed monoidal category $\Cc$ every object is semi-dualizable. By the foregoing this means that $\Cc$ is pre-rigid. Explicitly, if the functor $\left( -\right) \otimes X:%
\Cc\rightarrow \Cc$ has a right adjoint $\left[ X,-\right] :%
\Cc\rightarrow \Cc$, the pre-dual is given by $X^{\ast }:=\left[ X,\unit%
\right] \in \Cc.$ Moreover, if $\epsilon :%
\left[ X,-\right] \otimes X\rightarrow \mathrm{Id}$ is the counit of the adjunction, then the evaluation is $\ev%
_{X}:=\epsilon _{\unit}.$
\begin{invisible}
As $\Cc$ is assumed to be right closed, this means that for
every object $X\in \Cc$ the functor $\left( -\right) \otimes X:%
\Cc\rightarrow \Cc$ has a right adjoint $\left[ X,-\right] :%
\Cc\rightarrow \Cc$. Set $X^{\ast }:=\left[ X,\unit%
\right] \in \Cc.$ Consider the counit of the adjunction $\epsilon :%
\left[ X,-\right] \otimes X\rightarrow \mathrm{Id}$ and set $\ev%
_{X}:=\epsilon _{\unit}.$ Clearly we have the following bijection.
\begin{equation*}
\hom _{\Cc}\left( T,X^{\ast }\right) \rightarrow \hom %
_{\Cc}\left( T\otimes X,\unit\right) :u\mapsto \ev%
_{X}\circ \left( u\otimes X\right) .
\end{equation*}%
This shows that $\Cc$ is pre-rigid.
\end{invisible}
\end{proof}

\subsection{Connections with rigid categories}We here discuss some connections with rigidity or, more generally, with the existence of a dualizable object.

\begin{remark}
We have observed that the notion of pre-dual coincides with the one of weak dual in the sense of  \cite{DP}. As a consequence, if $X$ is a dualizable object in a pre-rigid monoidal category  $\Cc$, then the pre-dual $X^*$ is exactly its dual in $\Cc$, cf. \cite[Theorem 1.3]{DP} plus the uniqueness of the pre-dual stated in Corollary \ref{cor:uniqueness}.
\end{remark}

\begin{remark}We observed in Subsection \ref{sub:basics} that a right rigid monoidal category  is right closed. On the other hand, by Proposition \ref{prop:prerigid2}, right closed implies pre-rigid. Thus, what is in a sense missing in the notion of pre-rigidity compared to the notion of rigidity is the coevaluation.
\end{remark}

\begin{remark}In Subsection \ref{sub:basics} we also observed that there is no distinction between left and right rigidity or closedness for a braided monoidal category  $(\Cc,\otimes,\unit)$. The same is true for pre-rigidity. For instance, if $\Cc$ is right pre-rigid, then it is also left pre-rigid with $^*X:=X^*$ and $\ev'_X:=\ev_X\circ c_{X,X^*}:X\otimes X^*\to\unit$, where $c$ denotes the braiding of $\Cc$.
\begin{invisible}The following bijection
\begin{equation*}
\hom _{\Cc}\left( T,X^*\right) \rightarrow\hom %
_{\Cc}\left( T\otimes X,\unit\right) \overset{\hom %
_{\Cc}\left( c_{X,T},\unit\right)}\rightarrow \hom %
_{\Cc}\left( X\otimes T,\unit\right)
\end{equation*}
maps $u:T\to X^*$ to $\ev_X\circ ( u\otimes X)\circ c_{X,T}=\ev_X\circ c_{X,X^*}\circ ( X\otimes u)=\ev'_X\circ ( X\otimes u).$
\end{invisible}
\end{remark}

We describe the pre-dual in case of a particular  example of compact closed monoidal category.

\begin{example}\label{ex:Rel}
Consider the category $\Rel $ whose objects are sets
and morphisms are binary relations, see \cite[page 26]{MacLane}. A binary relation $\Rr \subseteq
I\times J$ is then denoted by $\Rr :I\relto J.$ Given $\mathcal{S%
}:U\relto I$ the composition $\Rr \circ \mathcal{S}:U\relto
J$ is defined by setting
\begin{equation*}
\Rr \circ \mathcal{S}:=\left\{ \left( u,j\right) \in U\times J\mid
\exists i\in I\quad \left( i,j\right) \in \Rr \quad \left( u,i\right)
\in \mathcal{S}\right\} .
\end{equation*}%
By \cite[page 194]{Kelly-Laplaza}, the category $\Rel $ is compact closed where $\otimes$ is the cartesian product $\times$ and  the unit object the
singleton $\left\{ \ast \right\} .$ In particular it is rigid whence a fortiori both closed and pre-rigid.
Explicitly the internal-hom is given by $\left[ J,K\right] :=J\times K$ as
we have the obvious bijection%
\begin{equation*}
\hom _{\Rel }\left( I,\left[ J,K\right] \right) \cong \mathrm{%
Hom}_{\Rel }\left( I\times J,K\right) .
\end{equation*}%
As a consequence the pre-dual of an object $I$ is $I^*:=I$. The evaluation $\ev_{I}:I\times I\relto \left\{ \ast \right\} $ is the
binary relation $\ev_{I}:=\left\{ \left( i,i,\ast \right) \mid i\in
I\right\} .$  Given a binary relation $\Rr:I\times J\relto \left\{ \ast
\right\} \ $we define $\Rr^\dag:I\relto J$ by setting $%
\Rr^\dag:=\left\{ \left( i,j\right) \in I\times J\mid \left(
\left( i,j\right) ,\ast \right) \in \Rr\right\} $ so that $\Rr^\dag$ is the unique binary relation such that $\ev_{J}\circ
\left( \Rr^\dag\times \mathrm{Id}_{J}\right) =\Rr$.

For future use, note that, given binary relations $%
\Rr :I\relto J$ and $\Rr ^{\prime }:I^{\prime
}\relto J^{\prime }$, their cartesian product $\Rr \times
\Rr ^{\prime }:I\times I^{\prime }\relto J\times J^{\prime }$ is
defined by
\begin{equation*}
\Rr \times \Rr ^{\prime }:=\left\{ \left( i,i^{\prime
},j,j^{\prime }\right) \mid \left( i,j\right) \in \Rr \quad \left(
i^{\prime },j^{\prime }\right) \in \Rr ^{\prime }\right\} .
\end{equation*}%
\end{example}

\subsection{Connections with closed categories} We observed in Subsection \ref{sub:basics} that $\Vec$ is closed but not rigid. In particular, by Proposition \ref{prop:prerigid2}, $\Vec$ is pre-rigid but not rigid. Yet, it can be given a braided (symmetric) structure by considering the twist. We now consider two examples of not necessarily braided categories that provide an instance of a strict monoidal functor that does not preserve pre-duals. In contrast, as mentioned in Subsection \ref{sub:basics}, the strong monoidal functors between right rigid categories preserve duals.

\begin{example}\label{ex:closed}
The monoidal category of vector spaces graded by a monoid $G$ will be denoted by $\Vec_G$. Let us briefly recall its structure, see e.g. \cite[Chapter A]{NvO} and \cite[Example 2.3.6]{EGNO}.
\\If $V,W \in \Vec_{G}$, then $V\ot W:= \bigoplus_{g}(V\ot W)_g$, where $(V\ot W)_g:=\oplus_{xy=g} V_x \ot_k W_y$, becomes an object in $\Vec_G$. The unit object is $k=k_e$, $e$ being the neutral element in $G$.
For objects $V=\oplus_{g\in G}V_g, W=\oplus_{g\in G}W_g \in \Vec_{G}$, we set  $[V,W] =\oplus _{g\in G}[ V,W] _{g}$ where

\begin{equation*}
[V,W]_{g}=\left\{ f\in \mathrm{Hom} _{k}\left(
V,W\right) \mid f\left( V_{h}\right) \subseteq W_{gh}\text{ for every }h\in
G\right\}.
\end{equation*}
This  defines an endofunctor $[V,-]$ of $\mathsf{Vec} _{G}$ which is a right adjoint of the endofunctor $\left( -\right) \otimes V$.
On the other hand the right adjoint of the endofunctor $ V \otimes\left( -\right)$ is $\mathrm{HOM}\left(V,-\right)$ which is defined by setting $\mathrm{HOM}\left(
V,W\right) =\oplus _{g\in G}\mathrm{HOM}\left( V,W\right) _{g}$ where
\begin{equation*}
\mathrm{HOM}\left( V,W\right) _{g}=\left\{ f\in \mathrm{Hom} _{k}\left(
V,W\right) \mid f\left( V_{h}\right) \subseteq W_{hg}\text{ for every }h\in
G\right\},
\end{equation*}
for any $V,W\in \mathsf{Vec}_{G}$. Note that this latter inner-hom is the one usually employed in graded theory.

Since the endofunctor $\left( -\right) \otimes V$ of $\mathsf{Vec%
}_{G}$ has a right adjoint $[V,-]$, the category $\Vec_G$ is right closed.  Proposition \ref{prop:prerigid2} teaches that $\Vec_G$ is also pre-rigid and the pre-dual  $V^{\ast}$ becomes $[V,k]$ for any $G$-graded vector space $V$. Now, unless $G$ is a commutative monoid, $\Vec_G$ can not be endowed with a braiding.
Notice also that, unless $G$ is trivial, the pre-dual $V^{\ast}$ of an object $V$ in $\Vec_G$ does not coincide with the set of morphisms from $V$ to $k$ in this category. Thus the forgetful functor $U:\Vec_G\to \Vec$ between these two pre-rigid monoidal categories is strict monoidal but does not preserve pre-duals.
\end{example}

The category $\Vec_G$ is a particular instance of the following more general situation for $H=kG$, the group algebra.

\begin{example}\label{ex:comodclosed}
Let $H$ be a $k$-Hopf algebra. Then the category $^H\mathfrak{M}$ of left $H$-comodules is closed monoidal (as mentioned in \cite[page 362]{Schauenburg-HopExt}; see also \cite[Theorem 1.3.1]{Hovey}). Hence by Proposition \ref{prop:prerigid2} it is also pre-rigid. Note that the category itself needs not to be braided unless $H$ is coquasi-triangular (dual of \cite[Proposition XIII.1.4.]{Kassel}). From a communication with G. B\"{o}hm, it emerged that the existence of the antipode in Hovey's proof is superfluous so that the category of left $H$-comodules is closed even if $H$ is just a bialgebra. \\As in the particular case of $\Vec_G$, here also the forgetful functor $U:\!^H\mathfrak{M}\to \Vec$ between these two pre-rigid monoidal categories is strict monoidal but does not preserve pre-duals.
\end{example}

The very definition of a pre-rigid monoidal category together with Proposition \ref{prop:prerigid2} suggests a close connection between the notions of closedness and pre-rigidity. However, these notions are no synonyms. We already encountered an example of a pre-rigid monoidal category which is not closed namely the category associated to a contractive pomonoid, cf. Example \ref{ex:lambek}. We come back to other examples soon in Examples \ref{exa:terminal} and \ref{examplerightclosed}.
\\But closed monoidal categories and pre-rigid ones share quite some properties, of course.
For instance, it is known that a monoidal co-reflective full subcategory of a monoidal closed category is also closed (although the closed structure may not be preserved by the inclusion), see e.g. \cite[Lemma 4.1]{Hasegawa}. The following proposition extends this result to the pre-rigid case but it holds in a more general setting than a co-reflection since the relevant isomorphism needs not to be the unit.

\begin{proposition}
\label{pro:prigadj}Let $\Aa$ and $\Bb$ be monoidal
categories and let $\left( L,R\right) $ be an adjunction such that $L:%
\Bb\rightarrow \Aa$ is strong monoidal. If $\Aa$ is
pre-rigid then the presheaf $\hom _{\Bb}\left( -\otimes B,RL\left(
\unit_{\Bb}\right) \right)$ is representable. If furthermore $\unit_{\Bb}\cong RL\left(
\unit_{\Bb}\right) $ (e.g. $L$ is fully faithful), then $\Bb$ is
pre-rigid too.
\end{proposition}

\begin{proof}
Denote by $\left( L,\psi _{2},\psi _{0}\right) $ the strong monoidal
structure of $L$. For every $B\in \Bb$ we set $B^{\ast }:=R\left( \left( LB\right)
^{\ast }\right) $. Then, for every $T\in \Bb$, we have the following
chain of isomorphisms%
\begin{align*}
\hom _{%
\Bb}\left( T,R\left( \left( LB\right) ^{\ast }\right) \right) &\cong
\hom _{\Aa}\left( LT,\left( LB\right) ^{\ast }\right) \cong
\hom _{\Aa}\left( LT\otimes LB,\unit_{\Aa%
}\right)  \cong\hom _{\Aa}\left( L\left( T\otimes B\right) ,\unit_{\Aa}\right)\\
&\cong  \hom _{\Bb}\left( T\otimes
B,R\left( \unit_{\Aa}\right) \right)\cong  \hom _{\Bb}\left( T\otimes
B,RL\left(
\unit_{\Bb}\right) \right)
\end{align*}
where the last isomorphism is induced by the isomorphism $R\psi _{0}: RL\left(
\unit_{\Bb}\right)\to R\left(\unit_{\Aa}\right)$.
The above displayed composition is natural in $T$ so that we get a natural isomorphism $\hom _{%
\Bb}\left( -,R\left( \left( LB\right) ^{\ast }\right) \right)\cong \hom _{\Bb}\left( -\otimes B,RL\left(
\unit_{\Bb}\right) \right)$ which says that the presheaf $\hom _{\Bb}\left( -\otimes B,RL\left(
\unit_{\Bb}\right) \right)$ is representable. Now $\unit_{\Bb}\cong RL\left(
\unit_{\Bb}\right) $ induces an isomorphism  $\hom _{\Bb}\left( -\otimes B,RL\left(
\unit_{\Bb}\right) \right)\cong \hom _{\Bb}\left( -\otimes B,\unit_{\Bb} \right)$ so that the latter presheaf is representable as well. By Lemma \ref{lem:presheaf} we conclude.
Note that if $L$ is fully faithful then the unit $\eta:\id_{\Bb}\to RL$ is an isomorphism by the dual of \cite[Proposition 3.4.1]{Borceux1}. In particular $\eta \unit_{\Bb}:\unit_{\Bb}\to RL\left(
\unit_{\Bb}\right)$ is an isomorphism.
\end{proof}

Now follows a variant of Proposition \ref{pro:prigadj} that holds in case $\Bb$ is Cauchy complete i.e. when every idempotent in $\Bb$ splits. Recall that an object $A$ in a category is called a \emph{retract} of an object $B$ if there are morphisms $i:A\to B$ and $p:B\to A$ such that $p\circ i= 1_A$.
A functor $F: \Cc \rightarrow \Dd$ is called \emph{separable} if the obvious natural transformation  $\mathcal{F} : \hom _{\Cc}(-,-)\rightarrow \hom _{\Dd}(F(-), F(-))$ is a split natural monomorphism.

\begin{proposition}
\label{pro:prigadj2}Let $\Aa$ and $\Bb$ be monoidal
categories and let $\left( L,R\right) $ be an adjunction such that $L:%
\Bb\rightarrow \Aa$ is strong monoidal. Assume that $\Bb$ is Cauchy complete and that $\unit_{\Bb}$ is a retract of $RL\left(
\unit_{\Bb}\right) $ (e.g. when $L$ is a separable functor). Then, if $\Aa$ is
pre-rigid, so is $\Bb$.
\end{proposition}

\begin{proof}
By definition of retract, there are morphisms $p:RL\left(
\unit_{\Bb}\right)\to \unit_{\Bb}$ and $i:\unit_{\Bb}\to RL\left(
\unit_{\Bb}\right)$ such that $p\circ i=1_{\unit}$. Then $\hom _{\Bb}\left( T\otimes B,p\right)\circ \hom _{\Bb}\left( T\otimes B,i\right)$ is trivial and hence $\hom _{\Bb}\left( -\otimes B,\unit_{\Bb} \right)$ is a retract of the presheaf $\hom _{\Bb}\left( -\otimes B,RL\left(
\unit_{\Bb}\right) \right)$, which is representable by Proposition \ref{pro:prigadj}. Thus, by \cite[Lemma 6.5.6]{Borceux1} (which can be applied since $\Bb$ is Cauchy complete), we get that $\hom _{\Bb}\left( -\otimes B,\unit_{\Bb} \right)$ is representable as well.
By Lemma \ref{lem:presheaf}, $\Bb$ is pre-rigid.
\\Notice that if $L$ is a separable functor then the unit $\eta:\id\to RL$ is a split natural monomorphism by Rafael's Theorem. In particular $i:=\eta \unit_{\Bb}:\unit_{\Bb}\to RL\left(
\unit_{\Bb}\right)$ is a split monomorphism.
\end{proof}

\begin{example}
Recall that a \emph{monoidal comonad} on a monoidal category $(\Mm,\otimes,\unit)$ consists of a comonad $(\bot,\delta,\epsilon)$ on $\Mm$ such that $\bot:\Mm\to\Mm$ is lax monoidal and $\delta:\bot\to\bot\bot$ and $\epsilon:\bot\to\unit$ are monoidal natural transformations, see e.g. \cite{Pastro-Street}.

Let $\bot$ be a comonad on a monoidal category $\Mm$. By the dual version of \cite[Corollary 3.13]{McCrudden} (see also \cite[Theorem 3.19]{Bohm}) there is a bijective correspondence between monoidal structures on the Eilenberg-Moore category of $\bot$-comodules $\Mm^\bot$ such that the forgetful functor $L:\Mm^\bot\to \Mm$ is strict monoidal and monoidal comonad structures on the functor $\bot$. Thus if $\bot$ is a monoidal comonad and also a coseparable comonad, i.e. when $L$ is a separable functor (see \cite[Theorem 1.6]{EV}), then we can apply Proposition \ref{pro:prigadj2} to conclude that $\Mm^\bot$ is pre-rigid in case $\Mm$ is pre-rigid and $\Mm^\bot$ is Cauchy complete (note that $L$ has a right adjoint given by the free functor).

Consider a $k$-bialgebra $B$ whose underlying coalgebra is coseparable. Let $\mathfrak{M}^B$ be the category of right $B$-comodules. Since $k$ is a field, $\bot:=-\otimes_k B:\mathfrak{M}\to \mathfrak{M}$ is left exact and hence the forgetful functor $U:\mathfrak{M}^B\to \mathfrak{M}$ creates equalizers. Since $U$ has a right adjoint given by the functor $-\otimes_k B$, we can apply Beck's Theorem (see \cite{MacLane}), to get that the comparison functor $K:\mathfrak{M}^B\to \mathfrak{M}^\bot$ is a category isomorphism. Since $B$ is coseparable, the functor $U=L\circ K$ is separable and hence also $L:\mathfrak{M}^\bot\to \mathfrak{M}$ is separable. Since $B$ is a bialgebra, $\mathfrak{M}^B$ is monoidal and $U$ is  strict monoidal. As a consequence $\mathfrak{M}^\bot$ is monoidal and $L$ is  strict monoidal. Hence, by the foregoing, $\mathfrak{M}^\bot$ is pre-rigid. Unfortunately for our theory developed so far, this can also be deduced from the stronger fact that $\mathfrak{M}^B$ is closed which holds even if $B$ is not coseparable, see Example \ref{ex:comodclosed}.
\end{example}

\subsection{Examples with constant pre-dual}
The next result enables us to obtain further examples of pre-rigid monoidal categories (although with trivial pre-dual) some of which are not closed and which do not necessarily allow for a braiding. It also gives a criterium to exclude right closedness.
\begin{proposition}\label{pro:terminal}
  Let $(\Cc,\otimes,\unit)$ be a monoidal category.
  \begin{enumerate}
    \item[$1)$] If the unit object $\unit$ is terminal in $\Cc$, then $\Cc$ is pre-rigid where $X^*:=\unit$ for every $X$ in $\Cc$.
    \item[$2)$] If the unit object $\unit$ is initial in $\Cc$ and $\Cc$ is pre-rigid, then $\unit$ is terminal in $\Cc$.
    \item[$3)$] If the unit object $\unit$ is initial in $\Cc$ and the skeleton of $\Cc$ is not the trivial category, then   $\Cc$ is not right closed.
  \end{enumerate}
\end{proposition}

\begin{proof}
  $1)$. If $\unit$ is terminal, for every object $X$ there is a unique morphism $t_X:X\to \unit$ such that $\hom _{\Cc}\left(X, \unit\right)=\{t_X\} $. Set $X^*:=\unit$ and $\ev_{X}:=t_{X^*\otimes X}$. The map  \eqref{maprerig} is trivially bijective for every object $T$ in $\Cc$ and hence $\Cc$ is pre-rigid.

  $2)$. For every object $X$ in $\Cc$ we have $\hom _{\Cc}\left(X, \unit\right)\cong\hom _{\Cc}\left(\unit\otimes X, \unit\right)\cong\hom _{\Cc}\left(\unit, X^*\right)$ and the latter is a singleton if $\unit$ is an initial object. Thus $\unit$ is also terminal in this case.

  $3)$. Assume that $\unit$ is an initial object and that $\Cc$ is  right closed. Then, for every object $X,Y$ in $\Cc$, the functor $(-)\otimes X$ has a right adjoint $[X,-]$ and hence we have $\hom _{\Cc}\left(X, Y\right)\cong \hom _{\Cc}\left(\unit\otimes X, Y\right)\cong \hom _{\Cc}\left(\unit, [X,Y]\right)$ and the latter is a singleton. Then $X$ is an initial object as well, for every $X$. So $X\cong \unit$ and the skeleton of $\Cc$ is the trivial category.
\end{proof}

\begin{remark}
A monoidal category $\Cc$ whose unit object $\unit$ is a terminal object is called \emph{semicartesian} in literature.
  Rephrasing the first item from the above proposition, we have that being semicartesian implies pre-rigidity for a monoidal category. The converse of this statement does not hold; consider $\Vec$ for instance.
\end{remark}

\begin{example}
As observed in \cite[page 15]{Kelly}, the symmetric cartesian monoidal category ${\sf Top}$ is not closed. Note that the unit object of this monoidal category is the singleton and it is also a terminal object. Therefore, by  Proposition \ref{pro:terminal}, ${\sf Top}$ is pre-rigid.
\end{example}

\begin{example}\label{ex:slice}
  Let $(\Cc,\otimes,\unit)$ be a monoidal category. Then the slice category $\Cc/\unit$ becomes monoidal with unit object $(\unit,\id_\unit)$, see \cite[page 3]{BJT}, which is also terminal. By Proposition \ref{pro:terminal},  $\Cc/\unit$ is pre-rigid.
\end{example}

\begin{examples}\label{exa:terminal}
We collect here further examples where Proposition \ref{pro:terminal} applies. Here $(\Cc,\otimes,\unit)$ denotes a braided monoidal category.
  \begin{enumerate}[1.]
    \item Since $\Cc$ is braided, the category $\Coalg(\Cc)$ of coalgebra objects in $\Cc$ is monoidal. It has unit object $\unit$ which is also a terminal object in $\Coalg(\Cc)$ via the counit. Hence this category is pre-rigid. Note that there are conditions on $\Cc$ ensuring that $\Coalg(\Cc)$ is closed, see e.g. \cite[3.2]{Po}.
      \item The category $\Alg^+(\Cc)$ of augmented algebra objects in $\Cc$ is monoidal too with unit object $\unit$ which is also a zero object (both terminal and initial) in $\Alg^+(\Cc)$. Hence this category is pre-rigid but not right closed. The fact that $\Alg^+(\Cc)$ is pre-rigid can also be deduced from Example \ref{ex:slice} and the identification $\Alg^+(\Cc)\cong \Alg(\Cc)/\unit$.
    \item If we further assume that $\Cc$ is symmetric monoidal, then the category $\Bialg(\Cc)$ of bialgebra objects in $\Cc$ is monoidal too (see \cite[1.2.7]{Aguiar-Mahajan}) with unit object $\unit$ which is also a zero object  in $\Bialg(\Cc)$. Hence this category is pre-rigid but not right closed, as its skeleton is not trivial, in general. The same argument applies to $\Hpfalg(\Cc)$.
    \item Assume that $\Cc$ has a terminal object $\terminal$ such that $\terminal\ncong \unit$ (e.g. in $\Vec$ one has $\{0\}\ncong k$). The unique morphisms $m:\terminal\otimes \terminal\to \terminal$ and $u:\unit\to \terminal$ turn $(\terminal,m,u)$ into an object in $\Alg(\Cc)$: indeed, $m\circ(m\otimes \terminal)$ and $m\circ(\terminal\otimes m)$ coincide as they both have $\terminal$ as target. Analogously one checks that $m$ is unitary. By a similar argument one shows that this algebra is a terminal object in $\Alg(\Cc)$. On the other hand $\unit$ is initial in $\Alg(\Cc)$ via the unit. We deduce that $\Alg(\Cc)$ is not pre-rigid by negation of 2) in Proposition \ref{pro:terminal}.
  \end{enumerate}
\end{examples}

\subsection{The functor induced by the pre-dual} Here, given a pre-rigid monoidal category $\Cc$ we consider the construction of a functor $(-)^{*}:\Cc^\op \rightarrow \Cc$ acting as the pre-dual on objects and investigate some of its properties. In Section \ref{finalsection} we will study under which conditions this functor is part of a liftable pair of adjoint functors.\\

The following fact is implicitly understood in \cite{GV-OnTheDuality}. We write it for future reference.

\begin{lemma}\label{lem:contravariant}
Let $\Cc$ be a pre-rigid monoidal category, the assignment $X\mapsto X^*$ induces a functor $(-)^{*}:\Cc^\op \rightarrow \Cc$ such that \eqref{maprerig} is natural in $T$ and $X$.
\end{lemma}

\begin{proof}
For every morphism $t:T\otimes X\rightarrow \unit$ there is a
unique morphism $t^\dag:T\rightarrow X^{\ast }$ such that $t=\ev%
_{X}\circ \left( t^\dag\otimes X\right) .$
For every morphism $f:X\rightarrow Y$ define $f^{\ast }:=({\ev%
_{Y}\circ \left( Y^{\ast }\otimes f\right) })^\dag:Y^{\ast }\rightarrow X^{\ast }$
so that%
\begin{equation}
\ev_{X}\circ \left( f^{\ast }\otimes X\right) =\ev_{Y}\circ
\left( Y^{\ast }\otimes f\right) .  \label{form:dual}
\end{equation}%
This defines the desired  functor. The naturality of \eqref{maprerig} in $T$ has already been observed, while the one in $X$ follows from \eqref{form:dual}.
\end{proof}


\begin{remark}
  We observed in Subsection \ref{sub:basics} that a $*$-autonomous category $\Cc$ is necessarily closed whence, a fortiori, pre-rigid. By definition, such a category is  equipped with a functor $(-)^{*}:\Cc^\op \rightarrow \Cc$ and a natural isomorphism $\hom _{\Cc}\left( A\otimes B,C^{\ast }\right) \cong \hom _{\Cc}\left( A,(B\otimes C)^{\ast }\right) $. Note that, if $\Cc$ is just a pre-rigid monoidal category, in view of Lemma \ref{lem:contravariant}, such a functor always exists and there is also the aforementioned  isomorphism (without asking $\Cc$ to be symmetric) as
  $$\hom _{\Cc}\left( A\otimes B,C^{\ast }\right)\cong\hom _{\Cc}\left( (A\otimes B)\otimes C,\unit\right) \cong \hom _{\Cc}\left( A\otimes (B\otimes C),\unit\right) \cong\hom _{\Cc}\left( A,(B\otimes C)^{\ast }\right) .$$
\end{remark}

The next result characterizes the situation when the functor $(-)^{*}:\Cc^\op \rightarrow \Cc$ is \emph{self-adjoint (on the right)} i.e. when there is a bijection $\hom _{\Cc}\left( Y,X^{\ast }\right) \cong \hom _{\Cc}\left( X,Y^{\ast }\right)$, natural both in $X$ and $Y$, see e.g. \cite[Chapter IV, Section 5]{MacLane-Moerdijk}.

\begin{proposition}\label{pro:adjprig}
Let $\Cc$ be a pre-rigid monoidal category. The following are equivalent.
\begin{enumerate}
   \item[$(1)$] There is a bijection $\hom _{\Cc}\left( Y,X^{\ast }\right) \cong \hom _{\Cc}\left( X,Y^{\ast }\right) $ natural both in $X$ and $Y$.
  \item[$(2)$] There is a bijection $\hom _{\Cc}\left( X\otimes Y,\unit\right) \cong \hom _{\Cc}\left( Y\otimes X,\unit\right) $ natural both in $X$ and $Y$.
\end{enumerate}
\end{proposition}

\begin{proof} The statement follows from the following diagram where the two vertical arrows are bijective by pre-rigidity and natural in $X$ and $Y$ by Lemma \ref{lem:contravariant}.
\begin{equation*}\xymatrixcolsep{1cm}\xymatrixrowsep{.5cm}
\xymatrix{ \hom _{\Cc}\left( X,Y^{\ast }\right) \ar[rr]\ar[d]^\wr  &&\hom _{\Cc}\left( Y,X^{\ast }\right)\ar[d]^\wr\\
\hom _{\Cc}\left( X\otimes Y,\unit\right) \ar[rr]&& \hom _{\Cc}\left( Y\otimes X,\unit\right)
}
\end{equation*}\qedhere
\end{proof}

\begin{remark}
The easiest way to apply Proposition \ref{pro:adjprig} to get that $(-)^*$ is self-adjoint is to require that the category is braided (this idea will be applied in the proof of Proposition \ref{prop:prerigid}). Let us show, by means of an example, that the existence of the braiding is not necessary.
Let $(\Cc,\otimes,\unit)$ be a monoidal category such that the unit object $\unit$ is terminal in $\Cc$. By Proposition \ref{pro:terminal}, $\Cc$ is pre-rigid. Although $\Cc$ is not necessarily braided, we can apply Proposition \ref{pro:adjprig} to get that the functor $(-)^*:\Cc^\op \to\Cc$ is self-adjoint since $\hom _{\Cc}\left( X\otimes Y,\unit\right) \cong \hom _{\Cc}\left( Y\otimes X,\unit\right) $ is trivially bijective, being $\unit$ a terminal object, and natural both in $X$ and $Y$. An instance of this situation is given by the category $\Coalg(\Cc)$ for $\Cc$ braided monoidal, see Example \ref{exa:terminal}. Note that in general $\Coalg(\Cc)$ is not braided unless $\Cc$ is symmetric, see e.g. \cite[1.2.7]{Aguiar-Mahajan}.
\end{remark}

\section{Natural constructions of pre-rigid categories}\label{Sectionexamples}

One of the aims of the present paper is to explore different constructions of new pre-rigid monoidal categories starting from known ones.

In \cite{GV-OnTheDuality}, the first examples of braided pre-rigid monoidal categories that are considered in the context of liftability of the adjoint pair of functors provided by taking pre-duals are $\Vec$ with the twist as symmetry and $\Vec_{\ZZ_{2}}$ endowed with the ``super'' symmetry. In this section, we consider some further examples some of which play a key role in loc. cit.: the category $\Fam(\Cc)$ of families of a base category $\Cc$ and the category $\Maf(\Cc)=\Fam(\Cc^{\op })^{\op }$, which were both appearing in \cite{CaeDel}, with the eye on Turaev's group Hopf-(co)algebras, see next definition below. Remark that in \cite{GV-OnTheDuality}, the question of their pre-rigidity was not studied.
	 \\In Proposition \ref{pro:Famprerig} we prove the pre-rigidity of the category $\Fam(\Cc)$ for $\Cc$ any  pre-rigid monoidal category possessing products of pre-duals; conversely if $\Cc$ has an initial object and $\Fam(\Cc)$ is pre-rigid then $\Cc$ necessarily has products of pre-duals.  We notice that for Proposition \ref{pro:Famprerig} to hold, arbitrary products of pre-duals are needed. Indeed, Remark \ref{Remark-predualsnecessary} teaches that $\Fam({\Vec ^{\textrm{f}}})$ is not pre-rigid, although ${\Vec ^{\textrm{f}}}$ has a rigid monoidal structure. For the sake of ``completeness'' we also show in Proposition \ref{pro:Famclosed} that $\Fam(\Cc)$ is closed monoidal whenever $\Cc$ is  closed and has products.
\\$\Maf(\Cc)$ is a slightly different story: $\Maf(\Cc)$ is \emph{never} pre-rigid (Proposition \ref{pro:Maf}). We are able, however, to adjust the situation a bit: we study a variant $\Faf(\Cc)$ of this category and prove, in Proposition \ref{pro:Faf} that it is pre-rigid monoidal whenever $\Cc$ is. Next, Proposition \ref{pro:catfun} asserts that, given a small category $\mathcal{I}$ and a complete pre-rigid monoidal category $\Cc$, the functor category $\left[ \mathcal{I},\Cc\right] $ is pre-rigid as well. Finally, we conclude this section by considering the category $\Mm^G$ of externally $G$-graded $\Mm$-objects where $G$ is a monoid and $\Mm$ is a given monoidal category: in Proposition \ref{pro:funcatG}, we prove that the category $\Mm^G$ is pre-rigid whenever $\Mm$ is pre-rigid under mild assumptions. This will allow us to provide a non-trivial example of a pre-rigid monoidal category which is not right closed, see Example \ref{examplerightclosed}.

\subsection{The ``Zunino'' category of families \texorpdfstring{$\Fam(\Cc)$}{TEXT}}
Recall from \cite[\S 3]{Benabou-Fibered} the definition of the category $\Fam(\Cc)$ of families of $\Cc$. An object in this category is a pair $\underline{X}:=\left( X_{i}\right)
_{i\in I}=\left( I,X_{i}\right) $ where $I$ is a set and $X_{i}$ is an
object in $\Cc$ for all $i\in I$. Given two objects $\underline{X}=\left(
I,X_{i}\right) $ and $\underline{Y}=\left( J,Y_{j}\right) $ a morphism $%
\underline{X}\rightarrow \underline{Y}$ is a pair $\underline{\phi }:=\left(
f,\phi _{i} \right):\underline{X}\rightarrow
\underline{Y}$ where $f:I\rightarrow J$ is map and $\phi _{i
}:X_{i}\rightarrow Y_{f(i)}$ is a morphism in $\Cc$ for all $i \in I$.

If $\Cc$ is a monoidal category so is $\Fam(\Cc)$, as
follows (see e.g. \cite[Section 4]{GV-OnTheDuality}). Given objects $\underline{X}$ and $\underline{Y}$ as above, their
tensor product is defined by $\underline{X}\otimes \underline{Y}=\left(
I\times J,X_{i}\otimes Y_{j}\right) .$ Given morphisms $\underline{\phi }%
:=\left( f,\phi _{i}\right) :\underline{X}%
\rightarrow \underline{Y}$ and $\underline{\phi }^{\prime }:=\left( f^{\prime },\phi _{i'}^{\prime }\right)
:\underline{X}^{\prime }\rightarrow \underline{Y}^{\prime }$, their tensor
product is $\underline{\phi }\otimes \underline{\phi }^{\prime }=\left(
f\times f^{\prime },\phi _{i}\otimes
\phi _{i'}^{\prime }\right) $. The unit
object of $\Fam(\Cc)$ is $\underline{\unit}=\left( \left\{ \ast \right\} ,\unit%
\right) $, where $\unit$ is the unit object $\Cc$.

\begin{proposition}\label{pro:Famprerig}
Let $\Cc$ be a pre-rigid monoidal category. If $\Cc$ has products of pre-duals, then $\mathsf{%
Fam}(\Cc)$ is pre-rigid and the pre-dual of $\underline{Y}=(J,Y_j)$ is $%
\underline{Y}^{\ast }=( \left\{ \ast \right\} ,\prod\limits_{j\in
J}Y_{j}^{\ast }) .$

Conversely, if $\Cc$ has an initial object and $\mathsf{%
Fam}(\Cc)$ is pre-rigid, then  $\Cc$ has products of pre-duals.
\end{proposition}

\begin{proof}
Consider the diagonal functor $F:\Cc\rightarrow \Fam(%
\Cc):X\mapsto \left( \left\{ \ast \right\} ,X\right) \ $considered
in \cite[\S 3]{Benabou-Fibered} (where it is denoted by $\eta _{\Cc}$). Note that for every set $S$, there is a unique map $S\rightarrow \left\{
\ast \right\} $ that will be denoted by $t_{S}.$ The assignment $\left( t_{I}:I\rightarrow \left\{ \ast \right\} ,\phi
_{i}:X_{i}\rightarrow Y\right) \mapsto \left( \phi _{i}\right) _{i\in I}$
defines a bijection (natural in $Y$ but not in $\underline{X}$)
\begin{equation}\label{form:homFam}
\hom _{\Fam(\Cc)}\left( \underline{X},FY\right) \cong
\prod\limits_{i\in I}\hom _{\Cc}\left( X_{i},Y\right) .
\end{equation}%
Assume $\Cc$ has products of pre-duals. Then we can consider $\prod\limits_{j\in J}Y_{j}^{\ast }$ and we get a chain of bijections%
\begin{gather*}
\hom _{\Fam(\Cc)}\left( \underline{X},F(
\prod\limits_{j\in J}Y_{j}^{\ast }) \right) \overset{\eqref{form:homFam}}\cong \prod\limits_{i\in I}%
\hom _{\Cc}\left( X_{i},\prod\limits_{j\in J}Y_{j}^{\ast
}\right) \cong \prod\limits_{i\in I}\prod\limits_{j\in J}\hom _{%
\Cc}\left( X_{i},Y_{j}^{\ast }\right)  \\
\cong \prod\limits_{\left( i,j\right) \in I\times J}\hom _{\Cc%
}\left( X_{i}\otimes Y_{j},\unit\right) \overset{\eqref{form:homFam}}\cong \hom _{\Fam%
(\Cc)}\left( (I\times J,X_i\otimes Y_j),F\unit\right) =%
\hom _{\Fam(\Cc)}\left( \underline{X}\otimes
\underline{Y},\underline{\unit}\right) .
\end{gather*}%
Since the first isomorphism above is not natural in $\underline{X}$ we can not immediately conclude that the composition is natural and hence that $\Fam(\Cc)$ is pre-rigid. Nevertheless this composition maps $\left( t_{I}:I\rightarrow \left\{ \ast \right\} ,\phi
_{i}:X_{i}\rightarrow \prod\limits_{j\in J}Y_{j}^{\ast }\right) $ to $\left(
t_{I\times J}:I\times J\rightarrow \left\{ \ast \right\} ,\left( \ev%
_{Y_{j}}\circ \left( p_{j}\phi _{i}\otimes Y_{j}\right) \right) \right)$, where \ $p_{j}:\prod\limits_{j\in J}Y_{j}^{\ast }\rightarrow Y_{j}^{\ast }$
denotes the canonical projection. We set $\underline{Y}^{\ast
}:=F\left( \prod\limits_{j\in J}Y_{j}^{\ast }\right) .$ If we take $%
\underline{X}=\underline{Y}^{\ast }$ and apply the above composition to $%
\mathrm{Id}_{\underline{Y}^{\ast }}$, we get $\ev_{\underline{Y}%
}:=\left( t_{\left\{ \ast \right\} \times J},\ev_{Y_{j}}\circ \left(
p_{j}\otimes Y_{j}\right) \right) .$ By the following computation we see
that the composition above is $\left( t_{I},\phi _{i}\right) \mapsto \mathrm{%
ev}_{\underline{Y}}\circ \left( \left( t_{I},\phi _{i}\right) \otimes
{\underline{Y}}\right) .$
\begin{eqnarray*}
\ev_{\underline{Y}}\circ \left( \left( t_{I},\phi _{i}\right)
\otimes {\underline{Y}}\right)  &=&\left( t_{\left\{ \ast
\right\} \times J},\ev_{Y_{j}}\circ \left( p_{j}\otimes Y_{j}\right)
\right) \circ \left( \left( t_{I},\phi _{i}\right) \otimes \left( {J},%
{Y_{j}}\right) \right)  \\
&=&\left( t_{\left\{ \ast \right\} \times J},\ev_{Y_{j}}\circ \left(
p_{j}\otimes Y_{j}\right) \right) \circ \left( t_{I}\times {J},\phi
_{i}\otimes {Y_{j}}\right)  \\
&=&\left( t_{\left\{ \ast \right\} \times J}\circ \left( t_{I}\times
{J}\right) ,\ev_{Y_{j}}\circ \left( p_{j}\otimes Y_{j}\right)
\circ \left( \phi _{i}\otimes {Y_{j}}\right) \right)  \\
&=&\left( t_{I\times J},\ev_{Y_{j}}\circ \left( p_{j}\phi
_{i}\otimes Y_{j}\right) \right)
\end{eqnarray*}%
As a consequence $\Fam(\Cc)$ is pre-rigid and the pre-dual
of $\underline{Y}$ is $\underline{Y}^{\ast }:=F( \prod\limits_{j\in
J}Y_{j}^{\ast }) .$

Conversely, assume that $\Cc$ has an initial object $\initial$ and that $\mathsf{%
Fam}(\Cc)$ is pre-rigid. Let $\underline{Y}=(J,Y_j)$ be a family of objects in $\Cc$. By hypothesis this family has a pre-dual $\underline{Y}^*\in \Fam(\Cc)$. Thus, there is a set $S$ and an object $L_s$ for every $s\in S$ such that $\underline{Y}^*=(S,L_s)$. We have a chain of bijections
\begin{gather*}
S\cong\hom _{\mathsf{Set}}\left( \{*\},S \right) \cong
\hom _{\Fam(\Cc)}\left( ({*},\initial),(S,L_s) \right) \cong
\hom _{\Fam(\Cc)}\left( F(\initial),\underline{Y}^* \right) \cong
\hom _{\Fam(\Cc)}\left( F(\initial)\otimes \underline{Y},\underline{\unit} \right)\\=
\hom _{\Fam(\Cc)}\left( F(\initial)\otimes \underline{Y},F{\unit} \right)
\overset{\eqref{form:homFam}} \cong
\prod\limits_{j\in J}\hom _{\Cc}\left( \initial\otimes Y_j,\unit \right)\cong
\prod\limits_{j\in J}\hom _{\Cc}\left( \initial,Y_j^* \right)
\end{gather*}
and the latter is a singleton as $\initial$ is an initial object. Thus $S$ is a singleton and hence we can choose $\underline{Y}^*=F(L)$ for some $L\in \Cc$. Note that the evaluation $\ev_{\underline{Y}}:\underline{Y}^*\otimes \underline{Y}\to \underline{\unit}$ is of the form $(t_{\{*\}\times J},\ev_{j})$ for some morphism $\ev_{j}:L\otimes {Y_j}\to {\unit}$. Set $p_j:=(\ev_{j})^\dag:L\to Y_j^*$. Let us show that $(L,(p_j)_{j\in J})$ is the product of the family of pre-duals $(Y_j^*)_{j\in J}$. To this aim, consider the following chain of  bijections
\begin{gather*}
\hom _{\Cc}\left( X,L \right)\overset{\eqref{form:homFam}}\cong
\hom _{\Fam(\Cc)}\left( FX,FL \right) =
\hom _{\Fam(\Cc)}\left( FX,\underline{Y}^* \right)\cong
\hom _{\Fam(\Cc)}\left( FX\otimes \underline{Y},\underline{\unit} \right)\\\cong
\hom _{\Fam(\Cc)}\left( FX\otimes \underline{Y},F{\unit} \right)\overset{\eqref{form:homFam}}\cong
\prod\limits_{j\in J}\hom _{\Cc}\left( X\otimes Y_j,\unit \right)\cong
\prod\limits_{j\in J}\hom _{\Cc}\left( X ,Y_j^* \right).
\end{gather*}
A direct computation shows that this composition gives the bijection $$\hom _{\Cc}\left( X,L \right)\to \prod\limits_{j\in J}\hom _{\Cc}\left( X ,Y_j^* \right),f\mapsto \left((\ev_{j}\circ (f\otimes Y_j))^\dag\ \right)_{j\in J}.$$ Note that $\ev_{Y_j}\circ ((p_j\circ f)\otimes Y_j)=\ev_{Y_j}\circ (p_j\otimes Y_j)\circ (f\otimes Y_j)=\ev_{Y_j}\circ ((\ev_{j})^\dag\otimes Y_j)\circ (f\otimes Y_j)= \ev_{j}\circ (f\otimes Y_j)$ so that $(\ev_{j}\circ (f\otimes Y_j))^\dag=p_j\circ f$. Hence the above bijection maps $f$ to the family $\left(p_j\circ f\right)_{j\in J}$. This means that  $(L,(p_j)_{j\in J})$ is the product of the family $(Y_j^*)_{j\in J}$.
\end{proof}

\begin{remark}\label{Remark-predualsnecessary}
Consider any pre-rigid monoidal category $(\Cc,\otimes, \unit)$ with an initial object and suppose $\Fam(\Cc)$ is pre-rigid. By Proposition \ref{pro:Famprerig} we get $\Cc$ has products of pre-duals. In particular, for every set $S$, one has that $\prod_S^{\Cc}\unit^*$ exists, where $\prod^{\Cc}$ denotes the product in $\Cc$. Now, by Corollary \ref{cor:uniqueness}, one has $\unit^*\cong   \unit.$ Thus $\prod_S^{\Cc}\unit$ exists.

As an instance of this situation, consider a monoidal category $(\Aa,\otimes, \unit)$ and let $\Cc$ be $\Aa^f$, the full subcategory of $\Aa$ consisting of all rigid objects in $\Aa$, which is known to form a rigid monoidal category  $(\Aa^f,\otimes, \unit)$. By the foregoing, if  $\Aa^f$ has an initial object and $\prod_S^{\Aa^f}\unit$ does not exist, we can conclude that $\Fam(\Aa^f)$ is not pre-rigid. The following are examples of this situation.
\begin{enumerate}
  \item Putting $\Aa=\Vec$, $\Aa^f$ comes out to be $\Vec^{\textrm{f}}$. It has $0$ as a zero (whence initial) object. Note that, if $\prod_{\NN}^{\Vec^{f}}k$ existed, then we would have isomorphisms of vector spaces
    \begin{gather*}
  \prod_{\NN}^{\Vec^{f}}k\cong \hom _{\Vec^{f}}\left(k,\prod_{\NN}^{\Vec^{f}}k\right)\cong
     \prod_{\NN}^{\mathsf{Set}}\hom _{\Vec^{f}}(k,k)\cong
     \prod_{\NN}^{\mathsf{Set}}k\notin \Vec^{f};
    \end{gather*}
     a contradiction. It follows that $\Fam(\Vec^{\textrm{f}})$ is not pre-rigid.
  \item More generally consider a commutative ring $R$ and the monoidal category $(R\text{-}\mathsf{Mod},\otimes_R,R)$ of $R$-modules. It is well-known that $R\text{-}\mathsf{Mod}^f$ coincides with the category of finitely-generated projective $R$-modules (cf. \cite[Proposition 2.6]{De}). It has $0$ as a zero object. Moreover, by an argument similar to the above one, one checks that $\prod_{\NN}^{R\text{-}\mathsf{Mod}^f}R$ does not exist. It follows that $\Fam(R\text{-}\mathsf{Mod}^f)$ is not pre-rigid.
  \item Similarly, for $R$ a noetherian commutative ring, if we denote by $(\Aa:=\mathsf{Comod}\text{-}H,\otimes_R,R)$ the category of right $H$-comodules for a Hopf $R$-algebra $H$, then every object in $\Aa^f$ is finitely-generated  projective over $R$ (see \cite[Example]{Ulbrich}). Moreover $0$ is a zero object in $\Aa^f$. Let us check that $\prod_{\NN}^{\Aa^f}R$ does not exist. Suppose it does; then it is a finitely generated $R$-module, hence noetherian ($R$ being a noetherian ring). As a consequence its $R$-submodule $\left(\prod_{\NN}^{\Aa^f}R\right)^{\mathrm{co}H}$ of $H$-coinvariant elements must be finitely generated. This leads to the desired contradiction. In fact, since $R$ is a comodule via trivial coaction, we get
        \begin{gather*}
  \left(\prod_{\NN}^{\Aa^f}R\right)^{\mathrm{co}H}\cong \hom _{\Aa^f}\left(R,\prod_{\NN}^{\Aa^f}R\right)\cong
     \prod_{\NN}^{\mathsf{Set}}\hom _{\Aa^f}\left(R,R\right)\cong
     \prod_{\NN}^{\mathsf{Set}}R^{\mathrm{co}H}= \prod_{\NN}^{\mathsf{Set}}R.
    \end{gather*}
    Thus $\Fam(\Aa^f)$ is not pre-rigid.
\end{enumerate}
\end{remark}

So far we have dealt with the pre-rigidity of $\Fam(\Cc)$ but, to the best of our knowledge, it is even unknown whether $\Fam(\Cc)$ inherits closeness from $\Cc$ except in case $\Cc$ is cartesian which was considered in \cite[Lemma 4.1]{Carboni} and in \cite[Theorem 2.11]{AR}, the latter giving a complete characterization of cartesian closeness of $\Fam(\Cc):=\Sigma\Cc$. The following result fills this gap.

\begin{proposition}\label{pro:Famclosed}
Let $\Cc$ be a complete closed monoidal category with products. Then $\Fam(\Cc)$ is closed monoidal,
with $[(J,Y_j),( U,Z_u)]:=([J,U],[Y,Z] _{\alpha})$ where for each funtion $\alpha:J\to U$, we set $\left[ \underline{Y},\underline{Z}\right] _{\alpha
}:=\prod\limits_{j\in J}\left[ Y_{j},Z_{\alpha \left( j\right) }\right]$.
Given $\underline{c}=\left( q:U\rightarrow U^{\prime
},z_{u}:Z_{u}\rightarrow Z_{q\left( u\right) }^{\prime }\right) :\underline{Z%
}\rightarrow \underline{Z}^{\prime }$ we set $$\left[ \underline{Y},%
\underline{c}\right] :=\left( \left[ J,q\right] :\left[ J,U\right]
\rightarrow \left[ J,U^{\prime }\right] ,\prod\limits_{j\in J}\left[
Y_{j},z_{\alpha \left( j\right) }\right] :\left[ \underline{Y},\underline{Z}%
\right] _{\alpha }\rightarrow \left[ \underline{Y},\underline{Z}^{\prime }%
\right] _{q\circ \alpha }\right) :\left[ \underline{Y},\underline{Z}\right]
\rightarrow \left[ \underline{Y},\underline{Z}^{\prime }\right] .$$
\end{proposition}

\begin{proof}
The category $\mathsf{Set}$ is monoidal closed. In fact, we have a bijection
\begin{equation*}
\hom _{\mathsf{Set}}\left( I\times J,U\right) \cong \hom _{%
\mathsf{Set}}\left( I,\left[ J,U\right] \right)
\end{equation*}%
that assigns to a map $f:I\times J\rightarrow U$ the map $f^\dag%
:I\rightarrow \left[ J,U\right] $, where $f^\dag\left( i\right) \left(
j\right) =f\left( i,j\right) .$ Set $\ev_{J,U}:\left[ J,U\right]
\times J\rightarrow U,\left( \alpha ,j\right) \mapsto \alpha \left( j\right)
.$

As a consequence we have that the map $
\alpha  =\alpha _{\underline{X},\underline{Z}}:\hom _{\Fam(%
\Cc)}\left( \underline{X}\otimes \underline{Y},\underline{Z}\right)
\rightarrow \hom _{\Fam(\Cc)}\left( \underline{X},%
\left[ \underline{Y},\underline{Z}\right] \right)$

defined by the assignment
\begin{eqnarray*}
\left( f:I\times J\rightarrow U,\phi _{\left( i,j\right) }:X_{i}\otimes
Y_{j}\rightarrow Z_{f\left( i,j\right) }\right)  &\mapsto &\left( f^\dag%
:I\rightarrow \left[ J,U\right] ,\Delta \left( (\phi _{\left(
i,j\right) })^\dag\right) _{j\in J}:X_{i}\rightarrow \prod\limits_{j\in J}\left[
Y_{j},Z_{f\left( i,j\right) }\right] \right)
\end{eqnarray*}%
is invertible, where, given $h:X\otimes Y\rightarrow Z,$ we denoted by $h^\dag:X\rightarrow \left[ Y,Z\right] $ the unique morphism such that $\mathrm{ev%
}_{Y,Z}\circ \left( h^\dag\otimes Y\right) =h,$ (here $\ev%
_{Y,Z}:\left[ Y,Z\right] \otimes Y\rightarrow Z$). Its inverse $\beta =\beta
_{\underline{X},\underline{Z}}:\hom _{\Fam(\Cc%
)}\left( \underline{X},\left[ \underline{Y},\underline{Z}\right] \right)
\rightarrow \hom _{\Fam(\Cc)}\left( \underline{X}%
\otimes \underline{Y},\underline{Z}\right) $ maps $\left( g:I\rightarrow %
\left[ J,U\right] ,\psi _{i}:X_{i}\rightarrow \prod\limits_{j\in J}\left[
Y_{j},Z_{g\left( i\right) \left( j\right) }\right] \right) $ to $\left(
\ev_{J,U}\circ \left( g\times J\right) :I\times J\rightarrow U,%
\ev_{Y_{j},Z_{g\left( i\right) \left( j\right) }}\circ \left(
p_{j}\psi _{i}\otimes Y_{j}\right) :X_{i}\otimes Y_{j}\rightarrow Z_{g\left(
i\right) \left( j\right) }\right) $. In fact,
\begin{eqnarray*}
\beta \alpha \left( \left( f,\phi _{\left( i,j\right) }\right) \right)
&=&\beta \left( \left( f^\dag,\Delta \left( (\phi _{\left(
i,j\right) })^\dag\right) _{j\in J}\right) \right)  \\
&=&\left( \ev_{J,U}\circ \left( f^\dag\times J\right) ,\mathrm{%
ev}_{Y_{j},Z_{f\left( i,j\right) }}\circ \left( p_{j}\Delta \left( (
\phi _{\left( i,j\right) })^\dag\right) _{j\in J}\otimes Y_{j}\right) \right)  \\
&=&\left( f,\ev_{Y_{j},Z_{f\left( i,j\right) }}\circ \left( (
\phi _{\left( i,j\right) })^\dag\otimes Y_{j}\right) \right) =\left( f,\phi
_{\left( i,j\right) }\right) .
\end{eqnarray*}%
Conversely
\begin{eqnarray*}
\alpha \beta \left( \left( g,\psi _{i}\right) \right)  &=&\alpha \left(
\ev_{J,U}\circ \left( g\times J\right) ,\ev%
_{Y_{j},Z_{g\left( i\right) \left( j\right) }}\circ \left( p_{j}\psi
_{i}\otimes Y_{j}\right) \right)  \\
&=&\left( \left( \ev_{J,U}\circ \left( g\times J\right) \right) ^\dag,\Delta \left( \left( \ev_{Y_{j},Z_{g\left( i\right)
\left( j\right) }}\circ \left( p_{j}\psi _{i}\otimes Y_{j}\right) \right) ^\dag\right) _{j\in J}\right)  \\
&=&\left( g,\Delta \left( p_{j}\psi _{i}\right) _{j\in J}\right) =\left(
g,\psi _{i}\right) .
\end{eqnarray*}%
In order to check that $(-)\otimes \underline{Y}\dashv\left[ \underline{Y},-\right]$
it suffices to check the naturality of $\beta _{\underline{X},\underline{Z}}$.

Given $\underline{a}=\left( p:I^{\prime }\rightarrow I,x_{i^{\prime
}}:X_{i^{\prime }}^{\prime }\rightarrow X_{p\left( i^{\prime }\right)
}\right) :\underline{X}^{\prime }\rightarrow \underline{X},$ we have
\begin{eqnarray*}
&&\left( \beta _{\underline{X^{\prime }},\underline{Z^{\prime }}}\circ
\hom _{\Fam(\Cc)}\left( \underline{a},\left[
\underline{Y},\underline{c}\right] \right) \right) \left( g,\psi _{i}\right)
\\
&=&\beta _{\underline{X^{\prime }},\underline{Z^{\prime }}}\left( \left[
\underline{Y},\underline{c}\right] \circ \left( g,\psi _{i}\right) \circ
\underline{a}\right) \\
&=&\beta _{\underline{X^{\prime }},\underline{Z^{\prime }}}\left( \left( %
\left[ J,q\right] ,\prod\limits_{j\in J}\left[ Y_{j},z_{\alpha \left(
j\right) }\right] \right) \circ \left( g,\psi _{i}\right) \circ \left(
p,x_{i^{\prime }}\right) \right) \\
&=&\beta _{\underline{X^{\prime }},\underline{Z^{\prime }}}\left( \left( %
\left[ J,q\right] \circ g\circ p,\left( \prod\limits_{j\in J}\left[
Y_{j},z_{g\left( i\right) \left( j\right) }\right] \right) \circ \psi
_{p\left( i^{\prime }\right) }\circ x_{i^{\prime }}\right) \right) \\
&=&\left( \ev_{J,U^{\prime }}\circ \left( \left[ J,q\right] gp\times
J\right) ,\ev_{Y_{j},Z_{q\left( g\left( i\right) \left( j\right)
\right) }^{\prime }}\circ \left( p_{j}\left( \prod\limits_{j\in J}\left[
Y_{j},z_{g\left( i\right) \left( j\right) }\right] \right) \psi _{p\left(
i^{\prime }\right) }x_{i^{\prime }}\otimes Y_{j}\right) \right) \\
&=&\left( \ev_{J,U^{\prime }}\circ \left( \left[ J,q\right] gp\times
J\right) ,\ev_{Y_{j},Z_{q\left( g\left( i\right) \left( j\right)
\right) }^{\prime }}\circ \left( \left[ Y_{j},z_{g\left( i\right) \left(
j\right) }\right] p_{j}\psi _{p\left( i^{\prime }\right) }x_{i^{\prime
}}\otimes Y_{j}\right) \right) \\
&=&\left( q\circ \ev_{J,U}\circ \left( g\times J\right) \circ \left(
p\times J\right) ,z_{g\left( i\right) \left( j\right) }\circ \ev%
_{Y_{j},Z_{g\left( i\right) \left( j\right) }}\circ \left( p_{j}\psi
_{p\left( i^{\prime }\right) }\otimes Y_{j}\right) \circ \left( x_{i^{\prime
}}\otimes Y_{j}\right) \right) \\
&=&\left( q,z_{u}\right) \circ \left( \ev_{J,U}\circ \left( g\times
J\right) ,\ev_{Y_{j},Z_{g\left( i\right) \left( j\right) }}\circ
\left( p_{j}\psi _{i}\otimes Y_{j}\right) \right) \circ \left( p\times
J,x_{i^{\prime }}\otimes Y_{j}\right) \\
&=&\left( q,z_{u}\right) \circ \left( \ev_{J,U}\circ \left( g\times
J\right) ,\ev_{Y_{j},Z_{g\left( i\right) \left( j\right) }}\circ
\left( p_{j}\psi _{i}\otimes Y_{j}\right) \right) \circ \left( \left(
p,x_{i^{\prime }}\right) \otimes \underline{Y}\right) \\
&=&\underline{c}\circ \beta _{\underline{X},\underline{Z}}\left( g,\psi
_{i}\right) \circ \left( \underline{a}\otimes \underline{Y}\right)
=\left( \hom _{\Fam(\Cc)}\left( \underline{a}%
\otimes \underline{Y},\underline{c}\right) \circ \beta _{\underline{X},%
\underline{Z}}\right) \left( g,\psi _{i}\right) .
\end{eqnarray*}%
This shows that $\beta _{\underline{X},\underline{Z}}$ is natural so that $%
-\otimes \underline{Y}\dashv \left[ \underline{Y},-\right] $ and hence $%
\Fam(\Cc)$ is closed.
\end{proof}

Since a closed category is necessarily pre-rigid (Proposition \ref{prop:prerigid2}), the previous result ensures that, if $\Cc$ is a complete closed monoidal category with products, then $\Fam(\Cc)$ is in particular pre-rigid. Anyway these conditions on $\Cc$ are decidedly heavier than those assumed in Proposition \ref{pro:Famprerig}.

\subsection{The ``Turaev'' category \texorpdfstring{$\Maf (\Cc)$}{TEXT}}
Recall from \cite[Section 4]{GV-OnTheDuality} the definition of the ``Turaev'' category $\Maf (\Cc)$ as being $\Fam(\Cc^{\op })^{\mathrm{%
		op}}$. Objects in this category are the same as in $\Fam(\Cc)$, i.e. pairs $\underline{X}:=\left( X_{i}\right)
_{i\in I}=\left( I,X_{i}\right) $ where $I$ is a set and $X_{i}$ is an
object in $\Cc$ for all $i\in I$. A morphism between two objects $\underline{X}=\left(
I,X_{i}\right) $ and $\underline{Y}=\left( J,Y_{j}\right) $, however, is a pair $\underline{\phi }:=\left(
f,\phi _{j} \right):\underline{X}\rightarrow
\underline{Y}$ where $f:J\rightarrow I$ is a map and $\phi _{j
}:X_{f(j)}\rightarrow Y_{j}$ is a morphism in $\Cc$ for all $i \in I$.

If $\Cc$ is a monoidal category, so is $\Maf (\Cc)$, as
follows (see e.g. \cite[Section 4]{GV-OnTheDuality}). Given objects $\underline{X}$ and $\underline{Y}$ as above, their
tensor product is defined by $\underline{X}\otimes \underline{Y}=\left(
I\times J,X_{i}\otimes Y_{j}\right) .$ Given morphisms $\underline{\phi }%
:=\left( f,\phi _{j}\right) :\underline{X}%
\rightarrow \underline{Y}$ and $\underline{\phi }^{\prime }:=\left( f^{\prime },\phi _{j'}^{\prime }\right)
:\underline{X}^{\prime }\rightarrow \underline{Y}^{\prime }$, their tensor
product is $\underline{\phi }\otimes \underline{\phi }^{\prime }=\left(
f\times f^{\prime },\phi _{j}\otimes
\phi _{j'}^{\prime }\right) $. The unit
object of $\Maf (\Cc)$ is $\underline{\unit}=\left( \left\{ \ast \right\} ,\unit%
\right) $, where $\unit$ is the unit object $\Cc$.

The following result shows how the category $\Maf (\Cc)$ fails to be pre-rigid.

\begin{proposition}\label{pro:Maf}
Let $\Cc$ be a monoidal category. Then the category $\Maf(\Cc)$ is never pre-rigid.
\end{proposition}

\begin{proof}$\Maf (\Cc )$ has a terminal object given by the
empty family of objects $\terminal:=\left( \emptyset ,-\right) .$ Note
that, given any object $\underline{X}=\left( I,X\right) $, we have $%
\underline{X}\otimes \terminal=\left( I\times \emptyset ,-\right) \cong
\left( \emptyset ,-\right) =\terminal$. Now suppose that $\terminal$
has a pre-dual $\terminal^{\ast }$ in $\Maf (\Cc )$. Then
$
\mathrm{Id}_{\terminal^{\ast }}\in \mathrm{Hom}_{\Maf (\Cc %
)}\left( \terminal^{\ast },\terminal^{\ast }\right) \cong \mathrm{Hom%
}_{\Maf (\Cc )}\left( \terminal^{\ast }\otimes \terminal,\underline{\unit}\right) \cong \mathrm{Hom}_{\Maf (%
\Cc )}\left( \terminal,\underline{\unit}\right)
$
but there can be no morphism $\terminal\rightarrow \underline{\unit}$ as part of it would
be a map $\{*\} \rightarrow \emptyset $, a contradiction.
\end{proof}

\begin{remark}
Proposition \ref{pro:Maf} assures that $\Maf (\mathcal{\Vec})$ is not pre-rigid. Contraposition of Proposition \ref{prop:prerigid2} delivers that $\Maf (\mathcal{\Vec})$ is not even closed.
Similarly $\Maf (\mathcal{\Vec^{\textrm{f}}})$ is not pre-rigid, whence not even closed.
	\end{remark}
	
	\begin{remark}\label{TuraevZunino}
$\Vec$ can be given the obvious symmetric monoidal structure by considering the twist. It is shown in \cite[Section 2.1]{CaeDel} (resp. Section 2.2) that this induces a symmetric monoidal structure on the Turaev category $\Maf(\Vec)$ (resp. Zunino category $\Fam(\Vec)$). \cite[Proposition 2.5]{CaeDel} (resp. Proposition 2.10) asserts that Hopf group-coalgebras (resp. Hopf group-algebras), intruduced in \cite{Tu}, are precisely Hopf algebra objects in $\Maf(\Vec)$ (resp. $\Fam(\Vec)$) endowed with this symmetry. Notice that the definition of a Hopf group-coalgebra is not self-dual. The above results obtained so far in this section show that, with regard to pre-rigidity and closedness, the categorical places where Hopf group-coalgebras and Hopf group-algebras live behave differently as well.
\end{remark}

\subsection{The variant of the category of families \texorpdfstring{$\Faf(\Cc)$}{TEXT}}
We now turn to the study of the category $\Faf(\Cc)$ where the fact that pre-rigidity is inherited from $\Cc$ is restored, contrary to $\Maf(\Cc)$ above.

\begin{definition}[The category $\Faf(\Cc)$]
An object in $\Faf(\Cc)$, see \cite{LMM}, is a pair $\underline{X}:=\left( X_{i}\right)
_{i\in I}=\left( I,X_{i}\right) $ where $I$ is a set and $X_{i}$ is an
object in $\Cc.$ Given two objects $\underline{X}=\left(
I,X_{i}\right) $ and $\underline{Y}=\left( J,Y_{j}\right) $, a morphism $%
\underline{X}\rightarrow \underline{Y}$ is a set of triples $(i,j,f)$ where $i\in I,j\in J$ nd $f:X_i\to Y_j$ is a morphism in $\Cc$.
We can reorganize such a set of triples into a pair $\underline{\phi }:=\left(
\Rr ,\phi _{\left( i,j\right) }\right) :\underline{X}\rightarrow
\underline{Y}$ where $\Rr :I\relto J$ is a binary relation (see Example \ref{ex:Rel}), i.e.
a subset $\Rr \subseteq I\times J,$ while $\phi _{\left( i,j\right)
}\subseteq \hom _\Cc (X_{i}, Y_{j})$ for every $\left(
i,j\right) \in \Rr $. Given two morphisms $\underline{\phi }:=\left(
\Rr ,\phi _{\left( i,j\right) }\right) :(I,X_i)\rightarrow
(J,Y_j)$ and $\underline{\psi }:=\left(
\Ss ,\psi _{\left( j,k\right) }\right) :(J,Y_j)\rightarrow
(K,Z_k)$ their composition is defined to be $\underline{\psi}\circ \underline{\phi }:=\left(
\Ss\circ\Rr ,\bigcup_{j\in J_{(i,k)}}\psi _{(j,k)}\circ\phi _{(i,j)}\right)$, where $J_{(i,k)}:=\{j\in J\mid (i,j)\in\Rr, (j,k)\in\Ss \}$ and $\psi _{(j,k)}\circ\phi _{(i,j)}:=\{f\circ g\mid f\in\psi _{(j,k)},g\in\phi _{(i,j)}\}$.
The identity morphism is $(\id_I,\{\id_{X_i}\}):(I,X_i)\to (I,X_i)$.
\begin{invisible}
  The associativity of the composition $(I,X_i)\overset{(\Rr,\phi_{(i,j)})}\to (J,Y_j)\overset{(\Ss,\psi_{(j,k)})}\to(K,Z_k)\overset{(\mathcal{T},\lambda_{(k,t)})}\to (T,W_t)$ follows from
\begin{align*}
(\lambda_{(k,t)})_{(k,t)\in \mathcal{T}}\circ((\psi_{(j,k)})_{(j,k)\in \Ss}\circ (\phi_{(i,j)})_{(i,j)\in \Rr})
  &=(\lambda_{(k,t)})_{(k,t)\in \mathcal{T}}\circ(\bigcup_{j\in J_{(i,k)}}\psi _{(j,k)}\circ\phi _{(i,j)})_{(i,k)\in \Ss\circ\Rr}\\
   &=(\bigcup_{k\in K_{(i,t)}}\lambda_{(k,t)}\circ\bigcup_{j\in J_{(i,k)}}\psi _{(j,k)}\circ\phi _{(i,j)})_{(i,t)\in \mathcal{T}\circ\Ss\circ\Rr}\\
   &=(\bigcup_{k\in K_{(i,t)}}\bigcup_{j\in J_{(i,k)}}\lambda_{(k,t)}\circ\psi _{(j,k)}\circ\phi _{(i,j)})_{(i,t)\in \mathcal{T}\circ\Ss\circ\Rr}\\
   &=(\bigcup_{j\in J_{(i,t)}}\bigcup_{k\in K_{(j,t)}}\lambda_{(k,t)}\circ\psi _{(j,k)}\circ\phi _{(i,j)})_{(i,t)\in \mathcal{T}\circ\Ss\circ\Rr}\\
   &=(\bigcup_{k\in K_{(j,t)}}\lambda_{(k,t)}\circ\psi _{(j,k)})_{(j,t)\in\mathcal{T}\circ\Ss}\circ(\phi_{(i,j)})_{(i,j)\in \Rr}\\
   &=((\lambda_{(k,t)})_{(k,t)\in \mathcal{T}}\circ(\psi_{(j,k)})_{(j,k)\in \Ss})\circ (\phi_{(i,j)})_{(i,j)\in \Rr}.
\end{align*}
\end{invisible}

If $(\Cc,\otimes,\unit)$ is a monoidal category  so is $\Faf(\Cc)$ as
follows. Given objects $\underline{X}$ and $\underline{Y}$ as above, their
tensor product is defined by $\underline{X}\otimes \underline{Y}=\left(
I\times J,X_{i}\otimes Y_{j}\right) .$ Given morphisms $\underline{\phi }%
:=\left( \Rr ,\phi _{\left( i,j\right) }\right) :\underline{X}%
\rightarrow \underline{Y}$ and $\underline{\phi }^{\prime }:=\left( \mathcal{%
R}^{\prime },\phi _{\left( i^{\prime },j^{\prime }\right) }^{\prime }\right)
:\underline{X}^{\prime }\rightarrow \underline{Y}^{\prime }$, their tensor
product is $\underline{\phi }\otimes \underline{\phi }^{\prime }=\left(
\Rr \times \Rr ^{\prime },\phi _{\left( i,j\right) }\otimes
\phi _{\left( i^{\prime },j^{\prime }\right) }^{\prime }\right) $ where $%
\Rr \times \Rr ^{\prime }:=\left\{ \left( \left( i,i^{\prime
}\right) ,\left( j,j^{\prime }\right) \right) \mid \left( i,j\right) \in
\Rr \text{ and }\left( i^{\prime },j^{\prime }\right) \in \Rr %
^{\prime }\right\} $ is the cartesian product of binary relations and $\phi _{\left( i,j\right) }\otimes
\phi _{\left( i^{\prime },j^{\prime }\right) }^{\prime }:=\{f\otimes f'\mid f\in\phi _{\left( i,j\right) },f'\in
\phi _{\left( i^{\prime },j^{\prime }\right) }^{\prime }\}$. The unit
object is $\underline{\unit}=\left( \left\{ \ast \right\} ,\unit%
\right) .$
\end{definition}

\begin{remark}
 As observed in \cite{LMM}, if $\terminal$ denotes the terminal category, then $\Faf(\terminal)$ identifies with the category $\Rel$.
\end{remark}

\begin{proposition}\label{pro:Faf}
If $\Cc$ is a pre-rigid monoidal category, so is $\Faf(\Cc)$.
\end{proposition}

\begin{proof}
  Given an object $\underline{Y}=\left(
J,Y_{j}\right) \in \Faf(\Cc)$ we set $\underline{Y}^{\ast }:=\left( J,Y_{j}^{\ast
}\right) .$ Consider now the map%
\begin{eqnarray*}
\hom _{\Faf(\Cc)}\left( \underline{X},\underline{Y}%
^{\ast }\right) &\rightarrow &\hom _{\Faf(\Cc)}\left(
\underline{X}\otimes \underline{Y},\underline{\unit}\right) \\
\left( \Rr :I\relto J,\phi _{\left( i,j\right)
}\subseteq\hom _\Cc(X_{i}, Y_{j}^{\ast })\right) &\mapsto &\left( \ev%
_{J}\circ \left( \Rr \times \mathrm{Id}_{J}\right) ,\{\ev_{Y_{j}}\}\circ \left( \phi _{\left( i,j\right) }\otimes \{\mathrm{Id}%
_{Y_{j}}\}\right) \right)\subseteq \hom _\Cc(X_i\otimes Y_j,\unit).
\end{eqnarray*}%
where $\ev_{J}\circ \left( \Rr \times \mathrm{Id}_{J}\right) $
is the binary relation considered in Example \ref{ex:Rel}. Its inverse is
given by $\left( \Rr ,\phi _{\left( \left( i,j\right) ,\ast \right)
}\right) \rightarrow \left( \Rr^\dag,(\phi _{\left(
\left( i,j\right) ,\ast \right) })^\dag\right) $ where for every binary relation $%
\Rr :I\times J\relto \left\{ \ast \right\}$ we define $%
\Rr^\dag:I\relto J$ as in Example \ref{ex:Rel} while for
every $\phi _{\left( \left( i,j\right) ,\ast \right) }\subseteq \hom _\Cc(X_{i}\otimes
Y_{j},\unit)$, we set $(\phi _{\left(
\left( i,j\right) ,\ast \right) })^\dag:=\{f^\dag\mid f\in \phi _{\left( \left( i,j\right) ,\ast \right) }\}\subseteq\hom (X_{i},Y_{j}^{\ast })$ and for every morphism $f:X_{i}\otimes
Y_{j}\rightarrow \unit$ the morphism $f^\dag:X_{i}\rightarrow Y_{j}^{\ast }$ is the unique
morphism in $\Cc$ defined by $\ev_{Y_{j}}\circ \left(
f^\dag\otimes
Y_{j}\right) =f$ given by
the pre-rigid category of $\Cc$.

Define $\ev_{\underline{Y}}:\underline{Y}^{\ast }\otimes \underline{Y}%
\rightarrow \underline{\unit}$ by setting $\ev_{\underline{Y}%
}:=\left( \ev_{J},\{\ev_{Y_{j}}\}\right) .$ We compute%
\begin{eqnarray*}
\ev_{\underline{Y}}\circ \left( \left( \Rr ,\phi _{\left(
i,j\right) }\right) \otimes \mathrm{Id}_{\underline{Y}}\right) &=&\left(
\ev_{J},\{\ev_{Y_{j}}\}\right) \circ \left( \left( \Rr %
,\phi _{\left( i,j\right) }\right) \otimes \left( \mathrm{Id}_{J},\{\mathrm{Id}%
_{Y_{j}}\}\right) \right) \\
&=&\left( \ev_{J},\{\ev_{Y_{j}}\}\right) \circ \left( \Rr %
\times \mathrm{Id}_{J},\phi _{\left( i,j\right) }\otimes \{\mathrm{Id}%
_{Y_{j}}\}\right) \\
&=&\left( \ev_{J}\circ \left( \Rr \times \mathrm{Id}%
_{J}\right) ,\{\ev_{Y_{j}}\}\circ \left( \phi _{\left( i,j\right)
}\otimes \{\mathrm{Id}%
_{Y_{j}}\}\right) \right) .
\end{eqnarray*}%
Thus the bijection above is exactly $\ev_{\underline{Y}}\circ \left(
\left( \Rr ,\phi _{\left( i,j\right) }\right) \otimes \mathrm{Id}_{%
\underline{Y}}\right) $ and hence $\Faf(\Cc)$ is pre-rigid.
\end{proof}

\begin{remark}
  Note that $\Fam(\Cc)$ is a subcategory of $\Faf(%
\Cc)$ via the strong monoidal embedding (identity-on-object faithful functor)%
\begin{equation*}
\Fam(\Cc)\rightarrow \Faf(\Cc):\quad\left(
I,X_{i}\right) \mapsto \left( I,X_{i}\right) ,\quad\left( f,\phi _{i}\right)
\mapsto \left( f,\{\phi _{i}\}\right) .
\end{equation*}%
This induces a functor from the Turaev category
\begin{equation*}
\Maf (\Cc):=\Fam(\Cc^{\op })^{\mathrm{%
op}}\rightarrow \Faf(\Cc^{\op })^{\op }.
\end{equation*}%
Note also that we have a strong monoidal embedding
\begin{equation*}
\Maf (\Cc)\rightarrow \Faf(\Cc):\left(
I,X_{i}\right) \mapsto \left( I,X_{i}\right), \left( f:I\rightarrow J,\phi
_{i}:X_{f\left( i\right) }\rightarrow Y_{i}\right) \mapsto \left( f^{\sharp
}:J\rightarrow I,\{\phi _{i}\}\subseteq\hom _\Cc(X_{f\left( i\right) }, Y_{i})\right)
\end{equation*}%
where $\Rr ^{\sharp }=\left\{ \left( j,i\right) \mid \left( i,j\right)
\in \Rr \right\} :J\relto I$ is the converse relation of a
binary relation $\Rr :I\relto J$.
\end{remark}

Summing up, for a monoidal category $\Cc$,  both the monoidal categories $\Fam(\Cc)$ and $\Maf(\Cc)$ embeds in $\Faf(\Cc)$.
Moreover $\Faf(\Cc)$ inherits the pre-rigidity from $\Cc$ with no further assumption, $\Fam(\Cc)$ inherits the pre-rigidity if $\Cc$ has products of pre-duals while $\Maf(\Cc)$ does not.

\subsection{The functor category \texorpdfstring{$\left[ \mathcal{I},\Cc\right] $}{TEXT}}
Given a small category $\mathcal{I}$ and a complete closed monoidal category
$\Cc$, it is well-known that the functor category $\left[ \mathcal{I},\Cc\right] $ is closed as well, see e.g. \cite[Theorem B.14]%
{Aguiar-Mahajan-Bimonoids}.  Explicitly, given functors $F,G:\mathcal{I}\to \Cc$, for any object $I$ in $\mathcal{I}$, $[F,G](I)$ is defined to be the universal object in $\Cc$ with the following property: For any morphism $f:I\to X$ in $\mathcal{I}$ there is a morphism $\eta_f:[F,G](I)\to [F(X),G(X)]$ in $\Cc$ such that for any $g:X\to Y$ in $\mathcal{I}$ the following diagram commutes.
\begin{equation*}
\begin{aligned}
\xymatrixcolsep{2cm}\xymatrix{[F,G](I)\ar[r]^{\eta_f}\ar[d]_-{\eta_{g\circ f }}&[F(X),G(X)]\ar[d]^-{[F(X),G(g)]} \\
[F(Y),G(Y)] \ar[r]^{[F(g),G(X)]} & [F(X),G(Y))]
}
\end{aligned}
\end{equation*} It has also the following description as an end of a functor, see e.g. \cite[(B.22)]%
{Aguiar-Mahajan-Bimonoids}.
\begin{equation}\label{endhomfuncat}
[F,G](I)=\int_{J\in \mathcal{I}}\prod_{\hom _{\mathcal{I}%
}(I,J)}[F(J),G(J)].
\end{equation}

 Our next aim is to show that a similar result
holds in case $\Cc$ is just pre-rigid.

\begin{proposition}\label{pro:catfun}
Let $\mathcal{I}$ be a small category and let $\Cc$ be a complete
monoidal category. If $\Cc$ is pre-rigid, so is the functor category $%
\left[ \mathcal{I},\Cc\right] .$
\end{proposition}

\begin{proof}
By e.g. \cite[Exercice 4, page 165]{MacLane}, we know that $\left[ \mathcal{I},%
\Cc\right] $ is monoidal where, for any functors $T,F:\mathcal{I}%
\rightarrow \Cc$, we have $\left( T\otimes F\right) \left( x\right)
=T\left( x\right) \otimes F\left( x\right) $ and the unit object of $%
\Cc^{\mathcal{I}}$ is the constant functor $\unit^{\prime }:%
\mathcal{I}\rightarrow \Cc$ on the unit object $\unit\in
\Cc$.
 Consider a functor $F:\mathcal{I}\rightarrow \Cc$. Define $%
S\left( x,y\right) :=\prod_{\hom _{\mathcal{I}}(x,y)}F(y)^{\ast }$
and denote by $p_{g}:S\left( x,y\right) \rightarrow F(y)^{\ast }$ the
canonical projection for every $g\in \hom _{\mathcal{I}}(x,y)$. Given
morphisms $u :x_{1}\rightarrow x_{2}$ and $v :y_{2}\rightarrow
y_{1}, $ there is a unique morphism $S\left( u,v\right) :S\left(
x_1,y_{1}\right) \rightarrow S\left( x_2,y_{2}\right) $ such that the following
diagram commutes for every $g:x_{2}\rightarrow y_{2}$
\begin{equation}\label{def:Suv}
\begin{aligned}
\xymatrixcolsep{1.5cm}\xymatrix{S\left( x_{1},y_{1}\right)\ar@{.>}[r]^{S\left( u ,v \right)}\ar[d]_-{p_{v \circ g\circ u }}&S\left( x_{2},y_{2}\right)\ar[d]^-{p_{g}} \\
F(y_{1})^{\ast } \ar[r]^{F(v )^{\ast }} & F(y_{2})^{\ast }
}
\end{aligned}
\end{equation}%
In this way we have defined a functor $S:\mathcal{I}\times \mathcal{I}^{^{%
\op }}\rightarrow \Cc:\left( x,y^\op \right)
\mapsto S\left( x,y\right) .$

Denote by $F^{\ast }(x):=\underleftarrow{\lim }_{y\in \mathcal{I}}\prod_{%
\hom _{\mathcal{I}}(x,y)}F(y)^{\ast }:=\underleftarrow{\lim }S\left(
x,-\right) $ i.e. the limit of the functor $S\left( x,-\right) :\mathcal{I}%
^\op \rightarrow \Cc:y^\op \mapsto S\left(
x,y\right) =\prod_{\hom _{\mathcal{I}}(x,y)}F(y)^{\ast }.$ Given $%
u :x_{1}\rightarrow x_{2}$ in $\mathcal{I}$ we set $F^{\ast }(u):=%
\underleftarrow{\lim }S\left(u,-\right) :\underleftarrow{\lim }%
S\left( x_{1},-\right) \rightarrow \underleftarrow{\lim }S\left(
x_{2},-\right) .$ This defines a functor $F^{\ast }=\underleftarrow{\lim }%
_{y\in \mathcal{I}}\prod_{\hom _{\mathcal{I}}(-,y)}F(y)^{\ast }.$ Let
us check that $F^{\ast }$ is a pre-dual of and construct explicitly an
isomorphism
\begin{equation*}
\mathrm{Nat}(T\otimes F,\unit^{\prime })\cong \mathrm{Nat}\left(
T,F^{\ast }\right)
\end{equation*}%
for any functors $T,F:\mathcal{I}\rightarrow \Cc$.

Given $\alpha :T\otimes F\rightarrow \unit^{\prime }$, its components
are morphisms $\alpha _{y}:T(y)\otimes F(y)\rightarrow \unit$ with $%
y\in \mathcal{I}.$ Since $\Cc$ is pre-rigid we can consider $%
(\alpha _{y})^\dag:T(y)\rightarrow F(y)^{\ast }\ $where $F(y)^{\ast }$
denotes the pre-dual of $F(y).$ For every $x\in \mathcal{I}$, there is a
unique morphism $\alpha _{y}^{x}:T(x)\rightarrow S\left( x,y\right) $ such
that, for every $g:x\rightarrow y$ in $\mathcal{I}$, we have
\begin{equation}\label{def:alphaxy}
\begin{aligned}
\xymatrixcolsep{1.5cm}\xymatrix{T\left( x\right)\ar@{.>}[r]^{\alpha _{y}^{x}}\ar[d]_-{T\left( g\right)}&S\left(
x,y\right)\ar[d]^-{p_{g}} \\
T\left( y\right) \ar[r]^{(\alpha _{y})^\dag} & F(y)^{\ast }
}
\end{aligned}
\end{equation}%
Given morphisms $u :x_{1}\rightarrow x_{2}$ and $v :y_{2}\rightarrow
y_{1},$ for every $g:x_{2}\rightarrow y_{2}$ we have
\begin{eqnarray*}
p_{g}\circ S\left( u ,v \right) \circ \alpha _{y_{1}}^{x_{1}}
&\overset{\eqref{def:Suv}}=&F(v )^{\ast }\circ p_{v \circ g\circ u }\circ \alpha
_{y_{1}}^{x_{1}}\overset{\eqref{def:alphaxy}}=F(v )^{\ast }\circ (\alpha _{y_{1}})^\dag\circ
T\left( v\circ g\circ u \right) \\
&=&F(v )^{\ast }\circ (\alpha _{y_{1}})^\dag\circ T\left( v
\right) \circ T\left( g\right) \circ T\left( u \right) \overset{\left( %
\ref{form:alphahat}\right) }{=}(\alpha _{y_{2}})^\dag\circ T\left(
g\right) \circ T\left( u \right)
\overset{\eqref{def:alphaxy}}=p_{g}\circ \alpha _{y_{2}}^{x_{2}}\circ T\left( u \right)
\end{eqnarray*}%
where we applied the following formula
\begin{equation}
F(v )^{\ast }\circ (\alpha _{y_{1}})^\dag\circ T\left(v \right) =%
(\alpha _{y_{2}})^\dag,\text{ for every }v :y_{2}\rightarrow y_{1},  \label{form:alphahat}
\end{equation}%
that can be proved as follows. The naturality of $\alpha $ tells $\alpha
_{y_{1}}\circ \left( T\left( v \right) \otimes F\left( v \right)
\right) =\alpha _{y_{2}}$ so that%
\begin{align*}
\ev_{F(y_{2})}&\circ \left( F(v )^{\ast }\otimes
F(y_{2})\right) \circ \left( (\alpha _{y_{1}})^\dag\otimes
F(y_{2})\right) \circ \left( T\left( v \right) \otimes F(y_{2})\right) \\
&\overset{\left( \ref{form:dual}\right) }{=}\ev_{F(y_{1})}\circ
\left( F(y_{1})^{\ast }\otimes F(v )\right) \circ \left( (\alpha
_{y_{1}})^\dag\otimes F(y_{2})\right) \circ \left( T\left( v \right) \otimes
F(y_{2})\right) \\
&=\ev_{F(y_{1})}\circ \left( (\alpha _{y_{1}})^\dag\otimes
F(y_{1})\right) \circ \left( T(y_{1})\otimes F(v )\right) \circ \left(
T\left( v \right) \otimes F(y_{2})\right) =\alpha _{y_{1}}\circ \left( T\left( v \right) \otimes F\left( v
\right) \right) \overset{\text{nat.}\alpha}=\alpha _{y_{2}}
\end{align*}%
which means that $F(v)^{\ast }\circ (\alpha _{y_{1}})^\dag\circ T\left( v
\right) =(\alpha _{y_{2}})^\dag.$ Thus $p_{g}\circ S\left( u ,v
\right) \circ \alpha ^{ x_{1}} _{y_{1}}=p_{g}\circ \alpha ^{
x_{2}} _{y_{2}}\circ T\left( u \right) $ and hence
\begin{equation}
S\left( u ,v \right) \circ \alpha _{y_{1}}^{x_{1}}=\alpha
_{y_{2}}^{x_{2}}\circ T\left( u\right) ,\text{ for every }u
:x_{1}\rightarrow x_{2}\text{ and }v :y_{2}\rightarrow y_{1}.
\label{form:S}
\end{equation}%
In particular, taking $u =1_{x},$ we obtain $S\left( x,v
\right) \circ \alpha _{y_{1}}^{x}=\alpha _{y_{2}}^{x}$ for all $v
:y_{2}\rightarrow y_{1}.$ This means that $\left( T(x),\alpha
_{y}^{x}:T(x)\rightarrow S\left( x,y\right) \right) _{y\in \mathcal{I}}$ is
a cone for the functor $S\left( x,-\right) :\mathcal{I}^{^{\op %
}}\rightarrow \Cc:y^\op \mapsto S\left( x,y\right) $ and
hence it defines a unique morphism $(\alpha^\dag )_{x}:T(x)\rightarrow
F^{\ast }(x):=\underleftarrow{\lim }_{y\in \mathcal{I}}S\left( x,y\right) $
such that $q^x_{y}\circ (\alpha^\dag)_{x}=\alpha _{y}^{x},$ where $%
q^x_{y}:F^{\ast }(x)\rightarrow S\left( x,y\right) $ is the canonical map
defining the limit.
\begin{equation}\label{def:alphahatx}
\begin{aligned}
\xymatrixcolsep{1.5cm}\xymatrix{T\left( x\right)\ar@{.>}[rd]_{(\alpha^\dag )_{x}}\ar[rr]^{\alpha _{y}^{x}}&&S\left(
x,y\right) \\
&F^*(x)\ar[ru]_{q^x_{y}}
}
\end{aligned}
\end{equation}%

Let us check it is natural in $x.$ Given $u :x_1\rightarrow x_2$
 in $\mathcal{I}$,  we have
\begin{equation*}
q^{x_2}_{y}\circ F^{\ast }(u)\circ(\alpha^\dag)_{x_1}\overset{\text{def.} F^*}=S\left(u
,y\right) \circ q^{x_1}_{y}\circ  (\alpha^\dag) _{x_1}\overset{\eqref{def:alphahatx}}=S\left(
u ,y\right) \circ \alpha _{y}^{x_1}\overset{\eqref{form:S}}=\alpha _{y}^{x_2}\circ
T(u )\overset{\eqref{def:alphahatx}}=q^{x_2}_{y}\circ (\alpha^\dag) _{x_2}\circ
T(u ).
\end{equation*}%
Thus $F^{\ast }(u )\circ (\alpha^\dag)_{x_1}=(\alpha^\dag)%
_{x_2}\circ T(u)$ i.e. $(\alpha^\dag)_{x}$ is natural in $%
x $ and it defines $\alpha^\dag:T\rightarrow F^{\ast }.$ This way
we get
\begin{equation*}
\Phi :\mathrm{Nat}(T\otimes F,\unit^{\prime })\rightarrow \mathrm{Nat}%
\left( T,F^{\ast }\right) :\alpha \rightarrow \alpha^\dag.
\end{equation*}%
We have to check it is invertible and that its inverse is the one arising
from evaluation. Let us first define this evaluation.

We have to construct a natural transformation $\ev_{F}:F^{\ast
}\otimes F\rightarrow \unit^{\prime }.$ We define it on the component $%
x $ as follows%
\begin{equation*}
F^{\ast }(x)\otimes F(x)\overset{q^{x}_{x}\otimes F(x)}{\longrightarrow }S\left(
x,x\right) \otimes F\left( x\right) \overset{p_{\mathrm{Id}}\otimes F(x)}{%
\longrightarrow }F\left( x\right) ^{\ast }\otimes F\left( x\right) \overset{%
\ev_{F\left( x\right) }}{\longrightarrow }\unit
\end{equation*}%
so that $\left( \ev_{F}\right) _{x}:=\ev_{F\left( x\right)
}\circ \left( p_{\mathrm{Id}}q^{x}_{x}\otimes F(x)\right) .$ The naturality
follows from the following computation performed for every $f:x\rightarrow y$%
.
\begin{align*}
\left( \ev_{F}\right) _{y}\circ &\left( F^{\ast }\otimes F\right)
\left( f\right) =\ev_{F\left( y\right) }\circ \left( p_{\mathrm{Id}%
}q^{y}_{y}\otimes F(y)\right) \circ \left( F^{\ast }\left( f\right) \otimes
F\left( f\right) \right) =\ev_{F\left( y\right) }\circ \left( p_{%
\mathrm{Id}}q^{y}_{y}F^{\ast }\left( f\right) \otimes F(f)\right) \\
&\overset{\text{def.} F^*}=\ev_{F\left( y\right) }\circ \left( p_{\mathrm{Id}}S\left(
f,y\right) q^{x}_{y}\otimes F(f)\right)\overset{\eqref{def:Suv}} =\ev_{F\left( y\right) }\circ
\left( F(\mathrm{Id}_{y})^{\ast }p_{f}q^{x}_{y}\otimes F(f)\right) =\ev%
_{F\left( y\right) }\circ \left( p_{f}q^{x}_{y}\otimes F(f)\right) \\
&=\ev_{F\left( y\right) }\circ \left( F(y)^{\ast }\otimes
F(f)\right) \circ \left( p_{f}q^{x}_{y}\otimes F(x)\right) \overset{\left( \ref%
{form:dual}\right) }{=}\ev_{F\left( x\right) }\circ \left(
F(f)^{\ast }\otimes F(x)\right) \circ \left( p_{f}q^x_{y}\otimes F(x)\right) \\
&=\ev_{F\left( x\right) }\circ \left( F(f)^{\ast }p_{f}q^{x}_{y}\otimes
F(x)\right) \overset{\eqref{def:Suv}}=\ev_{F\left( x\right) }\circ \left( p_{\mathrm{Id}%
}S\left( x,f\right) q^{x}_{y}\otimes F(x)\right) \\
&\overset{q^{x}_{y}\text{ cocone}}{=}\ev_{F\left( x\right) }\circ
\left( p_{\mathrm{Id}}q^{x}_{x}\otimes F(x)\right) =\left( \ev_{F}\right) _{x}=\unit^{\prime }\left( f\right)
\circ \left( \ev_{F}\right) _{x}.
\end{align*}

Define now%
\begin{equation*}
\Psi :\mathrm{Nat}\left( T,F^{\ast }\right) \rightarrow \mathrm{Nat}%
(T\otimes F,\unit^{\prime }):\lambda \rightarrow \ev_{F}\circ
\left( \lambda \otimes F\right)
\end{equation*}%
For $\alpha :=\Psi \left( \lambda \right) =\ev_{F}\circ \left(
\lambda \otimes F\right) $, we have
\begin{equation}\label{form:pslmbdht}(\alpha _{y})^\dag
=\left( \left( \ev%
_{F}\right) _{y}\circ \left( \lambda _{y}\otimes F\left( y\right) \right)
\right) ^\dag
=\left( \ev_{F\left( y\right) }\circ \left( p_{\mathrm{Id}%
}q^y_{y}\lambda _{y}\otimes F(y)\right) \right) ^\dag=p_{\mathrm{Id}%
}\circ q^y_{y}\circ\lambda _{y}
\end{equation}
and hence, for every $g:x\to y$,
\begin{align*}
p_{g}\circ q^{x}_{y}\circ \Phi \left( \Psi \left( \lambda \right) \right) _{x}
&=p_{g}\circ q^{x}_{y}\circ (\alpha^\dag) _{x} \overset{\eqref{def:alphahatx}}=p_{g}\circ \alpha ^{x}_{y} \overset{\eqref{def:alphaxy}}=%
(\alpha _{y})^\dag\circ T\left( g\right)
\overset{\eqref{form:pslmbdht}}=p_{\mathrm{Id}%
}\circ q^y_{y}\circ\lambda _{y}\circ T\left( g\right)
\\
&\overset{\text{nat.}\lambda}=p_{\mathrm{Id}%
}\circ q^y_{y}\circ F^*\left( g\right)\circ\lambda _{x}\overset{\text{def.} F^*}=p_{\mathrm{Id}%
}\circ S\left( g,y\right)\circ q^x_{y}\circ \lambda _{x}\overset{\eqref{def:Suv}}=p_{g%
}\circ q^x_{y}\circ \lambda _{x}
\end{align*}%
so that $\Phi \left( \Psi \left( \lambda \right) \right) =\lambda .$
Conversely
\begin{eqnarray*}
\Psi \left( \Phi \left( \alpha \right) \right) _{x} &=&\left( \ev%
_{F}\circ \left( \Phi \left( \alpha \right) \otimes F\right) \right)
_{x}=\left( \ev_{F}\right) _{x}\circ \left( (\alpha^\dag)%
_{x}\otimes F\left( x\right) \right) =\ev_{F\left( x\right) }\circ
\left( p_{\mathrm{Id}}q_{x}\otimes F(x)\right) \circ \left( (\alpha^\dag)%
_{x}\otimes F\left( x\right) \right) \\
&=&\ev_{F\left( x\right) }\circ \left( p_{\mathrm{Id}}q^x_{x}(\alpha^\dag)_{x}\otimes F(x)\right) \overset{\eqref{def:alphahatx}}{=}%
\ev_{F\left( x\right) }\circ \left( p_{\mathrm{Id}}\alpha
_{x}^{x}\otimes F(x)\right) \overset{\eqref{def:alphaxy}}{=}\mathrm{ev%
}_{F\left( x\right) }\circ \left((\alpha _{x})^\dag\otimes F(x)\right)
=\alpha _{x}
\end{eqnarray*}%
so that $\Psi \left( \Phi \left( \alpha \right) \right) =\alpha $. As a
consequence $\Psi $ is bijective and hence $\Cc^{\mathcal{I}}$ is
pre-rigid.
\end{proof}

\begin{remark}
For those who are familiar with the language of ends, we sketch here a
different approach to the proof of the previous result; it can be seen as
an adaptation of \cite{Shulman-mathoverflow}. Consider the same
functor $S:\mathcal{I}\times \mathcal{I}^\op \rightarrow
\Cc:\left( x,y^\op \right) \mapsto S\left( x,y\right) .$
Define now the functor $S^{\prime }\left( x\right) :\mathcal{I}^{^{\mathrm{op%
}}}\times \mathcal{I}\rightarrow \Cc:\left( y^{^{\op %
}},z\right) \mapsto S\left( x,y\right) $ which is constant in $z$. By \cite[%
Corollary 2, page 224]{MacLane}, the end of $S^{\prime }\left( x\right) $
exists and coincide with the limit of the functor $S\left( x,-\right) :%
\mathcal{I}^\op \rightarrow \Cc:y^{^{\op %
}}\mapsto S\left( x,y\right) \ $i.e. with $F^{\ast }(x)$ (left-hand version
of \cite[Proposition 3, page 225]{MacLane} applied $S^{\prime }\left(
x\right) $ represented as the composition $\mathcal{I}^{^{\op %
}}\times \mathcal{I}\overset{Q}{\rightarrow }\mathcal{I}^\op %
\overset{S\left( x,-\right) }{\rightarrow }\Cc$ where $Q$ is the
first projection). As a consequence we can write
\begin{equation*}
F^{\ast }(x)=\int_{y\in \mathcal{I}}\prod_{\hom _{\mathcal{I}%
}(x,y)}F(y)^{\ast }
\end{equation*}%
(note that this description agrees with \eqref{endhomfuncat} in case $\Cc$ is closed, once we take $G=\unit^{\prime }$).

We compute
\begin{align*}
\mathrm{Nat}(T\otimes F,\unit^{\prime })&\overset{(a)}{\cong }\int_{y\in
\mathcal{I}}\hom _{\Cc}\left( (T\otimes F)(y),\unit%
^{\prime }(y)\right) =\int_{y\in \mathcal{I}}\hom _{\Cc%
}\left( T(y)\otimes F(y),\unit\right) \\
\cong& \int_{y\in \mathcal{I}}\hom _{\Cc}\left(
T(y),F(y)^{\ast }\right) \overset{(b)}{\cong }\int_{y\in \mathcal{I}%
}\int_{x\in \mathcal{I}}\hom _{\Cc}\left( T(x),\prod_{\mathrm{%
Hom}_{\mathcal{I}}(x,y)}F(y)^{\ast }\right) \\
\overset{(c)}{\cong }&\int_{x\in \mathcal{I}}\int_{y\in \mathcal{I}}\mathrm{%
Hom}_{\Cc}\left( T(x),\prod_{\hom _{\mathcal{I}%
}(x,y)}F(y)^{\ast }\right) \\
\overset{(d)}{\cong }&\int_{x\in \mathcal{I}}\hom _{\Cc}\left(
T(x),\int_{y\in \mathcal{I}}\prod_{\hom _{\mathcal{I}%
}(x,y)}F(y)^{\ast }\right)\\
\cong &\int_{x\in \mathcal{I}}\hom _{%
\Cc}\left( T(x),F^{\ast }(x)\right) \overset{(a)}{\cong }\mathrm{Nat}%
\left( T,F^{\ast }\right)
\end{align*}%
where in $(a)$ we used \cite[(2) on page 223]{MacLane}, in $(c)$ the Fubini
rule for ends \cite[page 231]{MacLane}, in $(d)$ we used \cite[(4) on page 225]{MacLane} and in $\left( b\right) $ we applied for $C=F(y)^{\ast
}$ the isomorphism $\hom _{\Cc}\left( T(y),C\right) \cong
\int_{x\in I}\hom _{\Cc}\left( T(x),\prod_{\hom _{%
\mathcal{I}}(x,y)}C\right) \ $that can be achieved by the following
argument. Given a functor $G:\mathcal{J}\rightarrow \mathsf{Set}$, by Yoneda
Lemma one has, for $y\in \mathcal{J}$%
\begin{equation*}
G\left( y\right) \cong \mathrm{Nat}(\hom _{\mathcal{J}}\left(
y,-\right) ,G)\overset{(a)}{\cong }\int_{x\in \mathcal{J}}\hom _{%
\mathsf{Set}}(\hom _{\mathcal{J}}\left( y,x\right) ,G\left( x\right)
)=\int_{x\in \mathcal{J}}\prod_{\hom _{\mathcal{J}}(y,x)}G\left( x\right) .%
\text{ }
\end{equation*}%
Note that, by \cite[Formula (3), page 242]{MacLane}, the last term coincides
with the right Kan extension $\mathrm{Ran}_{\mathrm{Id}_{\mathcal{J}}}G$ of $G$ along $%
\mathrm{Id}_{\mathcal{J}}$. The above isomorphism can then be seen as a consequence of
the fact that $\mathrm{Ran}_{K}G$, for some functor $K,$ is uniquely
determined by $\mathrm{Nat}(T,\mathrm{Ran}_{K}G)\cong \mathrm{Nat}(T\circ
K,G)$ and by the trivial equality $\mathrm{Nat}(T,G)=\mathrm{Nat}(T\circ
\mathrm{Id}_{\mathcal{J}},G)$.

In case $\mathcal{J}=\mathcal{I}^\op $ and $G:=\hom _{%
\Cc}\left( T(-),C\right) :\mathcal{I}^\op \rightarrow
\mathsf{Set}:x^\op \mapsto \hom _{\Cc}\left(
T(x),C\right) ,$ we get
\begin{eqnarray*}
\hom _{\Cc}\left( T(y),C\right) &=&\int_{x^{^{\op %
}}\in \mathcal{I}^\op }\prod_{\hom _{\mathcal{I}^{^{%
\op }}}\left( y^\op ,x^\op \right) }\mathrm{%
Hom}_{\Cc}\left( T(x),C\right) \\
&=&\int_{x\in \mathcal{I}}\prod_{\hom _{\mathcal{I}}(x,y)}\hom %
_{\Cc}\left( T(x),C\right) \cong \int_{x\in \mathcal{I}}\hom _{\mathcal{%
C}}\left( T(x),\prod_{\hom _{\mathcal{I}}(x,y)}C\right) .
\end{eqnarray*}%
\end{remark}

\subsection{The category of \texorpdfstring{$G$}{TEXT}-graded vector spaces \texorpdfstring{$\Mm^G$}{TEXT}}
Here we consider the construction of the category of externally $G$-graded $\Mm$-objects where $G$ is a monoid and $\Mm$ is a given monoidal category. As we will se below, this will allow us to provide a non-trivial example of a pre-rigid monoidal category which is not right closed, see Example \ref{examplerightclosed}.

\begin{claim}\label{claim:MG}Let $\left( \Mm,\otimes ,\unit\right) $ be a monoidal category
and let $G$ be a monoid with neutral element $e$. Assume that $\Mm$
has an initial object $\initial$ and coproducts indexed by $S_{g}:=\left\{
\left( a,b\right) \in G\times G\mid ab=g\right\} $ for every $g\in G$ and
that $\otimes $ preserves them. Then we can consider the monoidal category $\Mm%
^{G}$ of externally $G$-graded $\Mm$-objects, see e.g. \cite[Section 3]{Mitchell-Low}. Recall that an object
in $\Mm^{G}$ is a sequence $(X_{g})_{g\in G}$ of objects in $%
\Mm$ and a morphism is a sequence $(f_{g})_{g\in G}$ of morphisms in
$\Mm$. We can define the tensor product $X\otimes Y$ in $\Mm^{G}$ of $%
X=(X_{g})_{g\in G}$ and $Y=(Y_{g})_{g\in G}$ by the rule  $$\left( X\otimes Y\right)
_{g}:=\oplus _{(a,b)\in S_g}X_{a}\otimes Y_{b}=\oplus _{ab=g}X_{a}\otimes Y_{b},$$ and the unit by $\unit%
^{G}:=\left( \delta _{g,e}\unit\right) _{g\in G}$ where $\delta _{g,e}\unit=\unit$ if $%
g=e$ and $\delta _{g,e}\unit=\initial$ otherwise.
\end{claim}

In the case when $\Mm$ is the category $\Vec$ of vectors spaces, the category $\Vec^G$ is monoidally equivalent to the category $\Vec_G$ of $G$-graded vector spaces through the functor  $F:\Vec^G\to\Vec_G$ which maps an object $(V_g)_{g\in G}$ to $\oplus_{g\in G}V_g$ and a morphism $(f_g)_{g\in G}$ to the morphism $\oplus_{g\in G}f_g$. We already observed that the monoidal category $\Vec_G$ is closed in Example \ref{ex:closed}. The following result shows to what extend the same property is true for $\Mm^G$.

\begin{proposition}
\label{pro:funcatClosed}In the setting of \ref{claim:MG}, assume further
that $\mathcal{M}$ has products indexed by
$G$. If $\mathcal{M}$ is closed so is the category $\mathcal{M}^{G}$ where $%
[V,W]$ is defined by $[V,W]_{g}:=\prod_{h\in G}[V_{h},W_{gh}]$ for every
objects $V$ and $W$ in $\mathcal{M}^{G}$.
\end{proposition}

\begin{proof}
Once recalled that $S_{g}=\left\{ \left( a,b\right) \in G\times G\mid
ab=g\right\} ,$ the conclusion comes from the following chain of bijections
whose composition is natural in $X$ and $Z$.
\begin{gather*}
\mathrm{Hom}_{\mathcal{M}^{G}}\left( X\otimes Y,Z\right) =\prod_{g\in G}%
\mathrm{Hom}_{\mathcal{M}}\left( \oplus _{\left( a,b\right) \in
S_{g}}X_{a}\otimes Y_{b},Z_{g}\right)  \\
\cong \prod_{g\in G}\prod_{\left( a,b\right) \in S_{g}}\mathrm{Hom}_{%
\mathcal{M}}\left( X_{a}\otimes Y_{b},Z_{g}\right) \cong \prod_{g\in
G}\prod_{\left( a,b\right) \in S_{g}}\mathrm{Hom}_{\mathcal{M}}\left(
X_{a},[Y_{b},Z_{g}]\right)  \\
\cong \prod_{a\in G}\prod_{b\in G}\mathrm{Hom}_{\mathcal{M}}\left(
X_{a},[Y_{b},Z_{ab}]\right) \cong \prod_{a\in G}\mathrm{Hom}_{\mathcal{M}%
}\left( X_{a},\prod_{b\in G}[Y_{b},Z_{ab}]\right)  \\
=\prod_{a\in G}\mathrm{Hom}_{\mathcal{M}}\left( X_{a},[Y,Z]_{a}\right) =%
\mathrm{Hom}_{\mathcal{M}^{G}}\left( X,[Y,Z]\right) .\qedhere
\end{gather*}
\begin{invisible}
  In generale siano \ $W_{a,b}^{g}$ degli insiemi. Abbiamo usato sopra la
formula $\prod_{g\in G}\prod_{\left( a,b\right) \in
S_{g}}W_{a,b}^{g}=\prod_{a\in G}\prod_{b\in G}W_{a,b}^{ab}.$ Vediamo che
\`{e} vera. Posto $T_{g}:=\left\{ \left( g,a,b\right) \mid g\in G,\left(
a,b\right) \in S_{g}\right\} =\left\{ \left( ab,a,b\right) \mid \left(
a,b\right) \in G\times G\right\} $
\begin{eqnarray*}
\prod_{g\in G}\prod_{\left( a,b\right) \in S_{g}}W_{a,b}^{g}
&=&\prod_{\left( g,a,b\right) \in T_{g}}W_{a,b}^{g} \\
&=&\left\{ f:T_{g}\rightarrow \bigcup \left\{ W_{a,b}^{g}:\left(
g,a,b\right) \in T_{g}\right\} \mid f\left( g,a,b\right) \in
W_{a,b}^{g},\forall \left( g,a,b\right) \in T_{g}\right\}  \\
&=&\left\{ f:T_{g}\rightarrow \bigcup \left\{ W_{a,b}^{ab}:\left(
a,b\right) \in G\times G\right\} \mid f\left( ab,a,b\right) \in
W_{a,b}^{ab},\forall \left( a,b\right) \in G\times G\right\}  \\
&&\overset{h\left( a,b\right) :=f\left( ab,a,b\right) }{=}\left\{ h:G\times
G\rightarrow \bigcup \left\{ W_{a,b}^{ab}:\left( a,b\right) \in G\times
G\right\} \mid h\left( a,b\right) \in W_{a,b}^{ab},\forall \left( a,b\right)
\in G\times G\right\}  \\
&=&\prod_{\left( a,b\right) \in G\times G}W_{a,b}^{ab}=\prod_{a\in
G}\prod_{b\in G}W_{a,b}^{ab}.
\end{eqnarray*}
\end{invisible}
\end{proof}

Next result concerns the pre-rigidity of $\Mm^G$.

\begin{proposition}
\label{pro:funcatG}In the setting of \ref{claim:MG}, assume further that the initial object $\initial$ is also terminal (i.e. a zero object). Then, if $\Mm$ is
pre-rigid so is the category $\Mm^{G}$. Explicitly, the pre-dual of $X=(X_{g})_{g\in G}$ is
defined by setting $\left( X^{\ast }\right) _{g}:=\left( \oplus _{h\in
G,gh=e}X_{h}\right) ^{\ast }.$
\end{proposition}

\begin{proof}
First note that, if we set $W^g _{a,b}:=\delta _{a,g}X_{b},$ we
have that, by hypothesis, $\Mm$ contains $\oplus _{\left( a,b\right)
\in S_{e}}W^g _{a,b}=\oplus _{\left( a,b\right) \in
S_{e}}\delta _{a,g}X_{b}\cong \oplus _{b\in G,gb=e}X_{b}$ so that it makes
sense to define $\left( X^{\ast }\right) _{g}:=\left( Y_g\right) ^{\ast }$, where we set $Y_g:= \oplus
_{h\in G,gh=e}X_{h}$

\begin{invisible}
Let $I$ and $J$ be sets. Consider $\left( X_{i}\right) _{i\in I\cup J}$
where $X_{i}=\mathbf{0}$ for every $i\in J.$ Let us check that $\oplus
_{i\in I\cup J}X_{i}\cong \oplus _{i\in I}X_{i}.$ We compite%
\begin{equation*}
\oplus _{i\in I\cup J}X_{i}\cong \left( \oplus _{i\in I}X_{i}\right) \oplus
\left( \oplus _{i\in J}X_{i}\right) \cong \left( \oplus _{i\in
I}X_{i}\right) \oplus \left( \oplus _{i\in J}\mathbf{0}\right) .
\end{equation*}%
Thus we have to check that $\left( \oplus _{i\in I}X_{i}\right) \oplus
\left( \oplus _{i\in J}\mathbf{0}\right) \cong \oplus _{i\in I}X_{i}.$ It
suffices to check that $X\oplus \left( \oplus _{i\in J}\mathbf{0}\right)
\cong X.$ We have
\begin{eqnarray*}
\hom _{\Mm}\left( X\oplus \left( \oplus _{i\in J}\mathbf{0}%
\right) ,Y\right) &\cong &\hom _{\Mm}\left( X,Y\right) \times
\hom _{\Mm}\left( \oplus _{i\in J}\mathbf{0},Y\right) \\
&\cong &\hom _{\Mm}\left( X,Y\right) \times
\prod\limits_{i\in J}\hom _{\Mm}\left( \mathbf{0},Y\right) \\
&\cong &\hom _{\Mm}\left( X,Y\right) \times
\prod\limits_{i\in J}\left\{ 0_{\mathbf{0},Y}\right\} \cong \hom _{%
\Mm}\left( X,Y\right) .
\end{eqnarray*}%
Since this is natural in $Y,$ by Yoneda we get $X\oplus \left( \oplus _{i\in
J}\mathbf{0}\right) \cong X$ as desired.
\end{invisible}

Since $\initial$ is a zero object, for every morphism $f$ in $\Mm$, we can define $\delta _{g,e}f$ to be $f$ if $g=e$ and to be the zero morphism otherwise. Consider the functors
\begin{eqnarray*}
L:\Mm^{G}\rightarrow \Mm, &&\qquad X=(X_{g})_{g\in G}\mapsto
X_{e},\qquad f=(f_{g})_{g\in G}\mapsto f_{e}, \\
R:\Mm\rightarrow \Mm^{G}, &&\qquad V\mapsto \left( \delta
_{g,e}V\right) _{g\in G},\qquad f\mapsto \left( \delta _{g,e}f\right) _{g\in
G}.
\end{eqnarray*}%
Note that $LRV=\left( RV\right) _{e}=V$ and let $\epsilon _{V}:=\mathrm{Id}%
_{V}.$ Moreover $RLX=RX_{e}=\left( \delta _{g,e}X_{e}\right) _{g\in G}.$
Define $\eta _{X}:=\left( \delta _{g,e}\mathrm{Id}_{X_{e}}\right) _{g\in
G}:X\rightarrow RLX.$ This way we get natural transformations $\eta $ and $%
\epsilon $ such that $\left( L,R,\eta ,\epsilon \right) $ is an adjunction.

\begin{invisible}
Given $f:X\rightarrow Y$ we check the naturality of $\eta $ as follows%
\begin{equation*}
RLf\circ \eta _{X}=\left( \delta _{g,e}f_{e}\right) _{g\in G}\circ \left(
\delta _{g,e}\mathrm{Id}_{X_{e}}\right) _{g\in G}=\left( \delta
_{g,e}f_{e}\right) _{g\in G}=\left( \delta _{g,e}\mathrm{Id}_{Y_{e}}\right)
_{g\in G}\circ \left( \delta _{g,e}f_{e}\right) _{g\in G}=\eta _{Y}\circ f.
\end{equation*}

Since $\epsilon =\mathsf{id}$, in order to check that $\eta $ and $\epsilon $
give rise to an adjunction we have to prove that $L\eta =\mathrm{Id}_{L}$
and $\eta R=\mathrm{Id}_{R}.$ We have%
\begin{eqnarray*}
L\eta _{X} &=&\left( \eta _{X}\right) _{e}=\mathrm{Id}_{X_{e}}=\mathrm{Id}%
_{LX}, \\
\eta _{RV} &=&\left( \delta _{g,e}\mathrm{Id}_{\left( RV\right) _{e}}\right)
_{g\in G}=\left( \delta _{g,e}\mathrm{Id}_{V}\right) _{g\in G}=\mathrm{Id}%
_{RV}.
\end{eqnarray*}

Note that $L$ is not strong monoidal. In fact $L\left( X\otimes Y\right)
=\left( X\otimes Y\right) _{e}=\oplus _{ab=e}X_{a}\otimes Y_{b}\neq
X_{e}\otimes Y_{e}=LX\otimes LY.$
\end{invisible}

Set $G^{r}:=\left\{ a\in G\mid \exists b\in G,ab=e\right\} .$ Then%
\begin{eqnarray*}
\oplus _{a\in G}T_{a}\otimes Y_a&=&\oplus _{a\in G}T_{a}\otimes (\oplus _{b\in G,ab=e}X_{b}) \cong\oplus _{a\in G}\oplus _{b\in G,ab=e}T_{a}\otimes X_{b} \\
&=&\left( \oplus
_{a\in G^{r}}\oplus _{b\in G,ab=e}T_{a}\otimes X_{b}\right) \oplus \left(
\oplus _{a\in G\backslash G^{r}}\oplus _{b\in G,ab=e}T_{a}\otimes
X_{b}\right) \\
&=&\left( \oplus _{\left( a,b\right) \in G\times G,ab=e}T_{a}\otimes
X_{b}\right) \oplus \left( \oplus _{a\in G\backslash G^{r}}\oplus _{b\in
\emptyset }T_{a}\otimes X_{b}\right) \\
&=&\left( \oplus _{ab=e}T_{a}\otimes X_{b}\right) \oplus \left( \oplus
_{a\in G\backslash G^{r}}\initial\right) \cong \oplus _{ab=e}T_{a}\otimes
X_{b}=(T\otimes X)_e=L(T\otimes X).
\end{eqnarray*}%
Moreover, since $\unit^{G}=\left( \delta _{g,e}\unit\right) _{g\in G}=R\unit$, we get
\begin{gather*}
\hom _{\Mm^{G}}\left( T\otimes X,\unit^{G}\right) =%
\hom _{\Mm^{G}}\left( T\otimes X,R\unit\right) \cong
\hom _{\Mm}\left( L\left( T\otimes X\right) ,\unit\right)
\\
\cong\hom _{\Mm}\left( \oplus _{a\in G}T_{a}\otimes Y_a,\unit%
\right)  \cong \prod_{a\in G}\hom %
_{\Mm}\left( T_{a}\otimes Y_a ,%
\unit\right) \\
\cong \prod_{a\in G}\hom _{\Mm}\left( T_{a},\left( Y_a\right) ^{\ast }\right) =\prod_{a\in G}\hom _{%
\Mm}\left( T_{a},\left( X^{\ast }\right) _{a}\right) =\hom _{%
\Mm^{G}}\left( T,X^{\ast }\right) .
\end{gather*}%
A direct computation shows that this yield the bijection $\hom _{%
\Mm^{G}}\left( T,X^{\ast }\right)\to \hom _{\Mm^{G}}\left( T\otimes X,\unit^{G}\right)$, $u\mapsto \ev_{X}\circ(u\otimes X)$ (whence $\Mm^G$ is pre-rigid), where $\ev_{X}$ is defined as follows.    Consider, for every $b\in G$ such that $ab=e$, the canonical inclusion $i_b:X_b\to Y_a$
and the morphism $f_{a,b}$ defined by $(X^*)_a\otimes X_b=(Y_a)^*\otimes X_b\overset{(Y_a)^*\otimes i_b}{\to}(Y_a)^*\otimes Y_a\overset{\ev_{Y_a}}{\to}\unit$. Then $(\ev_{X})_g:\oplus
_{ab=g}(X^*)_{a}\otimes X_b\to \unit^G_g$ is defined to be the zero morphism if $g\neq e$ and to be the codiagonal map of the $f_{a,b}$'s otherwise.
\end{proof}

As a consequence we get the following result.

\begin{proposition}
\label{pro:funcat} 
Let $\Mm$ be a monoidal category where $\Mm
$ has finite coproducts and $\otimes $ preserves them. Assume that the
initial object is also terminal. If $\Mm$ is pre-rigid so is $\Mm^{\NN}$. Explicitly, the pre-dual of $%
X=\left( X_{n}\right) _{n\in \NN}$ is defined by $(X^{\ast })_n:=\delta
_{n,0} \left(X_{0}\right) ^{\ast }$ for every $n\in\NN$. 
\end{proposition}

\begin{proof} Note that, since $\Mm$ has finite coproducts it has also the empty
coproduct i.e. the initial object. Given $n\in \NN$, then $%
S_{n}:=\left\{ \left( a,b\right) \in
\NN
\times
\NN
\mid a+b=n\right\} $ is finite so that $\Mm$ contains all coproducts
indexed by $S_{n}$. As a consequence we are in the setting of \ref{claim:MG} and hence can consider the monoidal category $%
\Mm^{\NN}$ with unit defined by $\mathbf{I}_{n}^{\NN}=\delta
_{n,0}\mathbf{I}$. By Proposition \ref{pro:funcatG}, we get that $\Mm%
^{\NN}$ is pre-rigid. Explicitly, the pre-dual of $X=(X_{n})_{n\in
\NN
}$ is defined by setting $\left( X^{\ast }\right) _{n}=\left( \oplus _{h\in
\NN
,n+h=0}X_{h}\right) ^{\ast }=\left( \delta _{n,0}X_{0}\right) ^{\ast }$.
Note that, since $\initial$ is an initial object, then $\initial^*$ is a terminal object as $\hom _{%
\Mm}\left( T,\mathbf{0}^{\ast }\right) \cong \hom _{\mathcal{M%
}}\left( T\otimes \mathbf{0},\mathbf{I}\right) \cong \hom _{\mathcal{M%
}}\left( \mathbf{0},\mathbf{I}\right) $ is a singleton (we are using that $T\otimes \left( -\right) $ preserves finite coproducts and in particular $\mathbf{0}$ i.e. $T\otimes \mathbf{0\cong 0}$). Thus, we get $\mathbf{0}^{\ast }\cong \mathbf{0}$. As a consequence we arrive at $\left( \delta _{n,0}X_{0}\right) ^{\ast }\cong
\delta _{n,0}\left( X_{0}\right) ^{\ast }$.
\end{proof}


\begin{example}
\label{examplerightclosed}
Consider $%
\Vec ^{\textrm{f}}$ and denote by $%
\Aa$ the category $\left( \Vec ^{\textrm{f}}\right) ^{\mathbb{N%
}}$ of externally $\NN$-graded $ \Vec ^{\textrm{f}}$-objects. Since $\Vec ^{\textrm{f}}$ is a  monoidal category where $\Vec ^{\textrm{f}}$ is abelian and the tensor product preserves finite coproducts, we can apply Proposition \ref{pro:funcat} to conclude that $%
\Aa$ is a pre-rigid monoidal category too. Let us check that
$\Aa$ is not right closed. Suppose the opposite, i.e. assume that $-\otimes V\dashv [V,-]$ for $%
V=(k)_{n\in \NN}\in \Aa$.
Thus, if we consider the unit object $U=\left( \delta _{n,0}k\right) _{n\in
\NN}$, we get
\begin{equation*}
\hom _{\Aa}\left( V,V \right)
\cong\hom _{\Aa}\left( U\otimes V,V \right)
\cong \hom _{\Aa}\left( U,\left[V ,V \right] \right) \cong \left[ V ,V %
\right] _{0}.
\end{equation*}%
Since the latter is finite-dimensional, we obtain the desired contradiction by observing that
\begin{gather*}
\hom _{\Aa}\left( V ,V \right)
=\hom _{\Aa}\left( (k)_{n\in \NN},(k)_{n\in \NN}\right) =\prod _{n\in \NN}\hom _{k}\left( k,k\right)
\cong \prod _{n\in \NN}k= k^{\NN}.
\end{gather*}%
Note that, by the same argument used above, the category $\Vec^{\NN}$  is a pre-rigid monoidal category too. In contrast $\Vec^{\NN}$ is closed, where $[V,W]_n:=\prod_{t\in \NN}\hom _{k}(V_t,W_{t+n})$, as Proposition \ref{pro:funcatClosed} shows.

\end{example}

\section{Pre-rigidity and liftability}\label{finalsection}
In this final section, we propose to study liftability of adjoint pairs of functors in the light of general pre-rigid braided monoidal categories. In \cite{GV-OnTheDuality}, the liftability condition in the motivating examples is shown to hold by rather ad-hoc methods. It is our purpose here to treat the case of generic pre-rigid braided monoidal categories in a more systematic way. We first recall what this liftability condition precisely is (Definition \ref{def:liftable}) and what this condition means for the bialgebra objects in the involved categories. In Example \ref{ex:liftable}, which seems to be new and is considered to be of independent interest, we show that not every adjunction is liftable. In Proposition \ref{prop:prerigid}, we show that the pre-dual construction defines a special type of self adjoint functor $R=(-)^*:\Cc^\op\to\Cc$. In Proposition \ref{lem:Barop}, we show that this type of functor gives rise to a liftable pair whenever the functor it induces at the level of algebras has a left adjoint and we apply this result to the specific functor $R=(-)^*:\Cc^\op\to\Cc$ in Corollary \ref{coro:Isar}. Then, in Proposition \ref{pro:monadj}, we provide a criterion to transport the desired liftability from one category to another in presence of a suitable monoidal adjunction and we apply it, in Corollary \ref{coro:externlift}, to transfer liftability from a category $\Mm$ to the category $\Mm^\NN$ of externally $\NN$-graded objects. As a consequence, we arrive at Example \ref{example-prerigidnotclosed} which revisits Example \ref{examplerightclosed} and provides an instance of a situation in which Corollary \ref{coro:externlift} (properly) holds; it shows how -in favorable cases- the notion of pre-rigid category allows to construct liftable pair of adjoint functors when the right-closedness is not available.
\subsection{Liftability of adjoint pairs}

Let $\left( L:\Bb\rightarrow \Aa,R:\Aa%
\rightarrow \Bb\right) $ be an adjunction with unit $\eta$ and counit $\epsilon$. It is known, see e.g. \cite[Proposition 3.84]%
{Aguiar-Mahajan}, that if $(L,\psi _{2},\psi _{0})$ is a colax monoidal functor,
then $(R,\phi _{2},\phi _{0})$ is a lax monoidal functor where, for every $%
X,Y\in \Aa$,%
%
%
%
\begin{gather}
\phi _{2}\left( X,Y\right) =\left(\xymatrixcolsep{35pt}\xymatrix{RX\otimes RY\ar[r]^-{\eta_{ \left(
RX\otimes RY\right)}}&RL\left( RX\otimes RY\right)\ar[r]^{R\psi _{2}\left( RX,RY\right)} &R\left( LRX\otimes
LRY\right)\ar[r]^-{R\left( \epsilon_{X}\otimes \epsilon_{Y}\right)} & R\left( X\otimes Y\right)}\right), \label{form:PhiFromPsi2}\\
\phi _{0} =\left( \xymatrix{ \unit_{\Bb}\ar[r]^-{\eta_{ \unit_{\mathcal{%
B}}}}& RL\left( \unit_{\Bb}\right) \ar[r]^-{R\psi _{0}}&R\left( \unit_{\Aa}\right)} \right).   \label{form:PhiFromPsi0}
\end{gather}
Conversely, if $(R,\phi _{2},\phi _{0})$ is a lax monoidal functor, then $%
(L,\psi _{2},\psi _{0})$ is a colax monoidal functor where, for every $%
X,Y\in \Bb$
\begin{gather}
\psi _{2}\left( X,Y\right) :=\left(\xymatrixcolsep{35pt}\xymatrix{ L\left( X\otimes Y\right) \ar[r]^-{%
L\left( \eta_{ X}\otimes \eta_{Y}\right) }&L\left( RLX\otimes
RLY\right) \ar[r]^-{L\phi _{2}\left( LX,LY\right) }%
&LR\left( LX\otimes LY\right) \ar[r]^-{\epsilon_{\left( LX\otimes LY\right) }}& LX\otimes LY}\right) ,   \label{form:PsiFromPhi2}\\
\psi _{0} =\left( \xymatrix{L\left( \unit_{\Bb}\right) \ar[r]^-{L\phi _{0}%
}& LR\left( \unit_{\Aa}\right)\ar[r]^-{%
\epsilon_{\unit_{\Aa}}}&\unit_{\Aa%
}}\right) .  \label{form:PsiFromPhi0}
\end{gather}

Let $(R,\phi_2 ,\phi_0 ):\Aa\rightarrow \Bb$ be a lax monoidal functor. It is
well-known that $R$ induces a functor $\overline{R}:={\Alg}(R):{\Alg}({\Aa})\rightarrow {\Alg}({\Bb})$ such that the
diagram on the right-hand side in (\ref{diag:bar}) commutes (cf. \cite[Proposition
6.1, page 52]{Benabou-IntrodBicat}; see also \cite[Proposition 3.29]%
{Aguiar-Mahajan}). Explicitly,%
\begin{equation*}
\overline{R}\left( A,m,u\right) =\left( RA,\xymatrix{RA\otimes RA\ar[r]^-{\phi
_{2}\left( A,A\right) }&R\left( A\otimes A\right)\ar[r]^-{%
Rm}&RA},\xymatrix{\unit_{{\Bb}}\ar[r]^-{\phi _{0}}&R\unit_{{\Aa}}\ar[r]^-{Ru}&
RA}\right).
\end{equation*}%
Dually, a colax monoidal functor $(L,\psi _{2},\psi _{0}):\Bb%
\rightarrow \Aa$ colifts to a functor $\underline{L}:={\Coalg}(L):%
{\Coalg}({\Bb})\rightarrow {\Coalg}({\Aa})$ such
that the diagram on the left-hand side in (\ref{diag:bar}) commutes. Explicitly,%
\begin{equation*}
\underline{L}\left( C,\Delta ,\varepsilon \right) =\left( LC,\xymatrix{LC\ar[r]^-{%
L\Delta }&L\left( C\otimes C\right) \ar[r]^-{\psi
_{2}\left( C,C\right) }&LC\otimes LC},\xymatrix{LC\ar[r]^-{%
L\varepsilon }&L\unit_{{\Bb}}\ar[r]^-{\psi _{0}%
}&\unit_{{\Aa}}}\right).
\end{equation*}%
The vertical arrows in the two diagrams below are the obvious forgetful functors.

\begin{equation}
\begin{array}{ccc}
\xymatrix{{\Coalg}({\Bb})\ar[d]_ {\mho '=\mho _{\Bb} }\ar[rr]^-{\underline{L}=\Coalg(L)}  && {{\Coalg}({\Aa})}\ar[d]^{\mho
=\mho _{{\Aa}}}  \\
{\Bb}   \ar[rr]^-{L} && {\Aa}
}
\end{array}%
\qquad
\begin{array}{ccc}
\xymatrix{{\Alg}({\Aa})\ar[d]_ {\Omega =\Omega _{{\Aa}}}\ar[rr]^-{\overline{R}=\Alg(R)}  && {{\Alg}({\Bb})}\ar[d]^{\Omega '
=\Omega _{{\Bb}}}  \\
{\Aa}   \ar[rr]^-{R} && {\Bb}
}
\end{array}
\label{diag:bar}
\end{equation}

\begin{definition}[{\protect\cite[Definition 2.3]{GV-OnTheDuality}}]
\label{def:liftable}Suppose $\Aa$ and $\Bb$  are monoidal
categories and $R:\Aa\rightarrow \Bb$ is a lax
monoidal functor with a left adjoint $L$. The pair $(L,R)$ is called \textit{liftable} if the induced functor $\overline{R}=\Alg(R):\Alg(\Aa)\to \Alg(\Bb)$ has a left adjoint, denoted by $\overline{L}$, and the induced functor $\underline{L}=\Coalg(L):\Coalg(\Bb)\to \Coalg(\Aa)$ has a right adjoint, denoted by $\underline{R}$.
\end{definition}

\subsection{Liftability for braided categories}
 Recall that when a category is \textit{braided} monoidal, its category of algebras and its category of coalgebras inherit the monoidal structure, see e.g. \cite[1.2.2]{Aguiar-Mahajan}.
Let $\Aa$ and $\Bb$ now be braided monoidal
categories and let $R:\Aa\rightarrow \Bb$ be a braided lax
monoidal functor having a left adjoint $L$. By \cite[Proposition 3.80]%
{Aguiar-Mahajan}, the functor
$\overline{R}$ is lax monoidal too. Explicitly, the lax monoidal functors $%
(R,\phi _{2},\phi _{0})$ and $(\overline{R},\overline{\phi }_{2},\overline{%
\phi }_{0})$ are connected by the following equalities, for every $\overline{%
A}=\left( A,m_{A},u_{A}\right) ,\overline{B}=\left( B,m_{B},u_{B}\right) \in
{\Alg}({\Aa})$
\begin{equation}
\Omega _{{\Bb}}\circ \overline{R}=R\circ \Omega _{\mathcal{%
A}} ,\qquad \Omega _{{\Bb}} (\overline{\phi }%
_{2}\left( \overline{A},\overline{B}\right)) =\phi _{2}\left( A,B\right)
,\qquad  \Omega _{{\Bb}}(\overline{\phi }_{0})=\phi _{0}.
\label{form:phiOver}
\end{equation}%
Note that $R$ is a braided lax monoidal functor if and only if $L$ is a braided colax monoidal functor, see e.g. \cite[Proposition 3.85]%
{Aguiar-Mahajan}. Moreover, if   $L$ is a braided colax monoidal functor one shows in a similar fashion as above that $\underline{L}$ is colax monoidal.
The colax monoidal functors $(L,\psi _{2},\psi _{0})$ and $(%
\underline{L},\underline{\psi }_{2},\underline{\psi }_{0})$ are connected by
the following equalities for every $\underline{C}=\left( C,\Delta
_{C},\varepsilon _{C}\right) ,\underline{D}=\left( D,\Delta _{D},\varepsilon
_{D}\right) \in {\Coalg}({\Bb})$
\begin{equation}
\mho _{\Aa}\circ \underline{L}=L\circ \mho _{{\Bb}%
} ,\qquad  \mho _{\Aa}( \underline{\psi }%
_{2}\left( \underline{C},\underline{D}\right) )=\psi _{2}\left( C,D\right)
,\qquad \mho _{\Aa}(\underline{\psi }_{0})=\psi _{0}.
\label{form:psiUnder}
\end{equation}
As the following Example \ref{ex:liftable} shows, a pair $(L,R)$, where $R:\Aa\rightarrow \Bb$ is a (braided) lax
monoidal functor between (braided) monoidal categories $\Aa$ and $\Bb$, having a left adjoint $L$, needs not to be liftable, \textit{a priori}. But, in case $\Aa$ and $\Bb$ are braided monoidal
categories and $R:\Aa\rightarrow \Bb$ is a braided lax
monoidal functor having a left adjoint $L$ such that the pair $(L,R)$ \textit{is} liftable, then, by \cite[Lemma 2.4 and Theorem 2.7%
]{GV-OnTheDuality}, there is an adjunction $\left( \overline{\underline{L}},%
\overline{\underline{R}}\right) $ that fits into the following commutative diagrams (and explains the choice of the -perhaps somewhat fuzzy- term ``liftable'')

\begin{equation}
\xymatrix{{\Bialg}({\Bb})\ar[d]_ {\overline{\mho^{\prime}}}\ar[rr]^-{\overline{\underline{L}}=\Coalg(\overline{L})}  && {{\Bialg}({\Aa})}\ar[d]^{\overline{\mho '}} \\
{\Alg(\Bb})   \ar[rr]^-{\overline{L}} && {\Alg(\Aa})
}
\qquad
\xymatrix{{\Bialg}({\Aa})\ar[d]_ {\underline{\Omega}}\ar[rr]^-{\underline{\overline{R}}=\Alg(\underline{R})}  && {{\Bialg}({\Bb})}\ar[d]^{\underline{\Omega '}}  \\
{\Coalg(\Aa)}   \ar[rr]^-{\underline{R}} && {\Coalg(\Bb})
}
\label{diag:bibar}
\end{equation}
In this diagram, all vertical arrows are forgetful functors.

One could wonder whether \textit{any} appropriate adjunction $(L,R)$ is liftable. The answer is no: below we present an (apparently original) example of a lax monoidal functor $R$ between monoidal categories that has a left adjoint $L$, but for which $\overline{R}$ does not have a left adjoint.

\begin{example}\label{ex:liftable}
Let $k$ be a field and set $S:=\frac{k \left[ X\right]}{\left( X^{2}\right) }.$ Consider the functor
$$R^{f}:\Vec ^{\textrm{f}}\rightarrow \Vec ^{\textrm{f}},\quad V\mapsto
S{\otimes}_{k} V.$$
Note that the functor $R^{f}$ has a left adjoint $L^{f}$, where $%
L^{f}\left( V\right) =S^{\ast }{\otimes}_{k} V.$ As $S$ is an algebra, the functor $R^{f}$ is lax
monoidal with respect to
\begin{gather*}
  \phi_2(X,Y):(S{\otimes}_{k}X)\otimes (S{\otimes}_{k}Y)\to S{\otimes}_{k}(X\otimes Y),\quad(s\otimes_kx)\otimes (t\otimes_ky)\mapsto st\otimes_k(x\otimes y)\\
  \phi_0:k\to S{\otimes}_{k}k,\quad q\mapsto 1_S\otimes_k q
\end{gather*}
so that it induces a functor $\overline{R^{f}}:\Alg%
^{f}\rightarrow \Alg^{f}$, where we used the notation $\Alg%
^{f}=\Alg\left( \Vec ^{f}\right) $ for the category of
finite-dimensional algebras.
\\Our aim is to check that $\overline{R^{f}}$ has
no left adjoint.
\\To this end, suppose that there is a left adjoint $%
\overline{L^{f}}$ of $\overline{R^{f}}$ and denote by $\overline{\eta ^{f}}$
and $\overline{\epsilon ^{f}}$ the corresponding unit and counit. Consider
the functor $R:\Vec \rightarrow \Vec :V\mapsto S{\otimes}_{k} V.$
This functor has a left adjoint $L$ and induces a functor $\overline{R}:\mathsf{%
Alg}\rightarrow \Alg.$ By a result of Tambara (cf. \cite[Remark 1.5]{Tambara}), this functor has a
left adjoint $\overline{L}=a(S,-)$ with unit and counit $\overline{\eta }$ and $%
\overline{\epsilon }.$ By \cite[Example 1.2(ii)]{Tambara}, one has
\begin{equation*}
\overline{L}\left( S\right) =a\left( S,S\right) =k\left\{ X,Y\right\}
/\left( X^{2},XY+YX\right) .
\end{equation*}
Notice that this algebra is not finite-dimensional.
Consider the forgetful functor $\overline{\Lambda }:\Alg%
^{f}\rightarrow \Alg$. Clearly $\overline{\Lambda }\circ \overline{%
R^{f}}=\overline{R}\circ \overline{\Lambda }$. 
We will negate that $\overline{\Lambda }\overline{L^{f}}%
\left( S\right)$ is finite-dimensional by showing that the following map is injective when restricted to some infinite-dimensional subspace of its domain. $$
\zeta :=\left( {\overline{\epsilon }}_{\overline{\Lambda }\overline{L^{f}}%
}\circ \overline{L}\,\overline{\Lambda }\overline{\eta ^{f}}\right) _{S}:\overline{L}\,\overline{%
\Lambda }\left( S\right) \rightarrow \overline{\Lambda }\overline{L^{f}}%
\left( S\right) .$$
It is easy to check that the obvious chain of isomorphisms $\Alg\left( \overline{\Lambda }\overline{L^{f}}(S),\overline{\Lambda }%
(B)\right)  \cong \Alg^{f}\left( \overline{L^{f}}(S),B\right) \cong
\Alg^{f}\left( S,\overline{R^{f}}(B)\right) \cong \Alg\left(
\overline{\Lambda }(S),\overline{\Lambda }\overline{R^{f}}(B)\right) =\mathsf{Alg%
}\left( \overline{\Lambda }(S),\overline{R}\,\overline{\Lambda }(B)\right) \cong
\Alg\left( \overline{L}\,\overline{\Lambda }(S),\overline{\Lambda }%
(B)\right)$ is exactly $\Alg\left( \zeta,\overline{\Lambda }%
(B)\right)$ so that the latter is invertible for every $B\in \Alg^{f}$.
Since $ k\left[ \left[ Y\right] \right] $ is the inverse limit of $\overline{\Lambda }\left(\frac{k\left[
Y\right] }{\left( Y^{n}\right) }\right),$ we have that $\Alg\left( \zeta,k\left[ \left[ Y\right] \right]\right)\cong \Alg\left( \zeta,\underleftarrow{\lim} \overline{\Lambda }\left(\frac{k\left[
Y\right] }{\left( Y^{n}\right) }\right)\right)\cong \underleftarrow{\lim}\Alg\left( \zeta, \overline{\Lambda }\left(\frac{k\left[
Y\right] }{\left( Y^{n}\right) }\right)\right)$ is invertible too. We now construct the diagram
\begin{equation*}
  \xymatrix{k\left[ Y\right]\ar@{^(->}[r]^\gamma\ar@{^(->}[dr]_\tau & \overline{L}\,\overline{\Lambda }(S)\ar[r]^\zeta\ar[d]_{\pi}& \overline{\Lambda }\overline{L^{f}}(S)\ar@{.>}[dl]^{\beta}\\
  &k\left[ \left[ Y\right] \right]
  }
\end{equation*}

Consider the following maps

\begin{itemize}
\item $\pi :\overline{L}\overline{\Lambda}\left( S\right) =\frac{k\left\{ X,Y\right\} }{%
\left( X^{2},XY+YX\right) }\longrightarrow k\left[ \left[ Y\right] \right]:\overline{X}\mapsto 0;\overline{Y}\mapsto \overline{Y}.$
\item $\gamma:k[Y]\hookrightarrow \frac{k\left\{ X,Y\right\} }{%
\left( X^{2},XY+YX\right) }$ and $\tau=\pi\circ\gamma :k[Y]\hookrightarrow k[[Y]]$ are the canonical injections.

\end{itemize}

Since $\Alg\left( \zeta,k[[Y]]\right)$ is invertible, there is a unique $\beta \in \Alg\left(
\overline{L}\,\overline{\Lambda}\left( S\right) ,K[[Y]]\right) $ such that $\beta \circ \zeta =\pi $.
Now we compute $ \beta \circ
\zeta \circ \gamma =\pi \circ \gamma = \tau .$  Thus, since $\tau$ is injective, so is  $%
\zeta \circ \gamma $ and we obtain that $\overline{\Lambda }\overline{L^{f}}\left(
S\right) $ contains a copy of $k[Y]$, which implies that $\overline{\Lambda }\overline{L^{f}}\left(
S\right) $ is not finite-dimensional. This is a contradiction.
\end{example}

\begin{remark}
With respect to the ``liftability'' terminology, it seems opportune to mention some related work, carried out by Porst and Street in \cite{PS}.
\\In Section 3 of loc. cit., the authors assume $\overline{R}$ to admit a left adjoint $\overline{L}$ and are concerned with investigating which of the properties of Sweedler's finite dual functor $(-)^{\circ}$ might be shared by $\overline{L}$. We note that they also use a notion of ``liftability'' (Definition 14 in loc. cit.) which does not coincide with the notion of a liftable pair of functors as in Definition \ref{def:liftable} here above.
%
%
\\It is also instructive to remark that, in Section 3.3.2 of \cite{PS}, the authors study symmetric monoidal functors, obtaining the following result (item 1 of Proposition 33 in their article). Let $\Aa$ and $\Bb$ be symmetric monoidal closed categories and $R:\Aa\rightarrow \Bb$ be a symmetric lax
monoidal functor having a left adjoint $L$ such that $\oR$ has a left adjoint. Assuming that $\Bb$ is locally presentable, $\uoL: \Bialg({\Bb})\to \Bialg({\Aa})$ has a right adjoint.
\end{remark}

\subsection{Liftability of the functor computing pre-duals} Having recalled the theory of liftable functors, starting from a pre-rigid braided monoidal category $\Cc$, which is not necessarily closed, we aim to construct a self-adjoint (on the right) functor $(-)^{*}: \Cc^{\op }\rightarrow \Cc$, in Proposition \ref{prop:prerigid}, and afterwards to provide sufficient conditions to obtain a liftable adjunction from it. An occurrence of this situation is the case $\Cc=\Vec$, expounded in \cite[Section 3]{GV-OnTheDuality}. This example, however, is closed monoidal. The first part of the following result is \cite[Proposition 4.2]{GV-OnTheDuality}: there is however some difference in the proof, which is explained in Remark \ref{rem:GV4.2}.

\begin{proposition}\label{prop:prerigid}
When $\Cc$ is a pre-rigid braided monoidal category, the assignment $X\mapsto X^*$ induces a functor $R=(-)^{*}:\Cc^{\op}\rightarrow \Cc$ with a left adjoint $L=R^{\op}=(-)^{*}:\Cc\rightarrow \Cc^{\op}.$ Moreover there are $\phi _{2},\phi _{0}$ such that  $\left( R,\phi _{2},\phi _{0}\right) $ is lax monoidal and the induced colax monoidal structure on $L$ by  \eqref{form:PsiFromPhi2} and \eqref{form:PsiFromPhi0} is specifically $(\phi _{2}^{\op},\phi _{0}^{\op})$.
\end{proposition}

\begin{proof}Let $\Cc$ be a pre-rigid braided monoidal category. Since $\Cc$ has a braiding, we can apply Proposition \ref{pro:adjprig} to get a bijection $\hom _{\Cc}\left( Y,X^{\ast }\right) \cong \hom _{\Cc}\left( X,Y^{\ast }\right) $ natural both in $X$ and $Y$, whence the claimed adjunction. In order to write it explicitly, note that for every morphism $t:T\otimes X\rightarrow \unit$ there is a
unique morphism $t^\dag:T\rightarrow X^{\ast }$ such that $t=\ev%
_{X}\circ \left( t^\dag\otimes X\right) .$ Set
\begin{align*}
\eta_{X} &:=({\ev_{X}\circ c_{X,X^{\ast }}})^\dag : X\rightarrow
X^{\ast \ast }, \\
j_{X} &:=({\ev_{X}\circ \left( c_{X^{\ast },X}\right) ^{-1}})^\dag %
:X\rightarrow X^{\ast \ast }.
\end{align*}%
 Equivalently%
\begin{eqnarray}
\ev_{X}\circ c_{X,X^{\ast }} &=&\ev_{X^{\ast }}\circ \left(
\eta_{X}\otimes X^{\ast }\right) , \label{form:etaX}\\
\ev_{X}\circ \left( c_{X^{\ast },X}\right) ^{-1} &=&\ev%
_{X^{\ast }}\circ \left( j_{X}\otimes X^{\ast }\right)\label{form:jX} .
\end{eqnarray}%
By Lemma \ref{lem:contravariant} we have a functor $R=(-)^{*}:\Cc^{\op}\rightarrow \Cc$
defined by $R(X^{\op }):=X ^{\ast }$ and $R(f^{\op }):=f ^{\ast }$.
Then $(L=R^{\op },R,\eta,\epsilon)$ is an adjunction, where we set ${\epsilon }_{X^{\op }}=\left(j_{X}\right) ^{\op }$.

Define $\varphi _{2}\left( X,Y\right) :=({\left( \ev%
_{X}\otimes \ev_{Y}\right) \circ ( X^{\ast }\otimes \left(
c_{X,Y^{\ast }}\right) ^{-1}\otimes Y) })^\dag :X^{\ast }\otimes Y^{\ast
}\rightarrow \left( X\otimes Y\right) ^{\ast }$, i.e. the morphism that
corresponds to $\left( \ev%
_{X}\otimes \ev_{Y}\right) \circ ( X^{\ast }\otimes \left(
c_{X,Y^{\ast }}\right) ^{-1}\otimes Y)$
via the bijection
\begin{equation}\label{mapvaphi}
\hom _{\Cc}\left( X^{\ast }\otimes Y^{\ast },\left( X\otimes
Y\right) ^{\ast }\right) \overset{\cong }{\longrightarrow }\hom _{%
\Cc}\left( X^{\ast }\otimes Y^{\ast }\otimes X\otimes Y,\unit%
\right) .
\end{equation}%
Define $\phi _{0}:\unit\rightarrow %
\unit^{\ast }$ by $\phi _{0}=(m_{\unit})^\dag$, i.e. such that $\ev_{\unit}\circ \left( \phi _{0}\otimes \unit\right)
=m_{\unit}$, and define $\phi _{2}\left( X^{\op},Y^{\op}\right) :=\varphi _{2}\left( X,Y\right)$.
It is straightforward to check that $\left( R,\phi _{2},\phi _{0}\right) $ is lax monoidal. Now by \eqref{form:PsiFromPhi2} and \eqref{form:PsiFromPhi0}, we know that $\left( L,\psi _{2},\psi
_{0}\right) $ is colax monoidal where $\psi _{2}\left( X,Y\right) =\epsilon
_{\left( LX\otimes LY\right) }\circ L\phi _{2}\left( LX,LY\right) \circ
L\left( \eta _{X}\otimes \eta _{Y}\right) $ and $\psi _{0}=\epsilon _{%
\unit^{\op}}\circ L \phi _{0} .$ We compute
\begin{eqnarray*}
\psi _{2}\left( X,Y\right) ^{\op} &=&\left[ L\left( \eta _{X}\otimes \eta _{Y}\right) \right] ^{\op}\circ \left[
L\phi _{2}\left( LX,LY\right) \right] ^{\op}\circ \left[ \epsilon _{\left(
LX\otimes LY\right) }\right] ^{\op} \\
&=&\left( \eta_{X}\otimes \eta_{Y}\right)^* \circ (\varphi _{2}\left( X^*,Y^*\right))^*
\circ j_{X^*\otimes Y^*} \overset{(*)}{=}\varphi _{2}\left( X,Y\right) =\phi _{2}\left( X^{\op},Y^{\op}\right),
\end{eqnarray*}%
where $(*)$ can be checked by applying the bijection (\ref{mapvaphi}) on both sides. Finally, $
 \psi _{0} ^{\op} =\left( L \phi _{0} %
\right) ^{\op}\circ \left( \epsilon _{\unit^{\op}}\right) ^{\op}=\left( \phi _{0}\right)^* \circ j_{\unit}=\phi _{0}
$, where the last equality follows by applying the bijection $\hom _{\Cc}\left( \unit,\unit^*\right)\to \hom _{\Cc}\left( \unit\otimes \unit,\unit\right)$, $u\mapsto \ev_{\unit}\circ \left( u\otimes \unit\right)$, on both sides.
\end{proof}

\begin{remark}\label{rem:GV4.2}
\cite[Proposition 4.2]{GV-OnTheDuality} asserts that if $\Aa$ is a pre-rigid braided monoidal category, then $(-)^*: \Aa^{\op}\rightarrow \Aa$ is a self-adjoint covariant functor. Although the assertion is true for general pre-rigid braided monoidal categories (as shown in the above Proposition \ref{prop:prerigid}), the proof is erroneously communicated in loc. cit.. Indeed, the last sentence of the argument appearing in the printed version of the above-cited proposition only works in case the braiding is moreover symmetric (notice this does not harm the conclusions of the work carried out in loc.cit., as all involved braidings there {\it are} symmetric) as, in general, the unit and counit are given by different underlying morphisms, see above (note that a functor $F:\Aa^{\op}\to \Aa$ such that $F^\op\dashv F$ where the unit and the counit are given by the same underlying morphism is sometimes called a ``self-dual adjunction'' in the literature). The requirement that the braiding is symmetric has been added in \cite[Proposition 4.2]{GV-OnTheDuality-rev}. We point out that, even if $\Cc$ is rigid, in general we cannot conclude that the unit and the counit are given by the same underlying morphism unless $\ev_{X}\circ c_{X,X^{\ast }}\circ  c_{X^{\ast },X} =\ev_{X}$ for every object $X$, in view of the equalities \eqref{form:etaX} and \eqref{form:jX}.
\end{remark}

In Proposition \ref{prop:prerigid}, the functor $R=(-)^{*}:\Cc^{\op}\rightarrow \Cc$ is proved to be self-adjoint on the right. Notice that the unit and counit do not share the same underlying morphism here. Moreover, the induced colax monoidal structure on $L=R^\op$ by  \eqref{form:PsiFromPhi2} and \eqref{form:PsiFromPhi0} is specifically $(\phi _{2}^{\op},\phi _{0}^{\op})$, where $(\phi _{2},\phi _{0})$ is the lax monoidal structure of $R$. The last property seems to be a particular feature of $(-)^{*}$ as we cannot prove in general that it holds true for an arbitrary functor $R$ which is self-adjoint on the right. Next aim is to show that, when it holds true, then $(L,R)$ is liftable whenever $\overline{R}$ has a left adjoint. Of course this will be applied to  examine whether the pair $((-)^{*\op}, (-)^*)$ is liftable.

To this aim recall that an adjunction $(L,R,\eta,\varepsilon)$ gives rise to an adjunction $(R^\op ,L^\op ,\varepsilon^\op ,\eta^\op )$.

%
\begin{proposition}
\label{lem:Barop} For a monoidal category $\Cc$, suppose a lax monoidal functor $\left( R,\phi _{2},\phi _{0}\right) :\Cc^{\op}\to \Cc$ has a left adjoint $L=R^{\op}$. If the induced colax monoidal structure on $L$ by  \eqref{form:PsiFromPhi2} and \eqref{form:PsiFromPhi0} is specifically $(\phi _{2}^{\op},\phi _{0}^{\op})$, then $\overline{R}=\left( \underline{L}\right) ^{\op}$.  Moreover, if $\overline{R}$ has a left adjoint, then $\left( L,R\right) $ is liftable.
\end{proposition}

\begin{proof}
We check that $\left( \underline{L}\right) ^{\op}=\overline{R}$.
First observe that the domain and codomain of $\left( \underline{L}\right) ^{%
\op}$ are respectively
\begin{equation*}
\left( \Coalg (\Cc)\right) ^{\op}=\Alg\left(\Cc^{\op}\right) \qquad \text{and}\qquad \left( \Coalg\left( \Cc^{\op}\right) \right) ^{\op}=\Alg\left( \Cc\right)
\end{equation*}%
so that the domain and codomain of $\left( \underline{L}\right) ^{\op%
}$ and $\overline{R}$ are the same.
Next, by means of the equality $\left(L,\psi _{2},\psi _{0}\right)=\left( R^{\op},\phi _{2}^{\op},\phi _{0}^{\op}\right)$, one checks, in a straightforward fashion, that $\left( \underline{L}\right) ^{\op}$ and $\overline{R}$ coincide on objects.
They also agree on morphisms (whence $\left(
\underline{L}\right) ^{\op}=\overline{R}$) by the following computation
$$\Omega _{\Cc}\circ \left( \underline{L}\right) ^{\op}=\left( \mho _{\Cc^{\op%
}}\right) ^\op \circ \left( \underline{L}\right) ^{\op}
=\left(  \mho _{\Cc^{\op}}\circ \underline{L} \right)^\op=\left(  L\circ \mho _{\Cc} \right)^\op
=L^\op\circ \left(\mho _{\Cc} \right)^\op
=R\circ \Omega _{\Cc^\op}=\Omega _{\Cc}\circ \overline{R}
$$ together with the faithfulness of $\Omega _{\Cc}$.
We now prove the final sentence of the statement. Assume $\overline{R}$ has a left adjoint $\overline{L}$. Thus we have the adjunction $((\overline{R})^{\op},(\overline{L})^{\op})$. Now, by the first part, we have that $\left( \underline{L}\right) ^{\op}=\overline{R}$ and hence  $\underline{L}=(\overline{R})^{\op}$. Thus $\underline{L}$ has a right adjoint and hence $\left( L,R\right) $ is liftable.
\end{proof}

\begin{corollary}\label{coro:Isar}
Let $\Cc$ be a pre-rigid braided monoidal category. If $\overline{(-)^{*}}:\Alg(\Cc^{\op})\to \Alg(\Cc)$ has a left  adjoint, then $\left( (-)^{*}:\Cc\rightarrow \Cc^{\op},(-)^{*}:\Cc^{\op}\rightarrow \Cc\right)$ is a liftable pair of adjoint functors.
\end{corollary}

\begin{proof}
It follows by Proposition \ref{prop:prerigid} and Proposition \ref{lem:Barop}.
\end{proof}

Recall that an adjunction $(L,R)$ between two lax monoidal functors $L$  and $R$ is called a \emph{monoidal adjunction} whenever the unit and the counit of the adjunction are monoidal natural transformations. The following result allows to transfer the condition required in Corollary \ref{coro:Isar} to have liftability from a pre-rigid braided monoidal category $\Mm$ to another one $\Nn$ whenever these categories are connected by a suitable monoidal adjunction $L\dashv R:\Mm\to\Nn$.

\begin{proposition}
\label{pro:monadj}Let $\Mm$ and $\Nn$ be braided monoidal categories. Assume that $\Mm$ is pre-rigid  and that there is a monoidal adjunction $L\dashv R:\Mm\to\Nn$ with $L$ and $R$ both strict monoidal and $L$ braided monoidal. Then $\Nn$ is pre-rigid, with pre-dual $N^*=R((LN)^*)$, for every object $N$ in $\Nn$. If the
assumption in Corollary \ref{coro:Isar} holds for $\Mm$, then the
analogous conclusion holds for  $\Nn$.
\end{proposition}

\begin{proof}Since $L$ is strict monoidal, it is in particular strong monoidal. Moreover, since $L$ and $R$ are strict monoidal we have $\unit_{\Nn}= R(\unit_{\Mm})=  RL(\unit_{\Nn})$. Thus we are in the setting of Proposition \ref{pro:prigadj} so that $\Nn$ is pre-rigid, with pre-dual $N^*=R((LN)^*)$, for every object $N$ in $\Nn$.

Now, consider the adjunctions
\begin{align*}
\left( L_{1},R_{1}\right) &=\left( (-)^{\ast }:\mathcal{%
M}\rightarrow \Mm^{\op },(-)^{\ast }:\Mm^{\op %
}\rightarrow \Mm\right) \\
\left( L_{2},R_{2}\right) &=\left( (-)^{\ast }:\Nn\rightarrow \Nn ^{\op %
},(-)^{\ast }:\Nn ^{\op %
}\rightarrow \Nn\right) .
\end{align*} as in the following diagram \begin{equation*}
\xymatrixcolsep{1.5cm}\xymatrix{{\Mm}^\op\ar@<.5ex>[r]^{R^\op}\ar@<.5ex>[d]^{R_1} & \Nn^\op
\ar@<.5ex>[d]^{R_2}\ar@<.5ex>[l]^{L^\op}\\
\Mm\ar@<.5ex>[u]^{L_1}\ar@<.5ex>[r]^{R}&\Nn\ar@<.5ex>[u]^{L_2}\ar@<.5ex>[l]^{L}}
\end{equation*}

By Proposition \ref{prop:prerigid}, we have the functors $\overline{R_{1}}=\Alg(R_1):{%
\Alg}(\Mm^{\op })\rightarrow {\Alg}(\mathcal{M%
})$ and $\overline{%
R_{2}}=\Alg(R_2):{\Alg}(\Nn ^{\op })\rightarrow {\Alg}(\Nn).$ Assume that $\overline{R_{1}}$ has a left adjoint, say $\overline{L_{1}}$, and let us check that  the functor $\overline{%
R_{2}}$ admits a left adjoint too. Since the functors $L$ and $R$ are strict monoidal, they are in particular lax monoidal whence they induce $\overline{L}=\Alg(L):{\Alg}(%
\Nn)\rightarrow {\Alg}(\Mm),$ $%
\overline{R}=\Alg(R):{\Alg}(\Mm)\rightarrow {\Alg}(\Nn)$ and by \cite[Proposition 3.91]{Aguiar-Mahajan}, we have  that $\overline{L}\dashv \overline{R}$. Since the functors $L^{\op }$ and $R^{\mathrm{%
op}}$ are also strict monoidal, we have the functors  $\overline{L^\op}=\Alg(L^\op):{\Alg}(\Nn ^{\op %
})\rightarrow {\Alg}(\Mm^{\op }),$ $\overline{R^{%
\op }}=\Alg(R^\op):{\Alg}(\Mm^{\op })\rightarrow {\mathsf{%
Alg}}(\Nn ^{\op })$. Note the $%
L\dashv R$ implies that ${R^{\op }}\dashv {L^{%
\op }}$ and hence $\overline{R^{\op }}\dashv \overline{L^{%
\op }}$. As a consequence $\overline{R\circ
R_{1}\circ L^{\op }}=\overline{R}\circ \overline{R_{1}}\circ
\overline{L^{\op }}$ has $\overline{R^{\op }}\circ \overline{%
L_{1}}\circ \overline{L}$ as a left adjoint. It remains to check that $\overline{R_{2}}=\overline{R\circ
R_{1}\circ L^{\op }}$. To this aim we have to check that $R_{2}$ and $R\circ
R_{1}\circ L^{\op }$ are the same as monoidal functor. For $N$ an object in $\Nn$, we have $R
R_{1} L^{\op }(N^\op)=RR_{1} ((LN)^\op)=R((LN)^*)=N^*=R_2(N^\op)$. The same holds on morphisms so that $R
\circ R_{1}\circ L^{\op }=R_2.$ Using the fact that the unit and the counit of the adjunction $(L,R)$ are monoidal natural transformations and that $L$ is braided monoidal, one easily checks that $R
\circ R_{1}\circ L^{\op }$ and $R_2$ have the same monoidal structure.
\begin{invisible}
  Let us check that $RR_{1}L^{\op }=R_{2}$ as monoidal functors i.e.
that $\phi _{2}^{RR_{1}L^{\op }}=\phi _{2}^{R_{2}}$ and that $\phi
_{0}^{RR_{1}L^{\op }}=\phi _{0}^{R_{2}}.$

If $L:\mathcal{N}\rightarrow \mathcal{M}$ is strict monoidal so is $L^{%
\op }:\mathcal{N}^{\op }\rightarrow \mathcal{M}^{\op }$
as%
\begin{eqnarray*}
L^{\op }X^{\op }\otimes L^{\op }Y^{\op }
&=&\left( LX\right) ^{\op }\otimes \left( LY\right) ^{\op %
}=\left( LX\right) ^{\op }\otimes \left( LY\right) ^{\op %
}=\left( LX\otimes LY\right) ^{\op }=\left( L\left( X\otimes Y\right)
\right) ^{\op } \\
&=&L^{\op }\left( X\otimes Y\right) ^{\op }=L^{\op %
}\left( X^{\op }\otimes Y^{\op }\right) , \\
L^{\op }I_{\mathcal{N}}^{\op } &=&\left( LI_{\mathcal{N}%
}\right) ^{\op }=I_{\mathcal{M}}^{\op }
\end{eqnarray*}%
Since also $R$ is strict monoidal, The monoidal structure of $RR_{1}L^{%
\op }$ is given by
\begin{eqnarray*}
\phi _{2}^{RR_{1}L^{\op }} &:&=\left( RR_{1}L^{\op }X^{\mathrm{%
op}}\otimes RR_{1}L^{\op }Y^{\op }=R\left( R_{1}L^{\op %
}X^{\op }\otimes R_{1}L^{\op }Y^{\op }\right) \overset{%
R\phi _{2}^{R_{1}}\left( L^{\op }X^{\op },L^{\op }Y^{%
\op }\right) }{\longrightarrow }RR_{1}\left( L^{\op }X^{%
\op }\otimes L^{\op }Y^{\op }\right) =RR_{1}L^{\mathrm{%
op}}\left( X^{\op }\otimes Y^{\op }\right) \right) , \\
\phi _{0}^{RR_{1}L^{\op }} &:&=\left( I_{\mathcal{N}}=RI_{\mathcal{M}}%
\overset{R\phi _{0}^{R_{1}}}{\longrightarrow }RR_{1}I_{\mathcal{M}}^{\mathrm{%
op}}=RR_{1}L^{\op }I_{\mathcal{N}}^{\op }\right) .
\end{eqnarray*}%
Since $R_{1}=\left( -\right) ^{\ast }:\mathcal{M}^{\op }\rightarrow
\mathcal{M}$, we can write explicitly $\phi _{2}^{R_{1}}\ $and $\phi
_{0}^{R_{1}}$. Explicitly%
\begin{equation*}
\phi _{2}^{R_{1}}\left( A^{\op },B^{\op }\right) =\varphi
_{2}^{R_{1}}\left( A,B\right) =\left[ \left( \mathrm{ev}_{A}\otimes \mathrm{%
ev}_{B}\right) \circ \left( A^{\ast }\otimes \left( c_{A,B^{\ast }}\right)
^{-1}\otimes B\right) \right] ^{\dag }
\end{equation*}%
i.e. $\varphi _{2}^{R_{1}}\left( A,B\right) $ is uniquely determined by the
equality%
\begin{equation*}
\mathrm{ev}_{A\otimes B}\circ \left( \varphi _{2}^{R_{1}}\left( A,B\right)
\otimes B\otimes A\right) =\left( \mathrm{ev}_{A}\otimes \mathrm{ev}%
_{B}\right) \circ \left( A^{\ast }\otimes \left( c_{A,B^{\ast }}\right)
^{-1}\otimes B\right) .
\end{equation*}%
In particular $\phi _{2}^{R_{1}}\left( L^{\op }X^{\op },L^{%
\op }Y^{\op }\right) =\phi _{2}^{R_{1}}\left( \left( LX\right)
^{\op },\left( LY\right) ^{\op }\right) =\varphi
_{2}^{R_{1}}\left( LX,LY\right) $ is uniquely determined by the equality
\begin{equation*}
\mathrm{ev}_{LX\otimes LY}\circ \left( \varphi _{2}^{R_{1}}\left(
LX,LY\right) \otimes LY\otimes LY\right) =\left( \mathrm{ev}_{LX}\otimes
\mathrm{ev}_{LY}\right) \circ \left( \left( LX\right) ^{\ast }\otimes \left(
c_{LX,\left( LY\right) ^{\ast }}\right) ^{-1}\otimes LY\right) .
\end{equation*}%
On the other hand since $R_{2}=\left( -\right) ^{\ast }:\mathcal{N}^{\mathrm{%
op}}\rightarrow \mathcal{N}$, we can write explicitly $\phi _{2}^{R_{2}}\ $%
and $\phi _{0}^{R_{2}}$. Explicitly  $\phi _{0}^{R_{2}}=\left( m_{I}\right)
^{\dag }$ while%
\begin{equation*}
\phi _{2}^{R_{2}}\left( X^{\op },Y^{\op }\right) =\varphi
_{2}^{R_{2}}\left( X,Y\right) =\left[ \left( \mathrm{ev}_{X}\otimes \mathrm{%
ev}_{Y}\right) \circ \left( X^{\ast }\otimes \left( c_{X,Y^{\ast }}\right)
^{-1}\otimes Y\right) \right] ^{\dag }
\end{equation*}%
i.e. $\varphi _{2}^{R_{2}}\left( X,Y\right) $ is uniquely determined by the
equality%
\begin{equation*}
\mathrm{ev}_{X\otimes Y}\circ \left( \varphi _{2}^{R_{2}}\left( X,Y\right)
\otimes X\otimes Y\right) =\left( \mathrm{ev}_{X}\otimes \mathrm{ev}%
_{Y}\right) \circ \left( X^{\ast }\otimes \left( c_{X,Y^{\ast }}\right)
^{-1}\otimes Y\right) .
\end{equation*}%
We have to check that $\phi _{2}^{RR_{1}L^{\op }}=\phi _{2}^{R_{2}}$
i.e. that $R\phi _{2}^{R_{1}}\left( L^{\op }X^{\op },L^{%
\op }Y^{\op }\right) =\phi _{2}^{R_{2}}\left( X^{\op %
},Y^{\op }\right) $ i.e. that $R\varphi _{2}^{R_{1}}\left(
LX,LY\right) =\varphi _{2}^{R_{2}}\left( X,Y\right) .$ Thus we have to check
that
\begin{equation*}
\mathrm{ev}_{X\otimes Y}\circ \left( R\varphi _{2}^{R_{1}}\left(
LX,LY\right) \otimes X\otimes Y\right) =\left( \mathrm{ev}_{X}\otimes
\mathrm{ev}_{Y}\right) \circ \left( X^{\ast }\otimes \left( c_{X,Y^{\ast
}}\right) ^{-1}\otimes Y\right) .
\end{equation*}%
By construction, for $X\in \mathcal{N}$ we have $X^{\ast }:=R\left( \left(
LX\right) ^{\ast }\right) $ and $\mathrm{ev}_{X}$ is uniquely determined
(easly checked) by%
\begin{equation*}
X^{\ast }\otimes X\overset{\eta _{X^{\ast }\otimes X}^{R}}{\longrightarrow }%
RL\left( X^{\ast }\otimes X\right) \overset{R\psi _{2}^{L}\left( X^{\ast
},X\right) }{\longrightarrow }R\left( L\left( X^{\ast }\right) \otimes
LX\right) =R\left( LR\left( \left( LX\right) ^{\ast }\right) \otimes
LX\right) \overset{R\left( \epsilon _{\left( LX\right) ^{\ast }}^{R}\otimes
LX\right) }{\longrightarrow }R\left( \left( LX\right) ^{\ast }\otimes
LX\right) \overset{R\left( \mathrm{ev}_{LX}\right) }{\longrightarrow }RI_{%
\mathcal{M}}\overset{R\left( \psi _{0}^{L}\right) ^{-1}}{\longrightarrow }%
RLI_{\mathcal{N}}\overset{\cong }{\longrightarrow }I_{\mathcal{N}}.
\end{equation*}%
Since $L$ and $R$ are strict monoidal this reduces to
\begin{equation*}
X^{\ast }\otimes X\overset{\eta _{X^{\ast }\otimes X}^{R}}{\longrightarrow }%
RL\left( X^{\ast }\otimes X\right) =R\left( L\left( X^{\ast }\right) \otimes
LX\right) =R\left( LR\left( \left( LX\right) ^{\ast }\right) \otimes
LX\right) \overset{R\left( \epsilon _{\left( LX\right) ^{\ast }}^{R}\otimes
LX\right) }{\longrightarrow }R\left( \left( LX\right) ^{\ast }\otimes
LX\right) \overset{R\left( \mathrm{ev}_{LX}\right) }{\longrightarrow }RI_{%
\mathcal{M}}=I_{\mathcal{N}}
\end{equation*}%
so that $\mathrm{ev}_{X}=R\left( \mathrm{ev}_{LX}\right) \circ R\left(
\epsilon _{\left( LX\right) ^{\ast }}^{R}\otimes LX\right) \circ \eta
_{X^{\ast }\otimes X}^{R}$ for very $X$ in $\mathcal{N}$. Via the adjunction
this equality becomes $\epsilon _{I_{\mathcal{M}}}^{R}\circ L\left( \mathrm{%
ev}_{X}\right) =\mathrm{ev}_{LX}\circ \left( \epsilon _{\left( LX\right)
^{\ast }}^{R}\otimes LX\right) .$

Note that, $\left( L,R,\eta ,\epsilon \right) $ is a monoidal natural
transformation, the two frunctors $L$ and $R$ are lax monoidal and $\epsilon
:LR\rightarrow \mathrm{Id}$ and $\eta :\mathrm{Id}\rightarrow RL$ are
monoidal natural transformations. Since $\epsilon :LR\rightarrow \mathrm{Id}$
and $\eta :\mathrm{Id}\rightarrow RL$ are a monoidal natural transformation
between strict monoidal functors, we have $\epsilon _{A\otimes
B}^{R}=\epsilon _{A}^{R}\otimes \epsilon _{B}^{R}$, $\epsilon _{I_{\mathcal{M%
}}}^{R}=I_{\mathcal{M}}$, $\eta _{A\otimes B}^{R}=\eta _{A}^{R}\otimes \eta
_{B}^{R}$ and $\eta _{I_{\mathcal{N}}}^{R}=I_{\mathcal{N}}.$ Since $L$ is
braided monoidal, we also have $L\left( c_{A,B}\right) =c_{LA,LB}.$

Thus we compute%
\begin{eqnarray*}
&&\left( \mathrm{ev}_{X}\otimes \mathrm{ev}_{Y}\right) \circ \left( X^{\ast
}\otimes \left( c_{X,Y^{\ast }}\right) ^{-1}\otimes Y\right)  \\
&=&\left( R\left( \mathrm{ev}_{LX}\right) \otimes R\left( \mathrm{ev}%
_{LY}\right) \right) \circ \left( R\left( \epsilon _{\left( LX\right) ^{\ast
}}^{R}\otimes LX\right) \otimes R\left( \epsilon _{\left( LY\right) ^{\ast
}}^{R}\otimes LY\right) \right) \circ \left( \eta _{X^{\ast }\otimes
X}^{R}\otimes \eta _{Y^{\ast }\otimes Y}^{R}\right) \circ \left( X^{\ast
}\otimes \left( c_{X,Y^{\ast }}\right) ^{-1}\otimes Y\right)  \\
&=&R\left[ \left( \mathrm{ev}_{LX}\otimes \mathrm{ev}_{LY}\right) \circ
\left( \epsilon _{\left( LX\right) ^{\ast }}^{R}\otimes LX\otimes \epsilon
_{\left( LY\right) ^{\ast }}^{R}\otimes LY\right) \right] \circ \left( \eta
_{X^{\ast }\otimes X}^{R}\otimes \eta _{Y^{\ast }\otimes Y}^{R}\right) \circ
\left( X^{\ast }\otimes \left( c_{X,Y^{\ast }}\right) ^{-1}\otimes Y\right)
\\
\text{if }\eta _{A\otimes B}^{R} &=&\eta _{A}^{R}\otimes \eta _{B}^{R} \\
&=&R\left[ \left( \mathrm{ev}_{LX}\otimes \mathrm{ev}_{LY}\right) \circ
\left( \epsilon _{\left( LX\right) ^{\ast }}^{R}\otimes LX\otimes \epsilon
_{\left( LY\right) ^{\ast }}^{R}\otimes LY\right) \right] \circ \eta
_{X^{\ast }\otimes X\otimes Y^{\ast }\otimes Y}^{R}\circ \left( X^{\ast
}\otimes \left( c_{X,Y^{\ast }}\right) ^{-1}\otimes Y\right)  \\
&=&R\left[ \left( \mathrm{ev}_{LX}\otimes \mathrm{ev}_{LY}\right) \circ
\left( \epsilon _{\left( LX\right) ^{\ast }}^{R}\otimes LX\otimes \epsilon
_{\left( LY\right) ^{\ast }}^{R}\otimes LY\right) \right] \circ RL\left(
X^{\ast }\otimes \left( c_{X,Y^{\ast }}\right) ^{-1}\otimes Y\right) \circ
\eta _{X^{\ast }\otimes Y^{\ast }\otimes X\otimes Y}^{R} \\
&=&R\left[ \left( \mathrm{ev}_{LX}\otimes \mathrm{ev}_{LY}\right) \circ
\left( \epsilon _{\left( LX\right) ^{\ast }}^{R}\otimes LX\otimes \epsilon
_{\left( LY\right) ^{\ast }}^{R}\otimes LY\right) \circ \left( L\left(
X^{\ast }\right) \otimes L\left( c_{X,Y^{\ast }}\right) ^{-1}\otimes
LY\right) \right] \circ \eta _{X^{\ast }\otimes Y^{\ast }\otimes X\otimes
Y}^{R} \\
\text{if }L\left( c_{A,B}\right)  &=&c_{LA,LB} \\
&=&R\left[ \left( \mathrm{ev}_{LX}\otimes \mathrm{ev}_{LY}\right) \circ
\left( \epsilon _{\left( LX\right) ^{\ast }}^{R}\otimes LX\otimes \epsilon
_{\left( LY\right) ^{\ast }}^{R}\otimes LY\right) \circ \left( L\left(
X^{\ast }\right) \otimes \left( c_{LX,L\left( Y^{\ast }\right) }\right)
^{-1}\otimes LY\right) \right] \circ \eta _{X^{\ast }\otimes Y^{\ast
}\otimes X\otimes Y}^{R} \\
&=&R\left[ \left( \mathrm{ev}_{LX}\otimes \mathrm{ev}_{LY}\right) \circ
\left( \left( LX\right) ^{\ast }\otimes \left( c_{LX,\left( LY\right) ^{\ast
}}\right) ^{-1}\otimes LY\right) \circ \left( \epsilon _{\left( LX\right)
^{\ast }}^{R}\otimes \epsilon _{\left( LY\right) ^{\ast }}^{R}\otimes
LX\otimes LY\right) \right] \circ \eta _{X^{\ast }\otimes Y^{\ast }\otimes
X\otimes Y}^{R} \\
\text{if }\epsilon _{A\otimes B}^{R} &=&\epsilon _{A}^{R}\otimes \epsilon
_{B}^{R} \\
&=&R\left[ \left( \mathrm{ev}_{LX}\otimes \mathrm{ev}_{LY}\right) \circ
\left( \left( LX\right) ^{\ast }\otimes \left( c_{LX,\left( LY\right) ^{\ast
}}\right) ^{-1}\otimes LY\right) \circ \left( \epsilon _{\left( LX\right)
^{\ast }\otimes \left( LY\right) ^{\ast }}^{R}\otimes LX\otimes LY\right) %
\right] \circ \eta _{X^{\ast }\otimes Y^{\ast }\otimes X\otimes Y}^{R}
\end{eqnarray*}%
which equals%
\begin{eqnarray*}
&&\mathrm{ev}_{X\otimes Y}\circ \left( R\varphi _{2}^{R_{1}}\left(
LX,LY\right) \otimes X\otimes Y\right)  \\
&=&R\left( \mathrm{ev}_{L\left( X\otimes Y\right) }\right) \circ R\left(
\epsilon _{\left( L\left( X\otimes Y\right) \right) ^{\ast }}^{R}\otimes
L\left( X\otimes Y\right) \right) \circ \eta _{\left( X\otimes Y\right)
^{\ast }\otimes X\otimes Y}^{R}\circ \left( R\varphi _{2}^{R_{1}}\left(
LX,LY\right) \otimes X\otimes Y\right)  \\
&=&R\left( \mathrm{ev}_{LX\otimes LY}\right) \circ R\left( \epsilon _{\left(
LX\otimes LY\right) ^{\ast }}^{R}\otimes LX\otimes LY\right) \circ RL\left(
R\varphi _{2}^{R_{1}}\left( LX,LY\right) \otimes X\otimes Y\right) \circ
\eta _{X^{\ast }\otimes Y^{\ast }\otimes X\otimes Y}^{R} \\
&=&R\left[ \left( \mathrm{ev}_{LX\otimes LY}\right) \circ \left( \epsilon
_{\left( LX\otimes LY\right) ^{\ast }}^{R}\otimes LX\otimes LY\right) \circ
\left( LR\varphi _{2}^{R_{1}}\left( LX,LY\right) \otimes LX\otimes LY\right) %
\right] \circ \eta _{X^{\ast }\otimes Y^{\ast }\otimes X\otimes Y}^{R} \\
&=&R\left[ \mathrm{ev}_{LX\otimes LY}\circ \left( \varphi _{2}^{R_{1}}\left(
LX,LY\right) \otimes LX\otimes LY\right) \circ \left( \epsilon _{\left(
LX\right) ^{\ast }\otimes \left( LY\right) ^{\ast }}^{R}\otimes LX\otimes
LY\right) \right] \circ \eta _{X^{\ast }\otimes Y^{\ast }\otimes X\otimes
Y}^{R} \\
&=&R\left[ \left( \mathrm{ev}_{LX}\otimes \mathrm{ev}_{LY}\right) \circ
\left( \left( LX\right) ^{\ast }\otimes \left( c_{LX,\left( LY\right) ^{\ast
}}\right) ^{-1}\otimes LY\right) \circ \left( \epsilon _{\left( LX\right)
^{\ast }\otimes \left( LY\right) ^{\ast }}^{R}\otimes LX\otimes LY\right) %
\right] \circ \eta _{X^{\ast }\otimes Y^{\ast }\otimes X\otimes Y}^{R}.
\end{eqnarray*}%
It remains to prove that $\phi _{0}^{RR_{1}L^{\op }}=\phi _{0}^{R_{2}}
$ i.e. that $R\phi _{0}^{R_{1}}=\left( m_{I_{\mathcal{N}}}\right) ^{\dag }$
i.e. that%
\begin{equation*}
\mathrm{ev}_{I_{\mathcal{N}}}\circ \left( R\phi _{0}^{R_{1}}\otimes I_{%
\mathcal{N}}\right) =m_{I_{\mathcal{N}}}.
\end{equation*}%
The fact that $\epsilon :LR\rightarrow \mathrm{Id}$ is a monoidal natural
transformation between strict monoidal functors, we have $\epsilon
_{A\otimes B}^{R}=\epsilon _{A}^{R}\otimes \epsilon _{B}^{R}$ and $\epsilon
_{I_{\mathcal{M}}}^{R}=I_{\mathcal{M}}$ so that
\begin{eqnarray*}
\mathrm{ev}_{I_{\mathcal{N}}}\circ \left( R\phi _{0}^{R_{1}}\otimes I_{%
\mathcal{N}}\right)  &=&R\left( \mathrm{ev}_{LI_{\mathcal{N}}}\right) \circ
R\left( \epsilon _{\left( LI_{\mathcal{N}}\right) ^{\ast }}^{R}\otimes LI_{%
\mathcal{N}}\right) \circ \eta _{\left( I_{\mathcal{N}}\right) ^{\ast
}\otimes I_{\mathcal{N}}}^{R}\circ \left( R\phi _{0}^{R_{1}}\otimes I_{%
\mathcal{N}}\right)  \\
&=&R\left( \mathrm{ev}_{LI_{\mathcal{N}}}\right) \circ R\left( \epsilon
_{\left( LI_{\mathcal{N}}\right) ^{\ast }}^{R}\otimes LI_{\mathcal{N}%
}\right) \circ RL\left( R\phi _{0}^{R_{1}}\otimes I_{\mathcal{N}}\right)
\circ \eta _{RI_{\mathcal{M}}\otimes I_{\mathcal{N}}}^{R} \\
&=&R\left[ \mathrm{ev}_{LI_{\mathcal{N}}}\circ \left( \epsilon _{\left( LI_{%
\mathcal{N}}\right) ^{\ast }}^{R}\otimes LI_{\mathcal{N}}\right) \circ
\left( LR\phi _{0}^{R_{1}}\otimes LI_{\mathcal{N}}\right) \right] \circ \eta
_{RI_{\mathcal{M}}\otimes I_{\mathcal{N}}}^{R} \\
&=&R\left[ \mathrm{ev}_{I_{\mathcal{M}}}\circ \left( \epsilon _{I_{\mathcal{M%
}}^{\ast }}^{R}\otimes I_{\mathcal{M}}\right) \circ \left( LR\phi
_{0}^{R_{1}}\otimes I_{\mathcal{M}}\right) \right] \circ \eta _{RI_{\mathcal{%
M}}\otimes I_{\mathcal{N}}}^{R} \\
&=&R\left[ \mathrm{ev}_{I_{\mathcal{M}}}\circ \left( \phi
_{0}^{R_{1}}\otimes I_{\mathcal{M}}\right) \circ \left( \epsilon _{I_{%
\mathcal{M}}}^{R}\otimes I_{\mathcal{M}}\right) \right] \circ \eta _{RI_{%
\mathcal{M}}\otimes I_{\mathcal{N}}}^{R} \\
&=&R\left[ m_{I_{\mathcal{M}}}\circ \left( \epsilon _{I_{\mathcal{M}%
}}^{R}\otimes I_{\mathcal{M}}\right) \right] \circ \eta _{RI_{\mathcal{M}%
}\otimes I_{\mathcal{N}}}^{R} \\
\text{since }\epsilon _{I_{\mathcal{M}}}^{R} &=&I_{\mathcal{M}} \\
&=&R\left[ m_{I_{\mathcal{M}}}\circ \left( \epsilon _{I_{\mathcal{M}%
}}^{R}\otimes \epsilon _{I_{\mathcal{M}}}^{R}\right) \right] \circ \eta
_{RI_{\mathcal{M}}\otimes I_{\mathcal{N}}}^{R} \\
\text{since }\epsilon _{A\otimes B}^{R} &=&\epsilon _{A}^{R}\otimes \epsilon
_{B}^{R} \\
&=&R\left[ m_{I_{\mathcal{M}}}\circ \epsilon _{I_{\mathcal{M}}\otimes I_{%
\mathcal{M}}}^{R}\right] \circ \eta _{RI_{\mathcal{M}}\otimes RI_{\mathcal{M}%
}}^{R} \\
&=&Rm_{I_{\mathcal{M}}}\circ R\epsilon _{I_{\mathcal{M}}\otimes I_{\mathcal{M%
}}}^{R}\circ \eta _{R\left( I_{\mathcal{M}}\otimes I_{\mathcal{M}}\right)
}^{R} \\
&=&Rm_{I_{\mathcal{M}}}\overset{(\ast )}{=}m_{RI_{\mathcal{M}}}=m_{I_{%
\mathcal{N}}}.
\end{eqnarray*}%
It remains to check $\left( \ast \right) $ but%
\begin{equation*}
Rm_{I_{\mathcal{M}}}=Rr_{I_{\mathcal{M}}}=Rr_{I_{\mathcal{M}}}\circ \phi
_{2}^{R}\left( I_{\mathcal{M}},I_{\mathcal{M}}\right) \circ \left( RI_{%
\mathcal{M}}\otimes \phi _{0}^{R}\right) \overset{R\text{ monoidal}}{=}%
r_{RI_{\mathcal{M}}}=m_{RI_{\mathcal{M}}}.
\end{equation*}
\end{invisible}
\end{proof}

Under mild assumptions, by means of Proposition \ref{pro:monadj}, we are now able to transfer the
main condition of Corollary \ref{coro:Isar} from a braided monoidal category $\Mm$ to the category $\Mm^{\NN}$. This will be applied to provide an explicit example of a pre-rigid braided monoidal category which is not right closed and where liftability is available.

\begin{proposition}
\label{coro:externlift}Let $\Mm$ be a braided monoidal category where
$\Mm$ is abelian and the tensor product is additive and exact in
each argument. Assume that $\Mm$ is pre-rigid and that the
assumption in Corollary \ref{coro:Isar} holds for $\Mm$. Then the
analogous conclusion holds for the category $\Mm^{\NN}$ of
externally $\NN$-graded $\Mm$-objects, too.
\end{proposition}

\begin{proof}  By Proposition \ref{pro:funcat}, we know that $\Mm^{\NN}$ is a
pre-rigid monoidal category with pre-dual of $X=\left( X_{n}\right) _{n\in \NN}$ given
by $X^{\ast }=\left( \delta _{n,0}X_{0}^{\ast }\right) _{n\in \NN}$. We also know that the hypotheses that $\Mm$ is abelian and the tensor product is additive and exact in
each argument guarantee that the category $\Mm^{\NN}$ is moreover braided, with braiding $c_{X,Y}$ defined by  $(c_{X,Y})_n=\oplus_{i=0}^nc_{X_i,Y_{n-i}}$ for all $X,Y$ objects in $\Mm^{\NN}$, see e.g. \cite[Definition 2.1]{Schauenburg}. We want to apply Proposition \ref{pro:monadj} in case $\Nn:=\Mm^{\NN}$. To this aim, note that the functors
\begin{eqnarray*}
L:\Mm^{\NN}\rightarrow \Mm, &&\qquad X=\left(
X_{n}\right) _{n\in \NN}\mapsto X_{0},\qquad f=\left( f_{n}\right)
_{n\in \NN}\mapsto f_{0}, \\
R:\Mm\rightarrow \Mm^{\NN}, &&\qquad V\mapsto \left(
\delta _{n,0}V\right) _{n\in \NN},\qquad f\mapsto \left( \delta
_{n,0}f\right) _{n\in \NN}.
\end{eqnarray*}%
considered in the proof of Proposition \ref{pro:funcatG}, where the counit is $\epsilon :=\id$ and the unit is defined on $X=(X_n)_{n\in\NN}$ by $\eta _{X}:=\left( \delta _{n,0}\mathrm{Id}_{X_{0}}\right) _{n\in
\NN}:X\rightarrow RLX=\left( \delta _{n,0}X_{0}\right) _{n\in \NN}$, are both strict
monoidal as
\begin{align*}
L\left( X\otimes Y\right)  &=L\left((\oplus _{i=0}^{n}X_{i}\otimes
Y_{n-i})_{n\in\NN}\right)  =\oplus _{i=0}^{0}X_{i}\otimes
Y_{0-i}=X_{0}\otimes Y_{0}=LX\otimes LY, \\
L\left( \unit^{\NN }\right)  &=L\left((\delta_{n,0} \unit)_{n\in\NN}\right) =\delta _{0,0}\unit=\unit, \\
R\left( V\right) \otimes R\left( W\right)  &=\left( \delta _{n,0}V\right)
_{n\in \NN}\otimes \left( \delta _{n,0}W\right) _{n\in \NN%
}=\left( \oplus _{i=0}^{n}\delta _{i,0}V\otimes \delta _{n-i,0}W\right)
_{n\in \NN}=\left( V\otimes \delta _{n-0,0}W\right) _{n\in \NN%
}\\&=\left( \delta _{n,0}V\otimes W\right) _{n\in \NN}=R\left( V\otimes
W\right) , \\
R\left( \unit\right)  &=\left( \delta _{n,0}\unit\right) _{n\in
\NN}=\unit^{\NN }\text{.}
\end{align*}
Moreover, $\epsilon_{V\otimes W}=\id_{V\otimes W} =\epsilon_{V}\otimes \epsilon_{W}$, $\epsilon_\unit=\id_\unit$, $\eta_{X\otimes Y}=\left( \delta _{n,0}\mathrm{Id}_{(X\otimes Y)_{0}}\right) _{n\in
\NN}=\left( \delta _{n,0}\mathrm{Id}_{X_0\otimes Y_0}\right) _{n\in
\NN}=\eta_X\otimes\eta_Y$, $\eta_{\unit^{\NN }}=\left( \delta _{n,0}\mathrm{Id}_{\unit^{\NN }_{0}}\right) _{n\in
\NN}=\left( \delta _{n,0}\mathrm{Id}_{\unit}\right) _{n\in
\NN}=\id_{\unit^{\NN }}$ so that $\eta$ and $\epsilon$ are monoidal natural transformations and hence $(L,R,\eta,\epsilon)$ is a monoidal adjunction. Furthermore, $L(c_{X,Y})=(c_{X,Y})_0=c_{X_0,Y_0}=c_{LX,LY}$, so that $L$ is braided monoidal.
We conclude by Proposition \ref{pro:monadj}.
\\Note that the pre-dual otained via this proposition is $R((LX)^*)=R(X_0^*)=\left( \delta _{n,0}X_{0}^{\ast }\right) _{n\in \NN}$ i.e. the same object declared at the beginning of this proof.
\end{proof}

\begin{example}\label{example-prerigidnotclosed}
Set $\Mm=\Vec ^{\text{f}}$. As we have seen in Example \ref%
{examplerightclosed}, $\Mm$ is a pre-rigid braided monoidal category
which is not right closed. Note that $\Mm$ fulfills the requirements
of Corollary \ref{coro:externlift}. In fact the adjunction $\left(
L_{1},R_{1}\right) =\left( (-)^{\ast }:\Mm\rightarrow \Mm^{%
\op },(-)^{\ast }:\Mm^{\op }\rightarrow \Mm%
\right) $ in this case is a category equivalence as $V^{\ast \ast }\cong V$
for $V\in \Vec ^{\text{f}}$. As a consequence, by definition of their
monoidal structures, both $L_{1}$ and $R_{1}$ are strong monoidal and the
equivalence $\left( L_1,R_1\right) $ induces,  by \cite[Proposition 3.91]{Aguiar-Mahajan}, an adjunction $\left( \overline{%
L_{1}},\overline{R_{1}}\right) $, which is a category equivalence as well. In particular, $\overline{R_{1}}$ has a
left adjoint i.e. $\overline{L_{1}}$. By Corollary \ref{coro:externlift},
if we consider the adjunction $\left( L_{2},R_{2}\right) =\left( (-)^{\ast }:%
\Mm^{\NN}\rightarrow \left( \Mm^{\NN}\right) ^{%
\op },(-)^{\ast }:\left( \Mm^{\NN}\right) ^{\op %
}\rightarrow \Mm^{\NN}\right) ,$ we get that $\overline{R_{2}}
$ has a left adjoint, too. In this way we have obtained that $\Mm^{%
\NN}$ is an example of a pre-rigid braided monoidal category which is
not right closed and such that Corollary \ref{coro:Isar} applies, so that $%
\left( L_{2},R_{2}\right) $ is a liftable pair of adjoint functors.
\end{example}

The above Example \ref{example-prerigidnotclosed} shows how, in favorable cases,  the notion of pre-rigid category allows to construct liftable pair of adjoint functors when the right-closedness is not available. This was achieved by first proving that the pre-dual construction defines a self adjoint functor $R=(-)^*:\Cc^\op\to\Cc$ whose lax monoidal structure is the opposite of the canonical colax monoidal structure induced on its left adjoint $L=R^\op$ (Proposition \ref{prop:prerigid}), then by showing that such a functor gives rise to a liftable pair $(L,R)$ whenever the induced functor $\overline{R}=\Alg(R)$ has a left adjoint (Proposition \ref{lem:Barop}), and finally by transporting the desired liftability from a possibly closed braided monoidal category $\Nn$ to a desirably not closed braided monoidal category $\Nn$ when these are connected by a suitable monoidal adjunction (Proposition \ref{pro:monadj}).

\end{document}